\documentclass[a4paper, 11pt]{article}
\usepackage[utf8]{inputenc}

\usepackage{amsmath}
\usepackage[normalem]{ulem}
\usepackage[margin=1in]{geometry}

\usepackage{mathrsfs}
\usepackage{algorithm}
\usepackage{algorithmic}
\usepackage{amssymb}
\usepackage{amsmath}
\usepackage{amsthm}
\usepackage{amsfonts}
\usepackage{enumitem} 
\usepackage{dsfont}
\usepackage{booktabs}

\usepackage[colorlinks,unicode]{hyperref}
\usepackage[nameinlink,capitalise,noabbrev]{cleveref}

\DeclareMathOperator*{\argmin}{arg\,min}

\DeclareMathOperator{\Sch}{Sch}

\newcommand{\R}{\mathbb R}
\newcommand{\N}{\mathbb N}
\newcommand{\DD}{\mathcal D}
\newcommand{\FF}{\mathcal F}
\newcommand{\cg}{\langle}
\newcommand{\cd}{\rangle}
\newcommand{\1}{\mathds 1}
\newcommand{\pf}{{}_\#}
\newcommand{\eps}{\varepsilon}
\renewcommand{\P}{\mathcal P}
\newcommand{\M}{\mathsf M}

\makeatletter
\DeclareRobustCommand*{\bfseries}{
  \not@math@alphabet\bfseries\mathbf
  \fontseries\bfdefault\selectfont
  \boldmath
}
\makeatother

\theoremstyle{plain}
\newtheorem{theo}{Theorem}
\newtheorem{conj}[theo]{Assumption}
\newtheorem{prop}[theo]{Proposition}
\newtheorem{lem}[theo]{Lemma}
\newtheorem{cor}[theo]{Corollary}
\theoremstyle{definition}
\newtheorem{defn}[theo]{Definition}
\theoremstyle{remark}
\newtheorem{rem}[theo]{Remark}

\DeclareUnicodeCharacter{2212}{-}
\DeclareUnicodeCharacter{FB01}{-}
\DeclareUnicodeCharacter{FB03}{-}
\usepackage{graphicx}
\usepackage{float}
\graphicspath{{figures/}}
\usepackage[labelfont=bf]{caption}
\usepackage{subcaption}

\usepackage[toc,page]{appendix}

\usepackage{tikz}
\usetikzlibrary{positioning, decorations, arrows, shapes}
\tikzset{
    rndproc/.style={circle,draw, thick}
}
\tikzset{
    hidden/.style={rndproc, fill=white!20!gray}
}

\usepackage{hyperref}
\hypersetup{
    colorlinks=true,
    citecolor=red,
    linkcolor=blue,
    linktoc=all
}

\usepackage{titlesec}
\titleformat{\paragraph}
{\normalfont\normalsize\bfseries}{\theparagraph}{1em}{}
\titlespacing*{\paragraph}
{0pt}{3.25ex plus 1ex minus .2ex}{1.5ex plus .2ex}

\author{Aymeric Baradat, Elias Ventre}

\begin{document}

\begin{center}
\noindent \uppercase{\textbf{Convergence of the Sinkhorn algorithm when the Schr\"odinger problem has no solution}}
\end{center}
    \vspace{.5cm}

\noindent \textbf{Authors}: Aymeric Baradat$^{1}$, Elias Ventre$^{2}$.\\
1 - Université  Claude Bernard Lyon 1, CNRS UMR 5208, Institut Camille Jordan, Villeurbanne, France.\\
2 - University of British Columbia, Vancouver, BC, Canada.\\
\textbf{Corresponding author}: aymeric.baradat@cnrs.fr.

\begin{abstract}
The Sinkhorn algorithm is the most popular method for solving the entropy minimization problem called the Schr\"odinger problem: in the non-degenerate cases, the latter admits a unique solution towards which the algorithm converges linearly. Here, motivated by recent applications of the Schr\"odinger problem with respect to structured stochastic processes (such as increasing ones), we study the Sinkhorn algorithm in degenerate cases where it might happen that no solution exist at all. We show that in this case, the algorithm ultimately alternates between two limit points. Moreover, these limit points can be used to compute the solution of a relaxed version of the Schr\"odinger problem, which appears as the $\Gamma$-limit of a problem where the marginal constraints are replaced by asymptotically large marginal penalizations, exactly in the spirit of the so-called unbalanced optimal transport. Finally, our work focuses on the support of the solution of the relaxed problem, giving its typical shape and designing a procedure to compute it quickly. We showcase promising numerical applications related to a model used in cell biology.
\end{abstract}

\vspace{.5cm}
\par

\bigskip


%\tableofcontents

\newpage

\section{Introduction}

The Schr\"odinger problem has been introduced by Schr\"odinger himself in the 30's~\cite{Schrodinger1931,Schrodinger1932} in the context of statistical mechanics. It is one of these problems in mathematics for which there is periodically a resurgence of interest, as witnessed by the numerous works which it was the object of for almost 100 years, among which~\cite{fortet1940resolution,Csiszar1975,zambrini1986variational,follmer1988random,ruschendorf1993note}. The last of these resurgences, over the past twenty years, occurred because of its close links with optimal transport. On the one hand, the Schr\"odinger problem, that comes with a temperature parameter in its classical formulation, converges towards optimal transport when the temperature goes to zero~\cite{mikami2004monge,Leonard2012,leonard2013survey}. On the other hand, it is often much easier to compute the solutions of the Schr\"odinger problem than the ones of the optimal transport problem~\cite{Cuturi2013,peyre2019computational}, thanks to the so-called Sinkhorn algorithm~\cite{Sinkhorn1964}. This algorithm converges exponentially fast (\emph{i.e.} at a linear rate, following the usual terminology in the field), at least when it is applied to a reference matrix whose entries are all below bounded by a positive number. 

It is known in the theory of matrix scaling that when the reference matrix has nonnegative but possibly cancelling entries, the data in the Schr\"odinger problem may be chosen in such a way that the latter admits no solution. This is the so-called \emph{non-scalable} case. Also, when the data are located at the boundary of those for which there is a solution, the so-called \emph{approximately scalable} case, the Schr\"odinger problem has a solution but the convergence of the Sinkhorn algorithm is not linear anymore.

In this paper, we want to study the Sinkhorn algorithm in the degenerate case where the Schr\"odinger problem has no solution. Our main finding is that for such problem, the Sinkhorn algorithm leads to exactly two limit points, each of them being the solution of a Schr\"odinger problem with modified data, that we characterize themselves as solutions of auxiliary optimization problems. Also, we show that these limit points are related to a problem where the marginal constraints of the original problem are replaced by marginal penalizations. Moreover, the Schr\"odinger problem related to the modified data is seen to belong to the approximately scalable case in general. We therefore provide a new outlook on the question of the support of the solution in this case, allowing to design an approximate method for improving the Sinkhorn algorithm's convergence both in the approximately scalable and non-scalable cases.

For simplicity and because it fits with the context of our numerical explorations and needs, we decided to work in finite spaces, even though some of the results might be generalizable.

\paragraph*{The Schr\"odinger problem in finite spaces} Let $\DD = \{ x_1, \dots, x_N \}$ and $\FF = \{y_1, \dots, y_M \}$ be two nonempty finite spaces and $R \in \mathcal M_+(\DD \times \FF)$ be a nonnegative measure on $\DD \times \FF$. Of course, we can identify $R$ with a matrix $R = (R_{ij}) \in \R_+^{N \times M}$ by setting $R_{ij} := R(\{(x_i,y_j)\})$. Assuming that $R$ models the coupling between the initial and final positions of the particles of a large system, we interpret $R_{ij}$ as the sum of the masses of all the particles being in $x_i$ at the initial time, and in $y_j$ at the final time.

Let us choose $\mu \in \mathcal M_+(\DD)$ and $\nu \in \mathcal M_+(\FF)$. Once again, we see $\mu = (\mu_i)$ and $\nu = (\nu_j)$ as vectors of $\R_+^N$ and $\R_+^M$ respectively. 

We call $\Pi(\mu, \nu)$ the subset of $\mathcal M_+(\DD \times \FF)$ consisting of all those matrices $\bar R$ whose row and column sums give $\mu$ and $\nu$ respectively, that is, such that

\begin{equation*}
\forall i = 1, \dots, N, \quad \sum_j \bar R_{ij} = \mu_i, \qquad \mbox{and} \qquad \forall j = 1, \dots M, \quad \sum_i \bar R_{ij} = \nu_j.
\end{equation*}

In our interpretation, it means that for the system described by $\bar R$, the sum of the masses of all the particles being in $x_i \in \DD$ at the initial time is $\mu_i$, and the sum of the masses of all the particles being in $y_j \in \FF$ at the final time is $\nu_j$. In particular, for $\Pi(\mu,\nu)$ to be nonempty, $\mu$ and $\nu$ need to share their total mass.

\begin{rem}
\label{rem:marginals}
Let us point out to the readers acquainted with the notations used in the Optimal Transport literature that calling $X : \DD \times \FF \to \DD$ and $Y: \DD \times \FF \to \FF$ the canonical projections and denoting by $\pf$ the push forward operation on measures, the measure $\bar R$ belongs to $\Pi(\mu,\nu)$ provided $X \pf \bar R = \mu$ and $Y \pf \bar R = \nu$. Actually, we will not use these notations, and prefer to define $\mu^{\bar R} := X \pf \bar R$ and $\nu^{\bar R} := Y \pf \bar R$, see formula~\ref{eq:notation_marginals}.
\end{rem}

We call the Schr\"odinger problem w.r.t.\ $R$ between $\mu$ and $\nu$ the convex optimization problem consisting in minimizing among $\Pi(\mu, \nu)$ the relative entropy w.r.t\ $R$:
\begin{equation*}
\Sch(R; \mu, \nu) := \min \Big\{ H(\bar R|R)  \, \Big | \, \bar R \in \Pi(\mu, \nu) \Big\},
\end{equation*}
where for all $\bar R \in \mathcal M_+(\DD \times \FF)$, the relative entropy of $P$ w.r.t.\ $R$ is defined by
\begin{equation*} 
H(\bar R|R) := \sum_{ij} \Big\{ \bar R_{ij} \log \frac{\bar R_{ij}}{R_{ij}} +  R_{ij} - \bar R_{ij}\Big\},
\end{equation*}
taking the conventions $a \log \frac{a}{0} = +\infty$ if $a>0$, and $0 \log 0 = 0 \log \frac{0}{0} = 0$. Notice that if $H(\bar R|R) < +\infty$, then for all $i,j$ such that $R_{ij} = 0$, we also have $\bar R_{ij} = 0$, \emph{i.e.}, $\bar R \ll R$ in the sense of measures.

\begin{rem}
\label{rem:coupling_total_mass}
 Of course, as before, for a solution $R^*$ to exist, $\mu$ and $\nu$ need to have the same total mass, and $R^*$ will then have the same total mass as $\mu$ and $\nu$. 
\end{rem}

By strict convexity of the relative entropy as a function of $\bar R$, when there is a solution, the latter is unique. Also, the relative entropy being lower semicontinuous w.r.t.\ $\bar R$\ and $\Pi(\mu,\nu)$ being compact, the existence of a solution $R^*$ for $\Sch(R; \mu,\nu)$ is equivalent to the existence of a $\bar R \in \Pi(\mu,\nu)$ satisfying $H(\bar R|R) < + \infty$. In what follows, such an $\bar R$ is called a competitor for $\Sch(R; \mu,\nu)$.

\bigskip

Heuristically, we seek for the measure $ R^*$ that is the closest possible to $R$ in the entropic sense while imposing its first and second marginals. 

In virtue of the Sanov theorem~\cite{sanov1958probability}, this problem has an interpretation in terms of large deviations. It is also known to be connected to optimal transport problems, see~\cite{Leonard2012,leonard2013survey,Gentil2017,carlier2017convergence}: if for all $i,j$, $c_{ij}$ models the cost to transport a unit of mass from $x_i$ to $y_j$, and $R_{ij} \propto \exp(-c_{ij}/ \eps)$ for some small $\eps>0$, then the solution of $\Sch(R; \mu,\nu)$ is a good approximation of a solution of the optimal transport problem between $\mu$ and $\nu$, of cost $(c_{ij})$.

\paragraph*{The Sinkhorn algorithm}
When the solution of $\Sch(R; \mu,\nu)$ exists, it is well known for a very long time that this solution turns out to be the limit of the sequences $(P^n)_{n \in \N^*}$ and $(Q^n)_{n \in \N^*}$ appearing in the following so-called the Sinkhorn algorithm, also called IPFP for \emph{iterative proportional fitting procedure}~\cite{Sinkhorn1964,sinkhorn1967diagonal,idel2016review,Nutz2021}:
\begin{equation}
\label{eq:sinkhorn}
\left\{
\begin{gathered}
Q^0 := R,\\
\begin{aligned}
\forall n\geq 0,\quad P^{n+1} &:= \argmin\Big\{ H(P | Q^n),  &&X \pf P = \mu \Big\},\\
\forall n\geq 0,\quad Q^{n+1} &:= \argmin\Big\{ H(Q | P^{n+1}), &&Y \pf Q = \nu \Big\}.
\end{aligned}
\end{gathered}
\right.
\end{equation}
This formulation is implicit as it involves minimization problems. In fact, easy results concerning these problems, detailed in Corollary~\ref{cor:one_step_sinkhorn} below, give access to an explicit and easily computable version, which takes the following form, when expressed in terms of the so called \emph{dual variables} or \emph{potentials} $(a^n)_{n \in \N^*} \in (\R_+^N)^\N$ and $(b^n)_{n \in \N}\in (\R_+^M)^\N$:
\begin{equation}
\label{eq:computable Sinkhorn}
\left\{
\begin{gathered}
\forall j, \quad  b^0_j := 1,\\[10pt]
\begin{aligned}
&\forall n\geq 0, \quad \forall i,j,& a^{n+1}_i &:= \frac{\mu_i}{\displaystyle \sum_{j'} b^n_{j'} R_{ij'}},& P^{n+1}_{ij} &:= a^{n+1}_i b^n_j R_{ij},\\
&\forall n\geq 0,\quad \forall i,j, & b^{n+1}_j &:= \frac{\nu_j}{\displaystyle\sum_{i'} a^{n+1}_{i'} R_{i'j}},& Q_{ij}^{n+1} &:= a^{n+1}_i  b^{n+1}_j R_{ij}.
\end{aligned}
\end{gathered}
\right.
\end{equation}

A reason for the popularity of this algorithm is that in a lot of contexts, the sequences of potentials $(a^n)$ and $(b^n)$, and hence the sequence of couplings $(P^n)$ and $(Q^n)$ converge at a linear rate, and the limit of $(P^n)$ and $(Q^n)$ coincide with the unique solution of $\Sch(R; \mu,\nu)$. For this reason, the Sinkhorn algorithm is nowadays the most efficient way to compute approximate solutions of optimal transport problems~\cite{Cuturi2013,benamou2015iterative,peyre2019computational}.

Observe that \emph{a priori}, the existence of a solution for the Schr\"odinger problem is not necessary to give a meaning to the Sinkhorn algorithm. Actually, we will see that there are lots of situations where the Schr\"odinger problem has no solution, and yet the Sinkhorn algorithm is perfectly well defined. These are the cases that we want to study in this text.

\paragraph*{A degenerate case}

As we just said, our aim is to study the Sinkhorn algorithm in the cases where the existence of a solution of the Schr\"odinger problem is either false, or at least nontrivial. This includes the case where $\mu$ and $\nu$ do not have the same total mass, see Remark~\ref{rem:coupling_total_mass}. However, this is not the main new situation that we want to encompass, since the Sinkhorn algorithm behaves trivially under normalization. More interestingly, we will give a detailed study of the case where some entries of $R$ cancel, or in optimal transport terms, when the cost function takes the value $+\infty$. 

In that situation, it can be hard to exhibit a competitor, since the natural candidate that is usually chosen, namely, the product measure of $\mu$ and $\nu$, is not absolutely continuous w.r.t.\ $R$ in general. In fact, there are cases where it is easy to see that no competitor exists. We give in Appendix~\ref{app:no_solution} an explicit and simple example of such a case. To illustrate our findings, we also describe the behaviour of the Sinkhorn algorithm applied to this example.\\

Note that beyond the theoretical interest, there are practical motivations for studying cases where the problem has no solution. Indeed, the Schr\"odinger problem can be used as follows. Suppose that $\mu$ and $\nu$ are some observed densities of a random phenomenon at two different timepoints, obtained for instance by building the empirical distributions associated to some collected data. Suppose also that we have at our disposal a good model for this phenomenon, that is, a reference stochastic process chosen based on our knowledge of the system prior to the observations of $\mu$ and $\nu$. Let us call $R$ the coupling of this process between the two studied timepoints. If we believe enough in our model and in our data, but still the marginals of $R$ are not $\mu$ and $\nu$, then it is reasonable to try to improve our model by looking for the coupling that is the closest to $R$ (for instance in the entropic sense), but which is compatible with the data: this means solving $\Sch(R; \mu,\nu)$.

Now imagine that there is no lower-bound for the coupling $R$, which can be perfectly justified (think for instance of a nondecreasing process, like the size of some randomly growing phenomenon). Then, small measurement errors due to imprecision of the devices or even to too restricted samplings may result in the non-existence of any coupling with marginals $\mu$ and $\nu$ being absolutely continuous with respect to $R$: the Schr\"odinger problem would thus have no solution. In that case, we would like to be able to find a coupling which explains the best the data while being entropically close to $R$. Some methods are available for doing so, like for instance algorithms solving the so-called unbalanced problem~\cite{Chizat2018}, but at the cost of introducing a new parameter quantifying the balance between the proximity to the data and to the reference coupling, whose value will often be arbitrarily chosen. We show in this article that interestingly, the Sinkhorn algorithm allows to overcome this choice in the specific situation where the data are more trustworthy than the model.

\bigskip

In particular, we were motivated by an application of the Sinkhorn algorithm related to systems biology, and more specifically to the treatment of single-cell data. The quick progresses of acquisition methods for such data raises the hope of a better understanding of the cell-differentiation process, which would in turn pave the way for major medical breakthroughs. In the seminal papers~\cite{Schiebinger2019,Lavenant2021}, Schiebinger and his coauthors suggest to analyse the collected data through an approach based on optimal transport and more specifically on the Schr\"odinger problem.

In this field, the unknown is the law of the evolution of the quantity of mRNA molecules in the cells through time: this evolution cannot be followed, as our techniques of measurement destroy the cells. Hence, to study it between two timepoints, the approach consists in:
\begin{enumerate}[label=(\roman*)]
    \item choosing a reference theoretical model $R$, where for all $i,j$, $R_{ij}$ is the expected quantity of cells whose mRNA levels are given by the vector $x_i$ at the initial time, and by $y_j$ at the final time;
    \item measuring the mRNA levels of samples of cells at the initial and final times to get approximate distributions $\mu$ and $\nu$ of these levels among the population of cells under study;
    \item solving the Schr\"odinger problem $\Sch(R; \mu,\nu)$ to get a law $R^*$ that is close to our model $R$, but which explains the data.
\end{enumerate}

In the case of Schiebinger, $R$ is the coupling produced by a Brownian motion between two time points, and therefore admits a below bound. In a separated work~\cite{Ventre2022thesis}, the second author argues that a more realistic model would be obtained by replacing the Brownian motion by a \emph{piecewise deterministic Markov process} as described in~\cite{Herbach2017}. For such models, dynamical constraints involving mRNAs half-life times lead to a degenerate $R$ and the corresponding Schr\"odinger problem could thus have no solution, not because of a lack in the model, but because of inaccuracies in the measurements. Our results show that the Sinkhorn algorithm can still be used in this situation, without any pre-treatment of the data. We refer once again to Appendix~\ref{app:no_solution} for a further discussion on this topic.

\paragraph*{Contributions}

In this article, we work with a potentially degenerate $R$, and our main contributions are the following. 
\begin{itemize}
	\item We show that the two sequences $(P^n)_{n \in \N^*}$ and $(Q^n)_{n \in \N^*}$ defined in~\eqref{eq:sinkhorn} converge towards two possibly different matrices $P^*$ and $Q^*$, each of them being the solution of a Schr\"odinger problem with modified marginals. More precisely, the matrice $P^*$ is the solution of the problem $\Sch(R;\mu, \nu^*)$, where $\nu^*$ minimizes the relative entropy w.r.t.\ $\nu$ within the set of marginals $\bar \nu$ for which the Schr\"odinger problem $\Sch(R; \mu, \bar \nu)$ admits a solution, and a similar statement holds for $Q^*$. This result, stated at Theorem~\ref{thm:convergence_sinkhorn}, is the main result of Section~\ref{sec:convergence_sinkhorn}.
	\item We show in Section~\ref{sec:gamma_cv} that the Sinkhorn algorithm enables to compute the solution of a modified Schr\"odinger problem where the marginal \emph{constraints} are replaced by marginal \emph{penalizations}: as shown at Theorem~\ref{thm:link_P*_Q*_R*}, the limit of the solution of the problem
	\begin{equation*}
	 \min \Big\{ H(\bar R |R) + \lambda \big( H(\mu^{\bar R }| \mu) + H(\nu^{\bar R}| \nu) \big)  \, \Big | \, \bar R \in \mathcal M_+(\DD \times \FF) \Big\},
	 \end{equation*}
	 (where once again, $\mu^{\bar R}$ and $\nu^{\bar R}$ are the first and second marginal of $\bar R$, see Remark~\ref{rem:marginals}) converges towards the componentwise geometric mean of the two limits $P^*$ and $Q^*$ of the Sinkhorn algorithm as $\lambda \to + \infty$. 
	\item In Section~\ref{sec:existence_of_a_solution}, we recall a well known necessary and sufficient condition on $R$, $\mu$ and $\nu$ for $\Sch(R; \mu, \nu)$ to admit a solution. Using this condition, we develop at Proposition~\ref{prop:find_S} a procedure to find the (common) support $\mathcal S$ of $P^*$ and $Q^*$ without computing them. As explained in Subsection~\ref{subsec:procedure}, our motivation is that the convergence rate of the Sinkhorn algorithm is linear if and only if $\mathcal S$ coincides with the support of $R$. When it is not the case, as often when $\Sch(R; \mu,\nu)$ has no solution, we can therefore improve the speed of convergence by first computing $\mathcal S$, and then by applying the Sinkhorn algorithm to $\Sch(\1_{\mathcal S} R; \mu, \nu)$ instead of $\Sch(R; \mu,\nu)$, which does not change the limits $P^*$ and $Q^*$.
	\item Section~\ref{sec_numerical} is an application of the developments made at Section~\ref{sec:existence_of_a_solution}. We implement an approximate but fast algorithm, usable in practice, allowing to recover an estimate of the support $\mathcal S$. We then compare the Sinkhorn algorithm and the technique coming from~\cite{Chizat2018} with our method consisting in first computing $\mathcal S$ with our approximate algorithm and then applying the Sinkhorn algorithm to $\Sch(\1_{\mathcal S} R; \mu, \nu)$. We also detail the regimes in which our method is a significant improvement of the other techniques.
\end{itemize}

Some of the results of this paper can be generalized by replacing $\DD$ and $\FF$ by general Polish spaces without much effort. This is the reason why we will often write $H(P|R) < + \infty$ instead of $P \ll R$: these are equivalent in the finite case, but not in the continuous one. In the latter case, we often need the stronger entropic assumption. Even if we decided to stick to the finite case in order to stress the key arguments that make everything work in practice, we believe that the continuous case is also interesting, and we wish to study it in a further work.

\bigskip

Before coming up with our contributions, we recall a few facts about the relative entropy functional at Section~\ref{sec:preliminaries}.

\section{Notations, properties of the entropy and terminology}
\label{sec:preliminaries}
In this preliminary section, we introduce some notations, provide well known elementary results concerning the entropy, and recall the terminology usually used in the theory of matrix scaling.

\subsection{Notations} 
\label{subsec:notations}
Let us first give a few notations that will be used systematically in this work. Most of them were already given in the introduction. 
\begin{itemize}
\item Whenever $I$ is a finite set of labels and $\mathcal E = \{u_k, \ k \in I \}$ is a finite set indexed by $I$, we denote by $\mathcal M_+(\mathcal E)$ the set of nonnegative measures on $\mathcal E$. This set is identified with $\R_+^I$ through the the correspondence 
\begin{equation*}
\mathsf r \in \mathcal M_+(\mathcal E) \leftrightsquigarrow (\mathsf r_k := \mathsf r(\{ u_k \}))_{k \in I} \in \R^I.
\end{equation*}
For all $\mathsf r \in \mathcal M_+(\mathcal E)$, we denote by $\M(\mathsf r) := \sum_k \mathsf r_k$ its total mass. If $\M(\mathsf r) = 1$, we say that $\mathsf r$ is a probability measure on $\mathcal E$, and we write $\mathsf r \in \P(\mathcal E)$. The topology considered on $\mathcal M_+(\mathcal E)$ is nothing but the one of $\R^I$.

\item In the same way, we identify the set $\FF ( \mathcal E ; \R)$ of real functions $Z$ on $\mathcal E$ with $\R^I$ through the correspondence  
\begin{equation*}
Z \in \FF ( \mathcal E ; \R) \leftrightsquigarrow (Z_k := Z(u_k))_{k \in I} \in \R_+^I.
\end{equation*}
Depending on the context, we will either call such functions $Z$ \emph{test functions}, or  \emph{random variables}, thinking of $\mathcal E$ as a measurable set. The random variables that we will consider will actually often be slightly more general, and be allowed to take the value $- \infty$, in which case we will tell it explicitly.

\item Through our identifications, the duality between $\mathcal M_+(\mathcal E)$ and $\FF ( \mathcal E ; \R)$ is nothing but the usual scalar product on $\R^I$, and denoted for all $Z \in \FF(\mathcal E; \R)$ and $\mathsf r \in \mathcal M_+(\mathcal E)$ by
\begin{equation*}
\cg Z, \mathsf r \cd := \sum_k Z_k \mathsf r_k.
\end{equation*}
When $Z$ possibly takes the value $- \infty$, we always choose by convention $- \infty \times 0= 0$.

\item In the context of the introduction, when $\DD = \{ x_1, \dots, x_N \}$ and $\FF = \{y_1, \dots, y_M \}$ are two nonempty finite spaces and $\mathcal E = \DD \times \FF$, then the corresponding $I$ is the product space $\{ 1, \dots, N\}\times \{ 1,\dots, M \}$, and $\bar R \in \mathcal M_+(\DD \times \FF)$, is seen as a matrix. We define its marginals $\mu^{\bar R} \in \mathcal M_+(\DD)$ and $\nu^{\bar R} \in \mathcal M_+(\FF)$ by the formulas
\begin{equation}
\label{eq:notation_marginals}
\forall i = 1, \dots, N, \quad \mu^{\bar R}_i := \sum_j \bar R_{ij}, \quad \mbox{and} \quad \forall j = 1, \dots M, \quad  \nu^{\bar R}_j := \sum_i \bar R_{ij} .
\end{equation}
Of course, $\mu^{\bar R}$ and $\nu^{\bar R}$ have the same total mass as $\bar R$, that is:
\begin{equation}
\label{eq:mass_identities}
    \mathsf M(\bar R) = \mathsf M(\mu^{\bar R}) = \mathsf M(\nu^{\bar R}).
\end{equation}
In particular, if $R$ is a probability measure, its marginals are probability measures as well.

\item As before, if $\mu \in \mathcal M_+(\DD)$ and $\nu \in \mathcal M_+(\FF)$, we call $\Pi(\mu,\nu)$ the set of those $R \in \mathcal M_+(\DD \times \FF)$ such that $\mu^R = \mu$ and $\nu^R = \nu$.

\item For the sake of simplicity, we do not use different notations for the same functions applied in different context. For instance, notations for the total mass $\mathsf M$ or the relative entropy $H$ (see Definition~\ref{def:relative_entropy} below) might be applied to different sets $\mathcal{E}$ namely $\mathcal{D}$, $\mathcal{F}$ and
$\mathcal{D}\times \mathcal{F}$.
\end{itemize}

\subsection{First properties of the relative entropy}

This subsection only contains easy and very well known results concerning the relative entropy that will be useful in the sequel. We stick to the finite case as this is the one studied in this paper, and we provide some proofs for the readers who are not acquainted with this notion of entropy, but all the properties given here are known in a much wider context, see for instance~\cite{Leonard2014}. 

As already said in the introduction, the relative entropy is defined as follows.
\begin{defn}
\label{def:relative_entropy}
	Let $\mathcal E = \{u_k, \ k \in I\}$ be a finite set and $\mathsf{r} = (\mathsf{r}_k) \in \mathcal M_+(\mathcal E)$. For all $\bar{\mathsf{r}} = (\bar{\mathsf{r}}_k) \in \mathcal M_+(\mathcal E)$, the relative entropy of $\bar{\mathsf{r}}$ w.r.t $\mathsf r$ is the value in $[0,+\infty]$ given by
	\begin{equation*}
	H(\bar{\mathsf r}|\mathsf r) := \sum_{k}\Big\{ \bar{\mathsf{r}}_{k} \log \frac{\bar{\mathsf{r}}_{k}}{\mathsf{r}_{k}} + \mathsf{r}_{k}  - \bar{\mathsf{r}}_{k} \Big\}= \sum_{k} \bar{\mathsf{r}}_{k} \log \frac{\bar{\mathsf{r}}_{k}}{\mathsf{r}_{k}} + \M(\mathsf r) - \M(\bar{\mathsf r}),
	\end{equation*}
	with convention $a \log \frac{a}{0} = + \infty$ for all $a>0$, and $0 \log 0 = 0 \log \frac{0}{0} = 0$. 
\end{defn}

First, this definition provides a convex function with good continuity properties. We state them in the following proposition, for which we omit the straightforward proof.

\begin{prop}
\label{prop:continuity_H}
	Let $\mathcal E$ be a finite set and $\mathsf r \in \mathcal M_+(\mathcal E)$. The functional
	\begin{equation*}
	\bar{\mathsf r} \in \mathcal M_+(\mathcal E) \mapsto H(\bar{\mathsf r} | \mathsf r) \in [0, + \infty]
	\end{equation*}
	is strictly convex, lower semicontinuous, and continuous on its domain, which is the closed set $\{ \bar{\mathsf r} \ll \mathsf r\} \subset \mathcal M_+(\mathcal E)$.
	
	For a given $\bar{\mathsf r} \in \mathcal M_+(\mathcal E)$, the functional
	\begin{equation*}
	\mathsf r \in \mathcal M_+(\mathcal E) \mapsto H(\bar{\mathsf r} | \mathsf r) \in [0, + \infty]
	\end{equation*}
	is convex and continuous for the canonical topology of $[0,+\infty]$. Its domain is the open set $\{ \mathsf r \gg \bar{\mathsf r}\} \subset \mathcal M_+(\mathcal E)$.
\end{prop}

The most useful property of the relative entropy is the computation of its Legendre transform. This property can be stated as follows.

\begin{theo}
	Let $\mathcal E = \{ u_k, \ k \in I\}$ be a finite set, and $\mathsf r \in \mathcal M_+(\mathcal E)$. For all test function $Z$ possibly taking the value $- \infty$ on $\mathcal E$ and all nonnegative measure $\bar{\mathsf{r}}$ on $\mathcal E$, we have
\begin{equation}
\label{eq:legendre_transform_entropy}
\cg Z, \bar{\mathsf{r}} \cd \leq H(\bar{\mathsf{r}} | \mathsf r ) + \cg e^{Z} - 1, \mathsf r \cd,
\end{equation}
with conventions $e^{-\infty} = 0$, $ -\infty \times 0 = 0$ and $+\infty - \infty = + \infty$.

Moreover, equality in $\R$ holds if and only if $\bar{\mathsf r} \ll \mathsf r$ and for all $k \in I$,
\begin{equation}
\label{eq:condition_equality_legendre_transform}
 Z_k =\log \frac{\bar{\mathsf{r}}_k}{\mathsf r_k} \in [-\infty,+\infty)
\end{equation}
 with convention $\log \frac{0}{a} = - \infty$ for all $a \geq 0$.
\end{theo}
\begin{proof}
	Let $\mathsf r, \bar{\mathsf{r}}$ and $Z$ be as in the statement of the theorem. If $H(\bar{\mathsf r} | \mathsf r) = + \infty$, there is nothing to prove, and we assume $\bar{\mathsf r} \ll \mathsf r$.
	
	By direct real computations, with the same conventions as in the statement of the theorem, we find that for all $k \in I$:
	\begin{equation*}
	Z_k \bar{\mathsf r}_k \leq \bar{\mathsf r}_k \log \frac{\bar{\mathsf r}_k}{\mathsf r_k} + \mathsf r_k - \bar{\mathsf r}_k +  \big( e^{Z_k} - 1 \big)\mathsf r_k,
	\end{equation*}
	with equality if and only if $\mathsf r_k = \bar{\mathsf r}_k = 0$ or $\mathsf r_k>0$ and
	\begin{equation*}
	Z_k = \log \frac{\bar{\mathsf{r}}_k}{\mathsf r_k} \in [-\infty,+\infty).
	\end{equation*}
	We find~\eqref{eq:legendre_transform_entropy} and~\eqref{eq:condition_equality_legendre_transform} by summing this inequality over $k$.
\end{proof}

This theorem will be useful as such, but also implies the following corollary which gives a full understanding of one step in the Sinkhorn algorithm~\eqref{eq:sinkhorn}.

\begin{cor}
\label{cor:one_step_sinkhorn}
	Let $\DD$ and $\FF$ be two finite sets, and $\bar R,R \in \mathcal M_+(\DD \times \FF)$. With the notations of~\eqref{eq:notation_marginals}, we have
	\begin{equation}
	\label{eq:ineq_proj}
	H(\mu^{\bar R} | \mu^R) \leq H(\bar R | R) \qquad \mbox{and} \qquad H(\nu^{\bar R} | \nu^R) \leq H(\bar R | R).
	\end{equation} 
	In the case where $H(\bar R | R)$ is finite, equality holds if and only if for all $i,j$, respectively:
	\begin{equation*}
	\bar R_{ij} = \frac{\mu^{\bar R}_i}{\mu^R_i} R_{ij} \qquad \mbox{and} \qquad \bar R_{ij} = \frac{\nu^{\bar R}_j}{\nu^R_j} R_{ij}, 
	\end{equation*}
	with convention $\frac{0}{0} = 0$.
	
	In particular, given $R \in \mathcal M_+(\DD \times \FF)$ and $\mu \in \mathcal M_+(\DD)$, the problem
	\begin{equation}
	\label{eq:pb_first_marginal}
	\min \Big\{ H(P|R) \, \Big| \, \mu^P = \mu \Big\}
	\end{equation}
	admits a solution if and only if $H(\mu | \mu^R) < + \infty$, and in this case, this solution $P$ is unique and satisfies for all $i,j$
	\begin{equation}
	\label{eq:sol_pb_first_marginal}
	P_{ij} = \frac{\mu_i}{\mu^R_i} R_{ij}
	\end{equation}
	with convention $\frac{0}{0} = 0$. Moreover, $H(P|R) = H(\mu | \mu^R)$.
	
	Similarly, given $R \in \mathcal M_+(\DD \times \FF)$ and $\nu \in \mathcal M_+(\FF)$, the problem
	\begin{equation*}
	\min \Big\{ H(Q|R) \, \Big| \, \nu^Q = \nu \Big\}
	\end{equation*}
	admits a solution if and only if $H(\nu | \nu^R) < + \infty$, and in this case, this solution $Q$ is unique and satisfies for all $i,j$
	\begin{equation*}
	Q_{ij} = \frac{\nu_j}{\nu^R_j} R_{ij}
	\end{equation*}
	with convention $\frac{0}{0} = 0$. Moreover, $H(Q|R) = H(\nu | \nu^R)$.
\end{cor}
\begin{proof}
	The first inequality in~\eqref{eq:ineq_proj} is a direct application of~\eqref{eq:legendre_transform_entropy} with $\mathsf r = R$, $\bar{\mathsf r} = \bar R$ and for all $i,j$, 
	\begin{equation*}
	Z_{ij} = \log \frac{\mu^{\bar R}_i}{\mu^R_i}.
	\end{equation*}
	The second inequality is proved in the same way, and the equality case is a consequence of~\eqref{eq:condition_equality_legendre_transform}.
	
	For the second part of the statement, let us observe that for all $P$ satisfying the constraint in~\eqref{eq:pb_first_marginal}, because of~\eqref{eq:ineq_proj}, $H(P|R) \geq H(\mu | \mu^R)$, which -- by the equality case -- is attained if and only if~\eqref{eq:sol_pb_first_marginal} holds. The problem involving the second marginal is treated in the same way.
\end{proof}

\subsection{The Schr\"odinger problem: assumptions and terminology}
\label{subsec:ass}

Let $\DD = \{ x_1, \dots, x_N \}$ and $\FF = \{y_1, \dots, y_M \}$ be two nonempty finite sets, and let us choose a reference measure $R \in \mathcal M_+(\DD \times \FF)$. Given $\mu \in \mathcal M_+(\DD)$ and $\nu \in \mathcal M_+(\FF)$, the Schr\"odinger problem, already defined in the introduction, rewrites with the notations of Subsection~\ref{subsec:notations}:
\begin{equation}
\label{eq:def_Sch}
\Sch(R; \mu, \nu) := \min \Big\{ H(\bar R|R)  \, \Big | \, \bar R \in \mathcal M_+(\mu, \nu) \mbox{ such that } \mu^{\bar R} = \mu \mbox{ and }\nu^{\bar R} = \nu \Big\}.
\end{equation}
\begin{rem}
Here, we define $\Sch(R; \mu,\nu)$ as the optimal value of our problem. However, with an abusive terminology, we will refer to the minimizer of the r.h.s.\ of~\eqref{eq:def_Sch} as "the solution of $\Sch(R; \mu,\nu)$". More generally, we will call "the problem $\Sch(R; \mu,\nu)$" the optimization problem consisting in computing the value $\Sch(R; \mu,\nu)$.
\end{rem}

As we will see in Theorem~\ref{thm:convergence_sinkhorn}, the Sinkhorn algorithm~\eqref{eq:sinkhorn} associated with the problem $\Sch(R; \mu,\nu)$ is well defined if and only if the following assumption holds.
\begin{conj}
\label{ass:condition_sinkhorn_well_defined}
Let $R \in \mathcal M_+(\DD \times \FF)$, $\mu \in \mathcal M_+(\DD)$ and $\nu \in \mathcal M_+(\FF)$, and let us call
\begin{equation}
\label{eq:def_E}
\mathcal{E} := \Big\{ (x_i,y_j) \in \DD \times \FF \mbox{ such that } R_{ij}>0, \, \mu_i>0 \mbox{ and }\nu_j>0 \Big\}.
\end{equation}
We say that the triple $(R;\mu,\nu)$ satisfies Assumption~\ref{ass:condition_sinkhorn_well_defined} provided $R^0 := \1_{\mathcal E} \cdot R$ is such that:
\begin{equation}
\label{eq:assumption_sinkhorn_well_defined}
H(\mu | \mu^{R^0}) < + \infty \qquad \mbox{and} \qquad H(\nu | \nu^{R^0}) < + \infty.
\end{equation}
\end{conj}

This assumption is easily seen to be necessary for $\Sch(R; \mu, \nu)$ to admit a solution. Under Assumption~\ref{ass:condition_sinkhorn_well_defined} either $\M(\mu) = \M(\nu) = 0$, or none of them is $0$. In the second case, up to replacing $\DD$ by $\DD'$, the support of $\mu$, $\FF$ by $\FF'$, the support of $\nu$, and $R$ by its restriction (or equivalently of the one of $R^0$) on $\DD' \times \FF'$, we end up with the following assumption, that will often be used in this paper.

\begin{conj}
\label{ass:full_support}
Let $R \in \mathcal M_+(\DD \times \FF)$, $\mu \in \mathcal M_+(\DD)$ and $\nu \in \mathcal M_+(\FF)$. We say that the triple $(R;\mu,\nu)$ satisfies Assumption~\ref{ass:full_support} provided the support of $\mu$ and $\mu^R$ is $\DD$ and the support of $\nu$ and $\nu^R$ is $\FF$.
\end{conj}

\bigskip

The Schr\"odinger problem~\eqref{eq:def_Sch} consists in minimizing a convex function under linear constraints. Therefore, the functional $(\mu,\nu) \in \mathcal M_+(\DD) \times \mathcal M_+(\FF) \mapsto \Sch(R; \mu,\nu) \in [0, + \infty]$ is convex. 

In the case where Assumption~\ref{ass:full_support} holds, following the usual terminology of the matrix scaling theory (except for the last item which is more exotic), see~\cite{idel2016review}, we say that:
\begin{itemize}
    \item The problem is \emph{scalable} if $(\mu,\nu)$ is in the relative interior of the domain of $\Sch(R; \cdot)$. In this case, $\M(\mu) = \M(\nu)$, the Schr\"odinger problem admits a unique solution $R^*$, $R^* \sim R$ in the sense of measures, and the Sinkhorn algorithm converges towards $R^*$, at a linear rate. In Lemma~\ref{lem:non_degenerate_case}, we recall an explicit necessary and sufficient condition on $R$, $\mu$, $\nu$ for $\Sch(R; \mu, \nu)$ to be scalable.
    \item The problem is \emph{approximately scalable} if $(\mu,\nu)$ is at the relative boundary of the domain of $\Sch(R; \cdot)$. In this case, $\M(\mu) = \M(\nu)$, the Schr\"odinger problem admits a unique solution $R^*$, and the Sinkhorn algorithm converges towards $R^*$. However, in this case, the support of $R^*$ is strictly included in the support of $R$ (else, we easily see that we are in the scalable case), and the rate cannot be linear anymore: as proved in~\cite{Achilles1993}, a linear rate of convergence for the Sinkhorn algorithm is not compatible with the appearance of new zero entries at the limit. We recall at Theorem~\ref{thm:CNS} a necessary and sufficient condition on $R$, $\mu$ and $\nu$ for $\Sch(R; \mu,\nu)$ to be at least approximately scalable, that is, either approximately scalable or scalable.
    \item The problem is \emph{non-scalable} if $\M(\mu) = \M(\nu)$, but the Schr\"odinger problem $\Sch(R; \mu,\nu)$ does not admit a solution. This is the case when the condition of Theorem~\ref{thm:CNS} does not hold. This case is the main case of interest in this work.
    \item The problem is \emph{unbalanced} if $\M(\mu) \neq \M(\nu)$. Calling $\mu' := \mu / \mu(\DD)$ and $\nu' := \nu/\nu(\FF)$ their normalized versions, we will say that $\Sch(R; \mu,\nu)$ is respectively unbalanced scalable, unbalanced approximately scalable and unbalanced non-scalable whenever $\Sch(R; \mu',\nu')$ is scalable, approximately scalable or non-scalable.
\end{itemize}
Yet, with an abuse of terminology, we will often refer to the non-scalable case for results that are true in \emph{any} situation, including the balanced and unbalanced non-scalable ones, which are often the most difficult.

\section{The Sinkhorn algorithm in the non-scalable case}
\label{sec:convergence_sinkhorn}

In this section, we consider $R \in \mathcal M_+(\DD \times \FF)$, $\mu \in \mathcal M_+(\DD)$ and $\nu \in \mathcal M_+(\FF)$ that we identify respectively with a matrix and two vectors, as before. 

The goal of this section is to show that under obvious necessary assumptions, then the algorithm given in~\eqref{eq:sinkhorn} is well defined, and that the sequences $(P^n)_{n \in \N^*}$ and $(Q^n)_{n \in \N^*}$ that it provides converge separately towards matrices $P^*$ and $Q^*$ that we define now. It will be obvious from their definition that these matrices coincide if and only if the problem $\Sch(R; \mu, \nu)$ defined in~\eqref{eq:def_Sch} admits a solution, that is, if it is at least approximately scalable. Hence our proof recovers the classical fact that the Sinkhorn algorithm converges towards the solution of the Schr\"odinger problem as soon as the latter exists.

 The first step to define $P^*$ and $Q^*$ is to define a pair of new marginals $\mu^* \in \mathcal M_+(\DD)$ and $\nu^* \in \mathcal M_+(\FF)$ as solutions of the following optimization problem:
 \begin{equation}
\label{eq:def_mu*_nu*}
\begin{gathered}
\mu^* := \argmin \Big\{ H(\bar \mu | \mu)\, \Big| \, \bar \mu = \mu^Q \mbox{ for some } Q  \mbox{ with } H(Q | R) < + \infty \mbox{ and } \nu^Q = \nu \Big\},\\
\nu^* := \argmin \Big\{ H(\bar \nu | \nu)\, \Big| \, \bar \nu = \nu^P \mbox{ for some } P  \mbox{ with } H(P|R) < + \infty \mbox{ and } \mu^P = \mu \Big\}.
\end{gathered}
 \end{equation}
The question of existence of $\mu^*$ and $\nu^*$ is treated in Theorem~\ref{thm:convergence_sinkhorn} below. Of course, if the problem $\Sch(R;\mu,\nu)$ admits a competitor, then $\mu^* = \mu$ and $\nu^* = \nu$.

\begin{rem}
\label{rem:mass_identities}
In the unbalanced case, notice that the total mass of $\nu^*$ is the one of $\mu$, and the total mass of $\mu^*$ is the one of $\nu$, that is, $\mathsf M(\nu^*) = \mathsf M(\mu)$ and $\mathsf M(\mu^*) = \mathsf M(\nu)$. 
\end{rem}

Then $P^*$ and $Q^*$ are simply defined as the solutions of the Schr\"odinger problems $\Sch(R;\mu, \nu^*)$ and $\Sch(R;\mu^*,\nu)$ respectively, that is:
\begin{equation}
\label{eq:def_P*_Q*}
P^* := \argmin \Big\{ H(P|R)  \, \Big | \, P \in \Pi(\mu, \nu^*) \Big\} \quad \mbox{and} \quad Q^* := \argmin \Big\{ H(Q|R)  \, \Big | \, Q \in \Pi(\mu^*, \nu) \Big\}.
\end{equation}
Of course, if the problem $\Sch(R;\mu,\nu)$ admits a competitor, and hence a solution, then both $P^*$ and $Q^*$ coincide with this solution.

Our convergence theorem can be stated as follows.
\begin{theo}
\label{thm:convergence_sinkhorn}
Let $R \in \mathcal M_+(\DD \times \FF)$, $\mu \in \mathcal M_+(\DD)$ and $\nu \in \mathcal M_+(\FF)$ satisfy Assumption~\ref{ass:condition_sinkhorn_well_defined}. The sequences $(P^n)_{n \in \N^*}$ and $(Q^n)_{n \in \N^*}$ from~\eqref{eq:sinkhorn}, the marginals $\mu^*$ and $\nu^*$ from~\eqref{eq:def_mu*_nu*} and the matrices $P^*$ and $Q^*$ from~\eqref{eq:def_P*_Q*} are well defined, and
\begin{equation*}
P^n \underset{n \to + \infty}{\longrightarrow} P^* \qquad \mbox{and} \qquad Q^n \underset{n \to + \infty}{\longrightarrow} Q^*.
\end{equation*}
\end{theo}
\begin{rem}
	\begin{itemize}
		\item Assumption~\ref{ass:condition_sinkhorn_well_defined} is necessary: it is straightforward to check that if $Q^1$ from~\eqref{eq:sinkhorn} is well defined, then $H(Q^1 | R^0) < + \infty$. In particular, projecting on the second marginal, we conclude that $H(\nu | \nu^{R^0}) < + \infty$. Arguing in the same way with $P^2$ in place of $Q^1$ and the second marginal in place of the first one, we see that if $P^2$ is well defined, then $H(\mu | \mu^{R^0}) < + \infty$. In particular, there is nothing to check before starting the algorithm: if the algorithm is able to compute $P^2$, then it means that our assumption is satisfied and that the convergence holds.
		\item Note that the topology for the convergence stated in the theorem does not matter since we are working in finite dimensional spaces. However, we believe that the result is still true replacing $\DD$ and $\FF$ by general Polish spaces. In this case, the convergence needs to be understood in the sense of the narrow topology, a topology for which the sequences $(P^n)$ and $(Q^n)$ can be proved to be compact due to the properties of their marginals.
		\item Remarkably, we will be able to prove this theorem without deriving the optimality conditions for $\mu^*$ and $\nu^*$. However, these optimality conditions will be needed in the next section, and hence written at Proposition~\ref{prop:optimality_conditions_mu*_nu*}.
		\item As developed in~\cite{Csiszar1975}, there exists a strong analogy between the relative entropy and square Euclidean distances, and this in spite of the lack of symmetry of the first. In particular, following this analogy, the Sinkhorn algorithm~\eqref{eq:sinkhorn} consists in iteratively othogonaly projecting on the convex sets of measures absolutely continuous w.r.t.\ $R$ satisfying the first and second marginal constraint respectively. 

With this picture in mind, we can give in~\cref{fig:sinkhorn} a visual representation of the scalable and non-scalable case. In the scalable case, the two convex sets intersect, and the sequences $(P^n)_{n \in \N^*}$ and $(Q^n)_{n \in \N^*}$ converge towards the point of the intersection that is the closest to~$R$. In the non-scalable case, the two convex sets do not intersect. However, the sequences $(P^n)_{n \in \N^*}$ and $(Q^n)_{n \in \N^*}$ still converge respectively to $P^*$ and $Q^*$, the two extreme points of the shortest line segment connecting both sets. Theorem~\ref{thm:convergence_sinkhorn} indeed justifies this type of behaviour for the Sinkhorn algorithm. 

One should still keep in mind that this analogy and our drawings are only sketchy. In reality, the projections are not orthogonal, and the convex sets have polygonal borders.

\begin{figure}
\includegraphics[width=.48\textwidth]{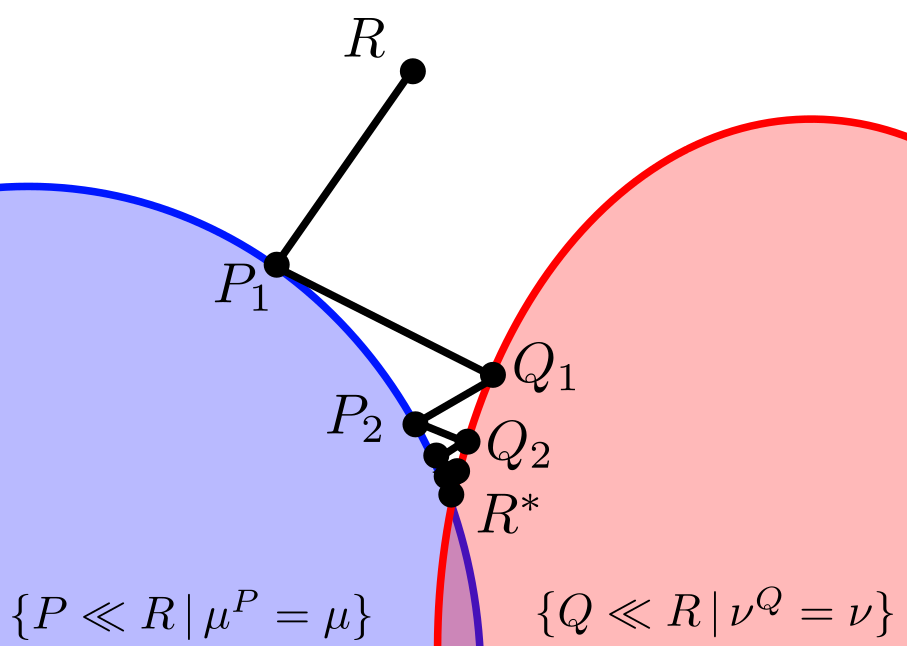}
\hfill
\includegraphics[width=.48\textwidth]{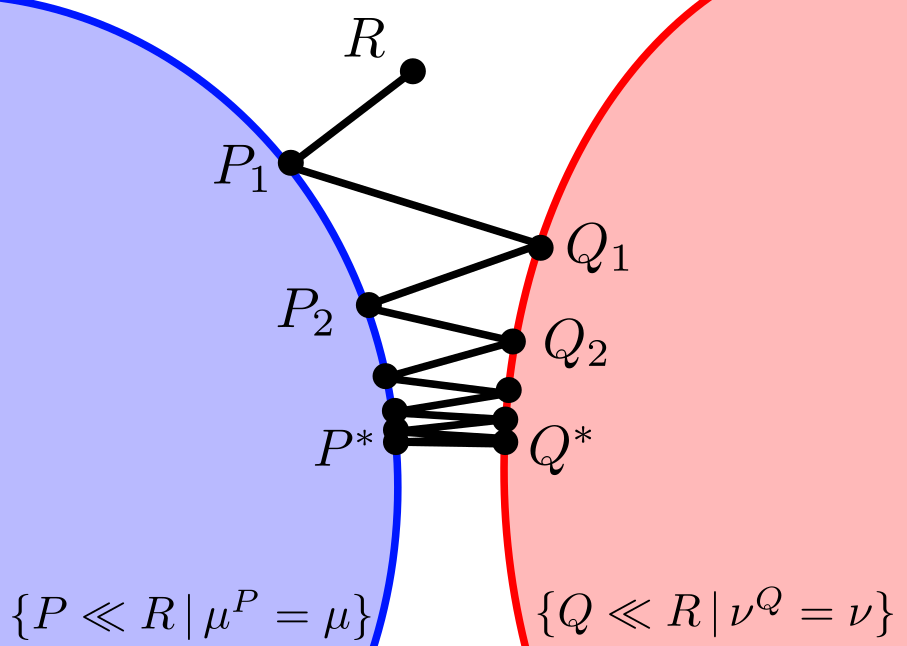}
\caption{\label{fig:sinkhorn} Sketchy representation of the Sinkhorn algorithm in the scalable case (to the left) and nonscalable case (to the right).}
\end{figure}
		\end{itemize}
\end{rem}
\begin{proof}
\underline{Step 1}: All the objects are well defined.

Let us first show that under the assumption of the theorem, the sequences $(P^n)_{n \in \N^*}$ and $(Q^n)_{n \in \N^*}$ are well defined. We start with $P^1$. As by assumption $\mu \ll \mu^{R^0} \ll \mu^R$, Corollary~\ref{cor:one_step_sinkhorn} shows that $P^1$ is well defined, and that for all $i,j$,
\begin{equation*}
P^1_{ij} = \frac{\mu_i}{\mu^R_i} R_{ij},
\end{equation*}
with convention $\frac{0}{0} = 0$. Clearly, $R^0 \ll P^1$, as the support of the latter is 
\begin{equation*}
\Big\{ (x_i,y_j)\in \DD \times \FF \mbox{ s.t.\ } R_{ij}>0 \mbox{ and } \mu_i>0\Big\} \supset \mathcal E.
\end{equation*}
Therefore, $\nu \ll \nu^{R^0} \ll \nu^{P^1}$. So once again, Corollary~\ref{cor:one_step_sinkhorn} shows that $Q^1$ is well defined, and that for all $i,j$,
\begin{equation*}
Q^1_{ij} = \frac{\nu_j}{\nu^{P^1}_{j}} P^1_{ij},
\end{equation*}
with convention $\frac{0}{0}=0$. The support of $Q^1$ is 
\begin{align*}
\Big\{ (x_i,y_j)\in \DD \times \FF &\mbox{ s.t.\ } P^1_{ij}>0 \mbox{ and } \nu_j>0\Big\}\\
&=\Big\{ (x_i,y_j) \in \DD \times \FF \mbox{ s.t. } R_{ij}>0 \mbox{ and } \mu_i>0 \mbox{ and } \nu_j>0\Big\} = \mathcal E.
\end{align*}

Then, a direct induction argument relying on the following formulas holding for all $n \in \N$ and all $i,j$:
\begin{equation}
\label{eq:induction_P_Q}
P^{n+1}_{ij} = \frac{\mu_i}{\mu^{Q^n}_i} Q^n_{ij} \qquad \mbox{and} \qquad Q^{n+1}_{ij} = \frac{\nu_i}{\nu^{P^{n+1}}_j} P^{n+1}_{ij},
\end{equation}
with convention $\frac{0}{0} = 0$ show that for all $n \geq 2$, $P^n$ and $Q^n$ are well defined and admit $\mathcal E$ as their common support.

Let us now show that $\mu^*$ and $\nu^*$ are well defined. Their role are symmetric, so we just need to show that $\mu^*$ is well defined. First $Q^1$ satisfies $\nu^{Q^1} = \nu$ and $H(Q^1 | R) < + \infty$. Therefore, the problem
\begin{equation*}
\inf \Big\{ H(\bar \mu | \mu),\, \bar \mu = \mu^Q \mbox{ for some } Q \mbox{ with } H(Q | R) < + \infty \mbox{ and } \nu^Q = \nu \Big\}
\end{equation*}
consists in minimizing the continuous (on its domain) and strictly convex function $\bar \mu \mapsto H(\bar \mu | \mu)$ over the nonempty compact convex set
\begin{equation*}
\Big\{ \bar \mu = \mu^Q \mbox{ for some } Q \mbox{ with } H(Q | R) < + \infty \mbox{ and } \nu^Q = \nu  \Big\}.
\end{equation*}
Hence, it admits a unique solution $\mu^*$.

Finally, let us show that $P^*$ and $Q^*$ are well defined. Once again, their role are symmetric, so we only show the existence of $Q^*$. We already saw that $\mu^*$ is well defined. By definition of the latter, there exists $\bar Q$ with $\nu^{\bar Q} = \nu$, $\mu^{\bar Q} = \mu^*$ and $H(\bar Q|R) < + \infty$. So $\Sch(R; \mu^*,\nu)$ consists in minimizing the continuous (on its domain) and strictly convex function $Q \mapsto H(Q|R)$ on the nonempty compact convex set 
\begin{equation*}
\Pi(\mu^*,\nu) \cap \Big\{ Q \in \P(\DD \times \FF) \mbox{ such that } H(Q|R)< + \infty \Big\}.
\end{equation*}
So it admits a unique solution $Q^*$.

\bigskip

\noindent \underline{Step 2}: A formula for $H(Q|R)$, for all $Q \in \Pi(\mu^*,\nu)$ with $H(Q|R) < + \infty$.
  
Recalling $Q^0=R$, we infer from \eqref{eq:induction_P_Q} that for all $(x_i,y_j)\in \mathcal E$ and $n \in \N^*$,
\begin{equation}
\label{eq:explicit_Qn}
Q^n_{ij} = \frac{\nu_j}{\nu^{P^n}_j} \times \frac{\mu_i}{\mu^{Q^{n-1}}_i} \times \dots \times \frac{\nu_j}{\nu^{P^1_j}} \times \frac{\mu_i}{\mu^{Q^0}_i} \times R_{ij}.
\end{equation}
Observe that in the product in the r.h.s., because we assumed that $(x_i,y_j) \in \mathcal E$, the common support of all the iterates of the Sinkhorn algorithm, all the factors are positive. 

In addition, for all $Q \in \Pi(\mu^*,\nu)$ with finite entropy w.r.t.\ $R$, the support of $Q$ is included in~$\mathcal E$. This is because $H(\mu^* | \mu) < + \infty$, and thereby $\mu^*\ll\mu$. Therefore, we deduce that $Q \ll Q^n$.

 So as a consequence of~\eqref{eq:explicit_Qn}, for all $i,j$ in the support of $Q$,
\begin{align*}
\log \frac{Q_{ij}}{R_{ij}} &=  \log \frac{Q_{ij}}{Q^n_{ij}} + \sum_{k=1}^n  \Big\{ \log \frac{\nu_j}{\nu^{P^k}_j} + \log \frac{\mu_i}{\mu^{Q^{k-1}}_i} \Big\}\\
&= \log \frac{Q_{ij}}{Q^n_{ij}}  + \sum_{k=1}^n  \Big\{ \log \frac{\nu_j}{\nu^{P^k}_j} - \log \frac{\mu^*_i}{\mu_i} + \log \frac{\mu^*_i}{\mu^{Q^{k-1}}_i}  \Big\}.
\end{align*}
Let us multiply this equality by $Q_{ij}$, and sum over $i,j$. We get
\begin{equation*}
\sum_{i,j} Q_{ij} \log \frac{Q_{ij}}{R_{ij}} = \sum_{i,j} Q_{ij} \log \frac{Q_{ij}}{Q^n_{ij}} + \sum_{k=1}^n \left\{ \sum_{i,j} Q_{ij}\log \frac{\nu_j}{\nu^{P^k}_j} - \sum_{i,j} Q_{ij}\log \frac{\mu^*_i}{\mu_i} + \sum_{i,j} Q_{ij}\log \frac{\mu^*_i}{\mu^{Q^{k-1}}_i} \right\}.
\end{equation*}
Fix $k$ in $\{1, \dots, n\}$, and consider the term:
\begin{equation*}
    \sum_{i,j} Q_{ij}\log \frac{\nu_j}{\nu^{P^k}_j} = \sum_j \left( \sum_i Q_{ij} \right) \log \frac{\nu_j}{\nu^{P^k}_j}.
\end{equation*}
As the second marginal of $Q$ is $\nu$, we find
\begin{equation*}
    \sum_{i,j} Q_{ij}\log \frac{\nu_j}{\nu^{P^k}_j} = \sum_{j} \nu_j\log \frac{\nu_j}{\nu^{P^k}_j}.
\end{equation*}
As the first marginal of $Q$ is $\mu^*$, we can reason in the same way for the other terms of the same type, and find 
\begin{equation*}
\sum_{i,j} Q_{ij} \log \frac{Q_{ij}}{R_{ij}} = \sum_{i,j} Q_{ij} \log \frac{Q_{ij}}{Q^n_{ij}} + \sum_{k=1}^n \left\{ \sum_{j} \nu_j \log \frac{\nu_j}{\nu^{P^k}_j} - \sum_{i} \mu^*_i\log \frac{\mu^*_i}{\mu_i} + \sum_{i} \mu^*_{i}\log \frac{\mu^*_i}{\mu^{Q^{k-1}}_i} \right\}.
\end{equation*}
We now use the Definition~\ref{def:relative_entropy} of the relative entropy to find that this identity means:
\begin{align*}
 H(Q|R) - \mathsf M(R) + \mathsf M(Q) &= H(Q|Q^n) - \mathsf M(Q^n) + \mathsf M(Q) \\
 &\qquad+\sum_{k=1}^n \Big\{ H(\nu| \nu^{P^k}) - \mathsf M(\nu) + \mathsf M(\nu^{P^k})
    - H(\mu^* | \mu) + \mathsf M(\mu^*) - \mathsf M(\mu) \\
    &\hspace{147pt}+ H(\mu^* | \mu^{Q^{k-1}}) - \mathsf M(\mu^*) + \mathsf M(\mu^{Q^{k-1}}) \Big\},
\end{align*}
or, simplifying the masses $\M(Q)$ and $\M(\mu^*)$ appearing several times,
\begin{align*}
 H(Q|R) - \mathsf M(R) = H(Q|Q^n) - \mathsf M(Q^n) \!+\!\sum_{k=1}^n \!\Big\{ H(\nu| \nu^{P^k}) - \mathsf M(\nu) +& \mathsf M(\nu^{P^k})
    - H(\mu^* | \mu) - \mathsf M(\mu) \\
    &+ H(\mu^* | \mu^{Q^{k-1}})  + \mathsf M(\mu^{Q^{k-1}}) \Big\}.
\end{align*}
Let us check how the masses simplify. By \eqref{eq:mass_identities} as $Q$ and $Q^k$, $k \in \N^*$ admit $\nu$ as their second marginals, they have the same total masses, and it coincides with the one of their first marginals. Namely,
\begin{equation*}
    \forall k \geq 1,\qquad \M(\nu) =\M(Q^k) = \M(Q) = \M(\mu^{Q^k}) =  \M(\mu^*).
\end{equation*}
In the same way,
\begin{equation*}
    \forall k \geq 1,\qquad \M(\mu) =\M(P^k) = \M(\nu^{P^k}).
\end{equation*}
And finally, as $Q^0 = R$, we also have
\begin{equation*}
   \M(R) =\M(Q^0) = \M(\mu^{Q^0}).
\end{equation*}
Coming back to our entropy identity, we can simplify more and get 
\begin{equation}
\label{eq:H(Q|Qn)}
H(Q|R) = H(Q | Q^n) + \sum_{k=1}^n \Big\{ H(\nu | \nu^{P^k}) - H(\mu^* | \mu)\Big\} + \sum_{k=1}^n H(\mu^* | \mu^{Q^{k-1}}).
\end{equation}
(The only subtlety is that for $k = 1$ and for $k=1$ only, $\M(\mu^{Q^{k-1}})$ does not simplify with $\M(\nu)$. But then $\M(\mu^{Q^{k-1}}) = \M(\mu^{Q^{0}})$ simplifies with $\M(R)$, and $\M(\nu)$ simplifies with $\M(Q^n)$.)

Now, we claim that every term in the first sum in the r.h.s.\ is nonnegative, \emph{i.e.}\ that $H(\nu | \nu^{P^k}) \geq H(\mu^* | \mu)$. First, we know that $\nu^{Q^k} = \nu$, so that:
\begin{align*}
H(\nu | \nu^{P^k}) &= \sum_j \nu_j \log \frac{\nu_j}{\nu^{P^k}_j} + \M(\nu^{P^k}) - \M(\nu) \\
&= \sum_{ij} Q^k_{ij} \log \frac{\nu_j}{\nu^{P^k}_j} +\M(\nu^{P^k}) - \M(\nu)\\
&= \sum_{ij} Q^k_{ij} \log \frac{Q^k_{ij}}{P^k_{ij}} + \M(P^k) - \M(Q^k)= H(Q^k|P^k),
\end{align*}
where we used at the third line that from \eqref{eq:induction_P_Q}, we know that for all $j$ and $1 \leq k \leq n$, $\nu_j / \nu^{P^k}_j = Q^k_{ij} / P^k_{ij}$.
Second, by Corollary~\ref{cor:one_step_sinkhorn},
\begin{equation*}
H(\nu | \nu^{P^k}) =  H(Q^k|P^k) \geq H(\mu^{Q^k} | \mu^{P^k}) = H(\mu^{Q^k} | \mu).
\end{equation*}
Finally, $Q^k$ has finite entropy w.r.t.\ $R$ (use for instance~\eqref{eq:explicit_Qn} with $n=k$) and admits $\nu$ as a first marginal. So by optimality of $\mu^*$, $H(\mu^{Q^k} | \mu) \geq H(\mu^*|\mu)$. Our claim follows. 

\bigskip

\noindent\underline{Step 3}: Consequence of~\eqref{eq:H(Q|Qn)}, convergence of the marginals.

As a consequence of Step~2, both sums in the r.h.s.\ of~\eqref{eq:H(Q|Qn)} are bounded sums of nonnegative terms. Therefore, they converge as $n \to + \infty$, and their terms tend to $0$ as $k \to + \infty$. We deduce in particular that
\begin{equation*}
H(\mu^* | \mu^{Q^n}) \underset{n \to + \infty}{\longrightarrow} 0.
\end{equation*}
In particular, by continuity of $H$ w.r.t.\ its second variable as stated in Proposition~\ref{prop:continuity_H}, and by compactness of~$\{ \bar \mu \in \mathcal M_+(\DD) \mbox{ s.t. } \M(\bar \mu) = \M(\nu)\}$, $\mu^{Q^n} \to \mu^*$. So now let us pick $\bar Q$ any limit point of $(Q^n)$. Such a limit point exist by compactness of~$\{ Q \in \mathcal M_+(\DD \times \FF) \mbox{ s.t. } \M(Q) = \M(\nu)\}$. It follows from $\mu^{Q^n} \to \mu^*$ that $\mu^{\bar Q} = \mu^*$.

\bigskip

\noindent \underline{Step 4}: $\bar Q = Q^*$.

Let us show that $\bar Q = Q^*$, so that actually the whole sequence $(Q^n)$ converges towards $Q^*$. On the one hand, passing to the limit $n \to + \infty$ along the subsequences generating $\bar Q$ in~\eqref{eq:H(Q|Qn)} and using the continuity of $H$ w.r.t.\ the second variable as stated in Proposition~\ref{prop:continuity_H}, we find
\begin{equation}
\label{eq:H(Q|barQ)}
H(Q|R) = H(Q |\bar Q) + \sum_{k=1}^{+\infty} \Big\{H(\nu | \nu^{P^k}) - H(\mu^* | \mu) \Big\}+ \sum_{k=1}^{+ \infty} H(\mu^* | \mu^{Q^{k-1}}).
\end{equation}
On the other hand, as for all $n \in \N^*$, $Q^n \ll R$, this is also true for $\bar Q$. In particular, $H(\bar Q | R) < + \infty$, and as $\bar Q \in \Pi(\mu^*,\nu)$, we can apply~\eqref{eq:H(Q|barQ)} with $\bar Q$ in place of $Q$, and find
\begin{equation}
\label{eq:H(barQ|R)}
H(\bar Q|R) = \sum_{k=1}^{+\infty} \Big\{ H(\nu | \nu^{P^k}) - H(\mu^* | \mu) \Big\} + \sum_{k=1}^{+ \infty} H(\mu^* | \mu^{Q^{k-1}}).
\end{equation}
Now it remains to apply~\eqref{eq:H(Q|barQ)} with $Q = Q^*$ and to plug the previous equality to find
\begin{equation*}
H(Q^*|R) = H(Q^* |\bar Q) + H(\bar Q | R).
\end{equation*}
As by optimality of $R^*$, $H(\bar Q | R) \geq H(Q^*|R)$, we can conclude that $H(Q^*|Q) = 0$. Therefore, $\bar Q = Q^*$, as announced.

The proof of $P^n \to P^*$ follows the same lines.
\end{proof}

As a free output of the proof of Theorem~1, we can show that we could have swapped $\mu$ and $\bar \mu$, and $\nu$ and $\bar \nu$ in the definitions~\eqref{eq:def_mu*_nu*} of $\mu^*$ and $\nu^*$ respectively. This is justified in the following remark.
\begin{rem}
	\label{rem:swap_mu_mu*}
	Observe the following optimization problem, where $R$, $\mu$ and $\nu$ are given, and where the competitor is $\bar \nu$: 
	\begin{equation}
	\label{eq:other_def_nu*}
	\min \Big\{ H(\nu | \bar \nu) \, \Big|\, \bar \nu = \nu^P, \mbox{ for some } P \mbox{ with } H(P|R) < + \infty \mbox{ and } \mu^P = \mu \Big\}.
	\end{equation}
    This problem is almost the same as the one defining $\nu^*$ in~\eqref{eq:def_mu*_nu*}, except from the fact that $\nu$ and $\bar \nu$ are swapped in the relative entropy. In this remark, we justify that the solution of this problem is $\nu^*$ as well, and that the corresponding optimal value is $H(\mu^*|\mu)$.
	
	Provided there exists a competitor $\bar \nu$ for this problem with $H(\nu |\bar \nu) < +\infty$, we can find $P$ such that $H(P|R)< + \infty$ and $P \in \Pi(\mu, \bar \nu)$, and $Q \in \P(\DD \times \FF)$ defined for all $i,j$ by
	\begin{equation*}
	Q_{ij} := \frac{\nu_j}{\bar\nu_j} P_{ij},
	\end{equation*}
	which is legitimate since $H(\nu | \bar \nu) < + \infty$. We have then $H(Q|R) < + \infty$ and $\nu^Q = \nu$. Hence, using the definition~\eqref{eq:def_mu*_nu*} of $\mu^*$, we have
	\begin{equation*}
	H(\nu | \bar \nu) = H(Q | P) \geq H(\mu^Q | \mu ) \geq H(\mu^* | \mu ),
	\end{equation*}
	where the first equality is a direct computation, and where the first inequality is obtained using Corollary~\ref{cor:one_step_sinkhorn}.
	
	On the other hand, as soon as the assumption of Theorem~\ref{thm:convergence_sinkhorn} holds, $\nu^*$ is a competitor for the problem in~\eqref{eq:other_def_nu*}, and so in particular $H(\nu | \nu^*) \geq H(\mu^*|\mu)$. But because the terms of the first series in~\eqref{eq:H(Q|barQ)} tend to $0$ and $\nu^{P^k} \to \nu^*$, we conclude that actually, $H(\nu | \nu^*) = H(\mu^*|\mu)$ and $\nu^*$ is a solution of~\eqref{eq:other_def_nu*}. Finally, it is easy to see that a solution $\bar \nu$ of~\eqref{eq:other_def_nu*} must satisfy $\bar \nu \ll \nu$ (because conditioning on the support of $\nu$ reduces the entropy), and by strict convexity of $\bar \nu \mapsto H(\nu | \bar \nu)$ on the set $\{ \bar \nu \ll \nu \}$, under the assumption of Theorem~\ref{thm:convergence_sinkhorn}, the problem~\eqref{eq:other_def_nu*} admits $\nu^*$ as its unique solution, so that~\eqref{eq:other_def_nu*} can be used as an alternative definition of $\nu^*$. 
	
	Of course, we could argue in the same way to provide an alternative definition of $\mu^*$, and we have the following equalities:
	\begin{equation*}
	H(\nu | \nu^*) = H(\mu^* | \mu) \qquad \mbox{and} \qquad H(\mu | \mu^*) = H(\nu^* | \nu).
	\end{equation*}
	In particular, $\mu^* \sim \mu$ and $\nu^* \sim \nu$ in the sense of measures.
\end{rem}

We also give another remark concerning the generalization of Theorem~\ref{thm:convergence_sinkhorn} to Polish spaces.
\begin{rem}
We crucially use the fact that $\DD$ and $\FF$ are finite in order to obtain~\eqref{eq:H(Q|barQ)} and~\eqref{eq:H(barQ|R)}. In the continuous case, as $H$ is not more than lower semicontinuous w.r.t the second variable, identity~\eqref{eq:H(Q|barQ)} becomes an inequality, where $=$ is replaced by $\geq$, which is the good direction for the proof. The difficulty is then to find an equality sign in~\eqref{eq:H(barQ|R)}.
\end{rem}

\section{$\mathbf{\Gamma}$-convergence in the marginal penalization problem}
\label{sec:gamma_cv}

In this section, we want to show that when $R$, $\mu$ and $\nu$ are such that the Schr\"odinger problem $\Sch(R; \mu,\nu)$ has no solution, then the limit points $P^*$ and $Q^*$ given by Theorem~\ref{thm:convergence_sinkhorn} are relevant in view of the possible applications of the Sinkhorn algorithm. 

\bigskip

To do so, let us think of $R$ as an imperfect theoretical model describing the coupling between the initial and final positions of the particles of a large system. Also, let us imagine that $\mu$ and $\nu$ are data obtained by measuring the positions of the particles of the actual system that $R$ is supposed to describe, at the initial and final time. In this situation, if  $\Sch(R; \mu,\nu)$ has a solution $R^*$, this solution is interpreted as the model that is the closest to $R$ that can explain the data. 

However, even when $R$ is a rather good model, and when $\mu$ and $\nu$ are rather precise measurements, it is possible that $\Sch(R; \mu,\nu)$ has no solution for several reasons:

\begin{itemize}
\item The first reason could be that our modeling does not take into account some physical phenomena. For instance, in Subsection~\ref{subsec:unbalanced}, we will consider the case where the true system allows creation or annihilation of mass with very small probability, whereas the modeling does not.

\item Another reason could be that $\mu$ and $\nu$ are only approximations of the real marginals. This can result from imprecise or biased measurements, or from a restricted amount of collected data. This will be considered in Subsection~\ref{subsec:balanced}. 
\end{itemize}

In both cases, it is very natural to relax the marginal constraints in~\eqref{eq:def_Sch} by introducing a fitting term in the value functional, that cancels when the constraints are satisfied, but which remains finite otherwise.

The main result of this section asserts that in these two situations, that are actually very close, the limit points $P^*$ and $Q^*$ of the Sinkhorn algorithm allow to compute the solution of the relaxed problem when the new fitting term takes the form of an entropy, in the limit where the level of marginal penalization tends to $+ \infty$. The second case is a direct consequence of the first one, but that we wanted to keep separated because it does not have the same interpretation.

\subsection{Unbalanced problems}
\label{subsec:unbalanced}
In this subsection, we give ourselves $R \in \mathcal M_+(\DD \times \FF)$, $\mu \in \mathcal M_+(\DD)$ and $\nu \in \mathcal M_+(\FF)$ as before, and we study the following optimization problem, which is a reasonable modification of $\Sch(R; \mu,\nu)$ where the marginal constraints are replaced with marginal penalizations:
\begin{equation}
\label{eq:unbalanced_problem_lambda}
 \min \Big\{ H(\bar R |R) + \lambda \big( H(\mu^{\bar R} | \mu) + H(\nu^{\bar R} | \nu) \big)  \, \Big | \, \bar R \in \mathcal M_+(\DD \times \FF) \Big\},
\end{equation}
where $\lambda>0$ parametrizes the level of penalization. 

This approach is extremely reminiscent of the idea introduced by Liero, Mielke and Savar\'e in~\cite{Liero2018} to deal with unbalanced data, that is, when $\M(\mu) \neq \M(\nu)$, in optimal transport problems. This was the starting point of the theory of \emph{unbalanced optimal transport}, also discovered independently by other teams~\cite{kondratyev2016new,Chizat2018}. 

More precisely, we will study the limit of the problem in~\eqref{eq:unbalanced_problem_lambda} as $\lambda \to + \infty$. In this limit, it is actually more convenient to call $\eps = 1/\lambda$ and to multiply the value functional by $\eps$, to find the problem that we call $\Sch^\eps(R; \mu,\nu)$:
\begin{equation*}
\Sch^\eps (R; \mu,\nu) := \min \Big\{ \eps H(\bar R |R) +  H(\mu^{\bar R} | \mu) + H(\nu^{\bar R} | \nu)  \, \Big | \, \bar R \in \mathcal M_+(\DD \times \FF) \Big\}.
\end{equation*}
As we want to study the behavior of this problem in the limit $\eps \to 0$, we define the following functionals:
\begin{align*}
\Lambda^\eps : \bar R \in \mathcal M_+(\DD \times \FF) &\mapsto \eps H(\bar R |R) +  H(\mu^{\bar R} | \mu) + H(\nu^{\bar R} | \nu), \\
\Lambda:  \bar R\in \mathcal M_+(\DD \times \FF) &\mapsto  \chi_{H(\bar R | R)< + \infty} + H(\mu^{\bar R} | \mu) + H(\nu^{\bar R} | \nu),
\end{align*}
where $\chi_{H(\bar R | R)< + \infty}$ is the convex indicatrix taking value $0$ on the set 
\begin{equation*}
\Big\{ \bar R \in \mathcal M_+(\DD \times \FF) \mbox{ such that } H(\bar R | R) < + \infty\Big\},
\end{equation*}
and $+\infty$ elsewhere.

The following proposition follows from standard arguments in the theory of $\Gamma$-convergence, see for instance~\cite[Theorem~1.47]{braides2002gamma}, and from the strict convexity of the relative entropy w.r.t.\ its first variable. We omit the proof.
 \begin{prop}
 	\label{prop:gamma_cv}
 	We have:
 	\begin{equation*}
 	\Gamma-\lim_{\eps \to 0} \Lambda^\eps = \Lambda.
 	\end{equation*}
 	
 	In particular, assuming that $\Lambda$ is not uniformly infinite, let us call $R_{\mathrm{opt}}$ one of its minimizers, $\mu^g := \mu^{R_{\mathrm{opt}}}$ and $\nu^g := \nu^{R_{\mathrm{opt}}}$. The marginals $\mu^g$ and $\nu^g$ do not depend on the choice of $R_{\mathrm{opt}}$, and as $\eps \to 0$, the unique solution $R^\eps$ of $\Sch^\eps(R; \mu,\nu)$ exists and converges towards the solution of $\Sch(R; \mu^g, \nu^g)$.
 \end{prop}
 
 \begin{rem}
 In the notations $\mu^g$ and $\nu^g$, the $g$ stands for \emph{geometric}. This is because as shown in Theorem~\ref{thm:link_P*_Q*_R*}, $\mu^g$ and $\nu^g$ are respectively the componentwise geometric means of $\mu$ and $\mu^*$, and of $\nu$ and $\nu^*$.
 \end{rem}

Therefore, studying the behavior of $\Sch^\eps(R; \mu,\nu)$ in the limit $\eps \to 0$ reduces to the study of the Schr\"odinger problem with modified marginals $\mu^g$ and $\nu^g$. The following theorem shows the link between $R^*$ -- the solution of $\Sch(R; \mu^g, \nu^g)$ -- on the one hand, and $P^*$ and $Q^*$ from Theorem~\ref{thm:convergence_sinkhorn} on the other hand.
\begin{theo}
	\label{thm:link_P*_Q*_R*}
	Let $R \in \mathcal M_+(\DD \times \FF)$, $\mu \in \mathcal M_+(\DD)$ and $\nu \in \mathcal M_+(\FF)$ satisfy Assumption~\ref{ass:condition_sinkhorn_well_defined}. Then the functional $\Lambda$ is not uniformly infinite. Moreover, considering $P^*$ and $Q^*$ as given by Theorem~\ref{thm:convergence_sinkhorn}, and $\mu^g$ and $\nu^g$ as given by Proposition~\ref{prop:gamma_cv}, the solution of $\Sch(R; \mu^g, \nu^g)$ is the componentwise geometric mean of $P^*$ and $Q^*$, that is, the matrix $R^*$ defined for all $i,j$ by
	\begin{equation}
	\label{eq:formula_R*}
	R^*_{ij} := \sqrt{P^*_{ij} Q^*_{ij}}.
	\end{equation}
	
	Also, if $\mu^*$ and $\nu^*$ are defined by~\eqref{eq:def_mu*_nu*}, $\mu^g$ and $\nu^g$ are the componentwise geometric means of $\mu^*$ and $\mu$ for the first one, and of $\nu^*$ and $\nu$ for the second one. In other terms, we have for all $i,j$,
    \begin{equation}
	\label{eq:bar_mu_bar_nu}
	\mu^g_i = \sqrt{\mu^*_i \mu_i} \qquad \mbox{and} \qquad \nu^g_j = \sqrt{\nu^*_j \nu_j}.
	\end{equation}
\end{theo}
\begin{rem}
\begin{itemize}
    \item Having in mind the approach of~\cite{Liero2018}, we can give the following interpretation of the matrix $R^*$: In the degenerate case where the Schr\"odinger problem has no solution, it is necessary to allow creation and annihilation of mass to find solutions. Following~\cite{Liero2018}, we can do this by replacing the balanced problem $\Sch(R; \mu,\nu)$ by the unbalanced problem $\Sch^\eps(R; \mu,\nu)$. Following this analogy, $\lambda = \frac{1}{\eps}$ parametrizes the cost of creating particles. The matrix $R^*$ from Theorem~\ref{thm:link_P*_Q*_R*} is therefore the limit of these solutions when the cost of creating or destroying matter tends to $+\infty$. 
    \item A small adaptation of the proof shows that given $\alpha \in [0,1]$, if we replace the problem in~\eqref{eq:unbalanced_problem_lambda} by
    \begin{equation*}
 \min \Big\{ H(\bar R |R) + \lambda \Big( (1-\alpha) H(\mu^{\bar R} | \mu) + \alpha H(\nu^{\bar R} | \nu) \Big)  \, \Big | \, \bar R \in \mathcal M_+(\DD \times \FF) \Big\},
\end{equation*}
and if we call $R^{\alpha,\lambda}$ its solution, then as $\lambda \to + \infty$, we have for all $i,j$:
\begin{equation*}
    R^{\alpha,\lambda}_{ij} \underset{\lambda \to +\infty}{\longrightarrow} \big(P^*_{ij}\big)^{1-\alpha}\big( Q^*_{ij}\big)^{\alpha}.
\end{equation*}
    \end{itemize}
\end{rem}

To prove this theorem, we will need to study carefully the optimality conditions for $\mu^*$ and $\nu^*$. This could be done writing the Karush-Kuhn-Tucker conditions for the corresponding optimalization problems. We will rather adopt a more hand by hand approach, that is more likely to be generalizable in the continuous case. This is done in the following proposition.

\begin{prop}
	\label{prop:optimality_conditions_mu*_nu*}
	Assume that the conditions of Theorem~\ref{thm:convergence_sinkhorn} are fulfilled. For all $i,j$, we have
	\begin{equation}
	\label{eq:optimal_P*_Q*}
	P^*_{ij} = \frac{\mu_i}{\mu^*_i} Q^*_{ij} \qquad \mbox{and} \qquad Q^*_{ij} = \frac{\nu_j}{\nu^*_j} P^*_{ij},
	\end{equation}
	with convention $\frac{0}{0} = 0$. In particular, $P^*$ and $Q^*$ are equivalent, and we call $\mathcal S$ their common support. Also, recall the definition of $\mathcal E$ in~\eqref{eq:def_E}. Of course $\mathcal S \subset \mathcal E$.
	Finally, we call for all $i,j$ 
	\begin{equation}
	\label{eq:def_phi_psi}
	\varphi_i := \log \frac{\mu^*_i}{\mu_i} \qquad \mbox{and} \qquad \psi_j := \log \frac{\nu^*_j}{\nu_j}. 
	\end{equation}
	
	For all $(i,j) \in \mathcal E$, $\varphi_i$ and $\psi_j$ are well defined in $\R$, and:
	\begin{equation}
	\label{eq:optimality_conditions_mu*_nu*}
	\left\{  
	\begin{aligned}
	\varphi_i + \psi_j &= 0, &&\mbox{if }(i,j) \in \mathcal S,\\
	\varphi_i + \psi_j &\geq 0, &&\mbox{if }(i,j) \in \mathcal E.	 
	\end{aligned}
	\right.
	\end{equation}
\end{prop}
\begin{proof}[Proof of Proposition~\ref{prop:optimality_conditions_mu*_nu*}]
	To get~\eqref{eq:optimal_P*_Q*}, it suffices to let $n$ tend to $+\infty$ in~\eqref{eq:induction_P_Q}. The fact that $\mathcal S \subset \mathcal E$ relies on the closed property of $P^n$ and $Q^n$ defined in~\eqref{eq:sinkhorn} to have its support included in $\mathcal E$ for $n\geq 2$. If $(i,j) \in \mathcal E$, let us check that $\varphi_i$ and $\psi_j$ are well defined. On the one hand, by definition of $\mathcal E$, $i$ is in the support of $\mu$ and $j$ is in the support of $\nu$. On the other hand, as observed in Remark~\ref{rem:swap_mu_mu*}, $\mu^* \sim \mu$ and $\nu^* \sim \nu$. Our claim follows.
	
	Now, let $(i,j) \in \mathcal S$. A consequence of~\eqref{eq:optimal_P*_Q*} is
	\begin{equation*}
	P^*_{ij} = \frac{\mu^*_i}{\mu_i}\frac{\nu^*_j}{\nu_j}P^*_{ij} =\exp(\varphi_i + \psi_j) P^*_{ij}. 
	\end{equation*} 
	As $(i,j)$ is in the support of $P^*$ by definition of $\mathcal S$, we conclude that $\varphi_i + \psi_j = 0$.
	
	Finally, it remains to prove that for all $(i,j) \in \mathcal E$, $\varphi_i + \psi_j \geq 0$. For this we use the optimality of $H(\mu^*|\mu) = H(\mu^{Q^*} | \mu)$ over all $Q$ satisfying $\nu^Q = \nu$. So let us take $(i,j) \in \mathcal E$. As $\nu_j>0$, there exists $i'$ such that $(i',j) \in \mathcal S$, that is, such that $Q^*_{i'j}>0$. Let us define for $\eps>0$
	\begin{equation*}
	Q^\eps = Q^* + \eps \delta_{ij} - \eps \delta_{i'j},
	\end{equation*}
	where $\delta_{ij}$ is the matrix whose only nonzero coefficient is a one at position $(i,j)$, and similarly for $\delta_{i'j}$. If $\eps$ is sufficiently small, $Q^\eps \in \mathcal M_+(\DD \times \FF)$, $\nu^{Q^\eps} = \nu$ and with obvious notations, $\mu^{Q^\eps} = \mu^* + \eps \delta_i - \eps \delta_{i'}$.
	Therefore, for such $\eps$,
	\begin{equation*}
	H(\mu^{Q^\eps} | \mu ) \geq H(\mu^*| \mu).
	\end{equation*}
	derivating to the right this inequality at $\eps = 0$, we find
	\begin{equation*}
	\log\frac{\mu_i^*}{\mu_i} - \log \frac{\mu_{i'}^*}{\mu_{i'}} \geq 0,
	\end{equation*}
	which rewrites $\varphi_i - \varphi_{i'} \geq 0$. But $(i',j) \in \mathcal S$ so $\varphi_{i'} = - \psi_j$, and so $\varphi_i + \psi_j \geq 0$.
\end{proof}

With this proposition at hand, we can prove Theorem~\ref{thm:link_P*_Q*_R*}.
\begin{proof}[Proof of Theorem~\ref{thm:link_P*_Q*_R*}]
	The fact that under Assumption~\ref{ass:condition_sinkhorn_well_defined}, $\Lambda$ is not uniformly infinite follows from observing that $\Lambda(R^0)< + \infty$, where $R_0$ was defined Assumption~\ref{ass:condition_sinkhorn_well_defined}. Now we reason in two steps. First we will prove using Proposition~\ref{prop:optimality_conditions_mu*_nu*} that $R^*$ defined by~\eqref{eq:formula_R*} is an optimizer of $\Lambda$, and then that it is the solution of the Schr\"odinger problem between its marginals.
	
	\bigskip
	
	\noindent \underline{Step 1}: $R^*$ is an optimizer of $\Lambda$.
	
	To see that $R^*$ is an optimizer of $\Lambda$, we first give a formula relating the vectors $\varphi$ and $\psi$ as defined by formula~\eqref{eq:def_phi_psi} and the marginals $\mu^{R^*}$ and $\nu^{R^*}$ of $R^*$. Using~\eqref{eq:optimal_P*_Q*} and the definition~\eqref{eq:formula_R*} of $R^*$, we see that for all $i,j$,
	\begin{equation}
		\label{eq:RN_R*_P*_Q*}
	R^*_{ij} =  \sqrt{\frac{\nu_j}{\nu^*_j}} P^*_{ij} =  \sqrt{\frac{\mu_i}{\mu^*_i}} Q^*_{ij}.
	\end{equation}
	Summing respectively these identities w.r.t.\ $i$ and $j$, we deduce that for all $i,j$,
	\begin{equation*}
	\mu^{R^*}_i = \sqrt{\mu_i^* \mu_i} =\sqrt\frac{\mu^*_i}{\mu_i} \mu_i \qquad \mbox{and} \qquad \nu^{R^*}_j =  \sqrt{\nu_j^* \nu_j} = \sqrt\frac{\nu^*_j}{\nu_j} \nu_j.
	\end{equation*}
	Let us define for all $i,j$:
	\begin{equation*}
	Z^\mu_i := \log \frac{\mu^{R^*}_i}{\mu_i} = \frac{1}{2} \varphi_i  \qquad \mbox{and} \qquad Z^\nu_j := \log \frac{\nu^{R^*}_j}{\nu_j} = \frac{1}{2} \psi_j .
	\end{equation*}
	Note that for all $(x_i,y_j) \in \mathcal E$, $Z^\mu_i$ and $Z^\nu_j$ are well defined in $\R$.
	
	Now let $\bar R$ be such that $\Lambda(\bar R)< + \infty$. Using inequality~\eqref{eq:legendre_transform_entropy} to bound from below each relative entropy, we have
	\begin{align*}
	\Lambda(\bar R) &= H(\mu^{\bar R} | \mu ) + H( \nu^{\bar R} | \nu ) \\
	&\geq \cg Z^\mu, \mu^{\bar R} \cd - \cg e^{Z^\mu} - 1, \mu \cd + \cg Z^\nu, \nu^{\bar R} \cd - \cg e^{Z^\nu} - 1, \nu \cd \\
	&=  \frac{1}{2} \cg \varphi, \mu^{\bar R} \cd + \frac{1}{2} \cg \psi, \nu^{\bar R}\cd - \sum_i \big\{ \mu_i^{R^*} - \mu_i \big\} - \sum_j \big\{ \nu_j^{R^*} - \nu_j \big\}\\
	&= \frac{1}{2}\cg  \varphi \oplus \psi, \bar R \cd  + \M(\mu) + \M(\nu) - 2 \M(R^*),
	\end{align*}
	where $\varphi \oplus \psi$ is the matrix defined for all $i,j$ by $(\varphi \oplus \psi)_{ij} := \varphi_i + \psi_j$. Now, because of the second line of~\eqref{eq:optimality_conditions_mu*_nu*}, as the support of $\bar R$ is easily seen to be a subset of~$\mathcal E$, we get
	\begin{equation*}
	\Lambda(\bar R) \geq \M(\mu) + \M(\nu) - 2 \M(R^*).
	\end{equation*} 
	
On the other hand, by definition of $Z^\mu$ and $Z^\nu$,
	\begin{align*}
	\Lambda(R^*) &= H(\mu^{R^*} | \mu ) + H( \nu^{ R^*} | \nu ) \\
	&= \cg Z^\mu, \mu^{R^*} \cd + \M(\mu) - \M(R^*)+ \cg Z^\nu, \nu^{R^*} \cd + \M(\nu) - \M(R^*)\\
	&= \frac{1}{2}\cg \varphi \oplus \psi, R^* \cd + \M(\mu) + \M(\nu) - 2 \M(R^*).
	\end{align*}
	But now, as the support of $R^*$ is precisely $\mathcal S$, by the first line of ~\eqref{eq:optimality_conditions_mu*_nu*}, we get
	\begin{equation*}
	   \Lambda(R^*) = \M(\mu) + \M(\nu) - 2 \M(R^*).
	\end{equation*}
	We deduce that $\Lambda(\bar R)\geq \Lambda(R^*)$ and $R^*$ is indeed an optimizer of $\Lambda$. In particular, $\mu^g = \mu^{R^*}$ and $\nu^g = \nu^{R^*}$, which proves~\eqref{eq:bar_mu_bar_nu}.
	
	\bigskip
	
	\noindent\underline{Step 2}: $R^*$ is the solution of $\Sch(R;\mu^g, \nu^g)$.
	
	To show that $R^*$ solves the Schr\"odinger problem between its marginals, we consider another $\bar R \in \mathcal M_+(\DD \times \FF)$ such that $\bar R \in \Pi(\mu^g, \nu^g)$ and $H(\bar R | R)< + \infty$. Then, for $\eps>0$, we define
	\begin{equation*}
	P^\eps := P^* + \eps (\bar R - R^*).
	\end{equation*}
	As $R^* \ll P^*$ (see~\eqref{eq:formula_R*}), whenever $\eps$ is sufficiently small, $P^\eps \in \mathcal M_+(\DD \times \FF)$, and in addition, we easily check that $P^\eps \in \Pi(\mu, \nu^*)$. So by definition~\eqref{eq:def_P*_Q*} of $P^*$,
	\begin{equation*}
	H(P^*|R) \leq H(P^\eps | R).
	\end{equation*}
	Derivating this inequality to the right at $\eps = 0$, we find
	\begin{equation*}
	\sum_{ij}R^*_{ij} \log \frac{P^*_{ij}}{R_{ij}} \leq \sum_{ij}\bar R_{ij} \log \frac{P^*_{ij}}{R_{ij}},
	\end{equation*}
	with convention $\frac{0}{0} = 0$, $0 \log 0 = 0$ and $a \log 0 = - \infty$ for all $a>0$. In particular, we deduce that $\bar R \ll P^* \sim R^*$, and our inequality rewrites
	\begin{equation*}
	H(R^* | R) + H(\bar R | P^*) - H(R^* | P^*) \leq H(\bar R | R).
	\end{equation*}
	The last thing to observe is that because of~\eqref{eq:RN_R*_P*_Q*}, $R^*$ is the solution of the Schr\"odinger problem $\Sch(P^*; \mu^g, \nu^g)$: a direct application of~\eqref{eq:legendre_transform_entropy} with $Z_{ij} = \log\frac{R^*_{ij}}{P^*_{ij}} = - \bar \psi_j$ (which is well defined on the support of $P^*$, and so on the support of $\bar R$) provides
	\begin{align*}
	H(\bar R | P^*) &\geq \cg Z, \bar R \cd - \cg e^Z -1, P^*\cd\\
	&= - \cg \bar \psi, \nu^g \cd + \sum_{ij}P^*_{ij} - R^*_{ij} \\
	&= \cg Z, R^*\cd + \sum_{ij}P^*_{ij} - R^*_{ij}  = H(R^* | P^*).
	\end{align*}
	The result follows.
\end{proof}
\begin{rem}
	In Step 2, we used a particular case of the following more general result that is proved in the same way:
\end{rem}
\begin{lem}
	Let $R \in \mathcal M_+(\DD \times \FF)$, $\mu,\mu' \in \mathcal M_+(\DD)$ and $\nu, \nu' \in \mathcal M_+(\FF)$. Assume that $\Sch(R; \mu,\nu)$ admits a solution $P$ and that $\Sch(P; \mu', \nu')$ admits a solution $Q$. Then the unique solution of $\Sch(R; \mu',\nu')$ exists: it is $Q$.
\end{lem} 

\subsection{Balanced version}

\label{subsec:balanced}

In the last subsection, we interpreted the fact that $\Sch(R; \mu,\nu)$ has no solution by the fact that our model does not incorporate the ability of the real system to create or destroy mass. In that case, the total mass of $R^*$ is not the same as the one of $\mu$ and $\nu$ in general, even when the latter two coincide. Therefore, $R^*$ cannot be interpreted directly as a joint law for the initial and final positions of the particles. Following the lines of~\cite{Liero2018}, we see that its interpretation is actually rather complicated.

In this subsection, we want to consider the case where the real system under study is truly balanced, that is, no creation of annihilation of mass is possible at all. In this situation, whatever the way we are obtaining the data, $\mu$ and $\nu$ must have the same mass, and up to renormalizing, we can assume that they are probability measures. We want to interpret the fact that $\Sch(R; \mu,\nu)$ has no solution by the fact that $\mu$ and $\nu$ are imperfect measurements of the true marginals, and we want to find a \emph{probability} measure $\bar R^*$ that is entropically close to $R$ while having its marginals entropically close to $\mu$ and $\nu$, that can be interpreted as a joint law.

Therefore, we introduce the following problem that is a slight modification of $\Sch^\eps$ where the competitor $\bar R$ needs to be a probability measure: for all $R \in \P(\DD \times \FF)$, $\mu \in \P(\DD)$ and $\nu \in \P(\FF)$,
\begin{equation*}
\overline{\Sch}{}^\eps (R; \mu,\nu) := \min \Big\{ \eps H(\bar R |R) +  H(\mu^{\bar R} | \mu) + H(\nu^{\bar R} | \nu)  \, \Big | \, \bar R \in \P(\DD \times \FF) \Big\}.
\end{equation*}

The following theorem states the behaviour of this optimization problem as $\eps \to 0$, and is a direct adaptation of Theorem~\ref{thm:link_P*_Q*_R*} to the balanced case.

\begin{theo}
\label{thm:balanced}
    Let $R \in \P(\DD \times \FF)$, $\mu \in \P(\DD)$ and $\nu \in \P(\FF)$ satisfy the conditions of Assumption~\ref{ass:condition_sinkhorn_well_defined}, and call 
    \begin{equation*}
        \mathcal Z := \sum_{ij} \sqrt{P^*_{ij} Q^*_{ij}},
    \end{equation*}
    where $P^*$ and $Q^*$ are given by Theorem~\ref{thm:convergence_sinkhorn}. Then for all $\eps>0$, the solution $\bar R^\eps$ of $\overline{\Sch}{}^\eps(R; \mu, \nu)$ exists, is unique, and satisfies for all $i,j$:
    \begin{equation*}
        \bar R^\eps_{ij} \underset{\eps \to 0}{\longrightarrow} \bar R^*_{ij} := \frac{\sqrt{P^*_{ij} Q^*_{ij}}}{\mathcal Z}.
    \end{equation*}
    
    Its marginals are given for all $i,j$ by
    \begin{equation*}
	\mu^{\bar R^*}_i = \frac{\sqrt{\mu^*_i \mu_i}}{\mathcal Z} \qquad \mbox{and} \qquad \nu^{\bar R^*}_j = \frac{\sqrt{\nu^*_j \nu_j}}{\mathcal Z}.
	\end{equation*}
    \end{theo}
\begin{proof}
Theorem~\ref{thm:balanced} is a direct consequence of Theorem~\ref{thm:link_P*_Q*_R*} once noticed the following fact:
If $R,\mu,\nu$ are as in the statement of the theorem, if $\eps>0$ and if $R^\eps$ is the solution of $\Sch^\eps(R; \mu,\nu)$, then $R^\eps / \mathsf M(R^\eps)$ is the solution of $\overline{\Sch}{}^\eps(R; \mu,\nu)$. To see this, consider $R' \in \P(\DD \times \FF)$. Direct computations imply
\begin{gather*}
    H(R'|R) = \frac{H\Big(\mathsf M (R^\eps) R' \Big| R\Big)}{\mathsf M (R^\eps)} + \log \frac{1}{\mathsf M (R^\eps)} + 1 - \frac{1}{\mathsf M (R^\eps)},\\
    H\bigg( \frac{R^\eps}{\mathsf M (R^\eps)} \bigg| R \bigg) = \frac{H(R^\eps | R)}{\mathsf M (R^\eps)} + \log \frac{1}{\mathsf M (R^\eps)} + 1 - \frac{1}{\mathsf M (R^\eps)}.
\end{gather*}
By optimality of $R^\eps$, $H(\mathsf M (R^\eps) R' | R ) \geq H(R^\eps | R)$, and therefore $H(R'|R) \geq H(R^\eps / \mathsf M(R^\eps) | R)$. Our claims follows, and hence the theorem as $\mathsf M$ is a continuous functional and $\mathcal Z = \mathsf M(R^*)$, where $R^*$ is given by Theorem~\ref{thm:link_P*_Q*_R*}.
\end{proof}

 \section{Existence and support of the solutions to Schr\"odinger problems}
\label{sec:existence_of_a_solution}

In this section, our goal is to give a detailed study of the support of the solution of $\Sch(R; \mu,\nu)$ when the latter exists, or of the common one of $P^*$, $Q^*$ and $R^*$ from Theorems~\ref{thm:convergence_sinkhorn} and~\ref{thm:link_P*_Q*_R*} in the non-scalable case. This study will rely on a new interpretation of the well known existence conditions for the Schr\"odinger problem in finite spaces, for which we refer to~\cite{Brualdi1968,idel2016review}.

\bigskip

We start with our new formulation of these conditions of existence, which is very close to the ones introduced by Brualdi~\cite{Brualdi1968}, but has the advantage of helping understanding the shape of the support of the optimizers seen as a bipartite graph. 

In the second part of the section, we provide a theoretical procedure allowing to get the support of the optimizers, both in the approximately scalable and non-scalable cases, without using the Sinkhorn algorithm. This procedure will be used in the next section as a preliminary step, before launching the Sinkhorn algorithm, in order to recover a linear rate for the latter.

\subsection{A necessary and sufficient condition of existence for the Schr\"odinger problem in finite spaces}
\label{conditions_scalability}

Let us state a necessary and sufficient condition on $R$, $\mu$ and $\nu$ for the existence of a solution $R^*$ of $\Sch(R; \mu, \nu)$, that is, for $\Sch(R;\mu,\nu)$ to be scalable or approximately scalable. In order to do so, we need to give a few definitions. First, we endow the set $\DD \cup \FF$ with a bipartite graph structure related to $R$: we set
\begin{equation*}
\forall i=1, \dots, N \mbox{ and } j = 1, \dots, M, \qquad x_i \triangle y_j \Leftrightarrow R_{ij}>0.
\end{equation*}
We have $x_i \triangle y_j$ whenever it is possible to travel from $x_i$ to $y_j$ under $R$. We write indifferently $x_i \triangle y_j$ or $y_j \triangle x_i$.

With this structure in hand, we are able to push forward or pull backward subsets of $\DD$ and $\FF$, that is, we define: 
\begin{equation}
\label{push forward subset}
\begin{aligned}
&\forall A \subset \DD, & F_R(A) &:= \Big\{ y \in \FF \, | \, \exists x \in A \mbox{ s.t.\ } x \triangle y \Big\}, \\ 
&\forall B \subset \FF, & D_R(B) &:= \Big\{ x \in \DD \, | \, \exists y \in B \mbox{ s.t.\ } x \triangle y \Big\}.
\end{aligned}
\end{equation}
 Heuristically, for all $A \subset \DD$, $F_R(A)$ is the set of all possible final positions of particles starting from $A$, under $R$. Correspondingly, for all $B \subset \FF$, $D_R(B)$ is the set of all possible initial positions of particles arriving in $B$ under $R$. Notice the explicit mention of $R$ in the notations: in the following, we will allow ourselves to replace $R$ by any other measure $\bar R \in \mathcal M_+(\DD \times \FF)$.

The main result of this section is the following. 
\begin{theo}
	\label{thm:CNS}
	Let $R \in \mathcal M_+(\DD \times \FF)$, $\mu \in \mathcal M_+(\DD)$ and $\nu \in \mathcal M_+(\FF)$. The three following assertions are equivalent:
	\begin{enumerate}[label={(\alph*)}, leftmargin=.1\textwidth]
	\item \label{item:cond_A} $\M(\mu) = \M(\nu)$ and for all $A \subset \DD$, $\mu(A) \leq \nu (F_R(A))$.
	\item \label{item:cond_B} $\M(\mu) = \M(\nu)$ and for all $B \subset \FF$, $\nu(B) \leq \mu(D_R(B))$.
	\item \label{item:existence_solution} $\Sch(R; \mu, \nu)$ is scalable or approximately scalable.
\end{enumerate}
\end{theo}

Note that the implications~\ref{item:existence_solution} $\Rightarrow$~\ref{item:cond_A} and~\ref{item:existence_solution} $\Rightarrow$~\ref{item:cond_B} are straightforward, and that only the reverse implications are challenging. Also, we already noticed in Subsection~\ref{subsec:ass} that~\ref{item:existence_solution} implies Assumption~\ref{ass:condition_sinkhorn_well_defined}. Hence, it is also the case for~\ref{item:cond_A} and~\ref{item:cond_B}. 

The proof relies on the following Lemma~\ref{lem:non_degenerate_case}, which gives a necessary and sufficient condition on $R$, $\mu$ and $\nu$ ensuring $R^*$ to have the same support as $R$, that is, to be in the scalable case. In this statement, we use the notations $\mu^R$ and $\nu^R$ as defined in~\eqref{eq:notation_marginals}, and we work under Assumption~\ref{ass:full_support}, which is always possible under Assumption~\ref{ass:condition_sinkhorn_well_defined} up to considering subspaces of $\DD$ and $\FF$, see Subsection~\ref{subsec:ass}. 
\begin{lem}
	\label{lem:non_degenerate_case}
		Let $R \in \mathcal M_+(\DD \times \FF)$, $\mu \in \mathcal M_+(\DD)$ and $\nu \in \mathcal M_+(\FF)$, satisfying Assumption~\ref{ass:full_support}. The three following assertions are equivalent:
	\begin{enumerate}[label={(\alph*')}, leftmargin=.1\textwidth]
		\item \label{item:cond_A'} $\M(\mu) = \M(\nu)$ and for all $A \subset \DD$, $\mu(A) \leq \nu (F_R(A))$, with a strict inequality whenever $\mu^R (A) < \nu^R(F_R(A))$.
		\item \label{item:cond_B'} $\M(\mu) = \M(\nu)$ and for all $B \subset \FF$, $\nu(B) \leq \mu(D_R(B))$, with a strict inequality whenever $\nu^R (B) < \mu^R(D_R(B))$.
		\item \label{item:existence_solution'} $\Sch(R; \mu, \nu)$ is scalable. 
	\end{enumerate}
\end{lem}

In plain words, it highlights the difference between the approximately scalable and scalable cases, by showing that the scalable case consists in assuming as much strict inequalities in~\ref{item:cond_A} or in~\ref{item:cond_B} as possible.
Although both Theorem~\ref{thm:CNS} and Lemma~\ref{lem:non_degenerate_case} can be directly deduced from the work of Brualdi~\cite{Brualdi1968}, we provide in Appendix~\ref{proof of theorem CNS} a short and independent proof based on topological arguments.

\subsection{Theoretical construction of the support}
\label{subsec:procedure}

In the scalable case, the Sinkhorn algorithm is known to have a linear rate of convergence. On the other hand, in the approximately scalable case, the algorithm still converges, but the (unknown) convergence rate cannot be linear~\cite{Achilles1993}.

In this subsection, we study the support of the solution of the Schr\"odinger problem in the approximately scalable and non-scalable cases for the following reason. Take $R$, $\mu$ and $\nu$ such that $\Sch(R; \mu, \nu)$ is approximately scalable, $R^*$ the solution of this problem, and $\mathcal S$ the support of $R^*$. Then $\Sch(\1_{\mathcal S} R; \mu,\nu)$ is scalable and its solution is $R^*$. In particular, the Sinkhorn algorithm applied to this problem has a linear rate of convergence. Interestingly, a similar reasoning is valid in the non-scalable case, as we show in Proposition~\ref{prop:recover_linear_rate} below. 
 
 \bigskip

Without loss of generality, and for the sake of simplicity, in the whole subsection, we work under Assumption~\ref{ass:full_support}. By Remark~\ref{rem:swap_mu_mu*}, if $\mu^*$ and $\nu^*$ are defined by~\eqref{eq:def_mu*_nu*}, we have $\nu \sim \nu^*$ and $\mu \sim \mu^*$. So under Assumption~\ref{ass:full_support}, they have a full support as well.
 
 \begin{prop}
\label{prop:recover_linear_rate}
Let $R \in \mathcal M_+(\DD \times \FF)$, $\mu \in \mathcal M_+(\DD)$ and $\nu\in \mathcal M_+(\FF)$ satisfying Assumption~\ref{ass:full_support}. Let us call $\mathcal S$ the common support of $P^*$ and $Q^*$ from Theorem~\ref{thm:convergence_sinkhorn}, and $R^*$ from Theorem~\ref{thm:link_P*_Q*_R*}. 
Let $(P^n)_{n \in \N^*}$ and $(Q^n)_{n \in \N^*}$ be given by~\eqref{eq:computable Sinkhorn} applied to $\Sch(\1_{\mathcal S}R; \mu,\nu)$. They converge respectively towards $P^*$ and $Q^*$, both of them at a linear rate.
\end{prop}
\begin{proof}
Let $(P^n)_{n \in \N^*}$ and $(Q^n)_{n \in \N^*}$ be given by the equivalent formulations \eqref{eq:sinkhorn} and \eqref{eq:computable Sinkhorn} applied to $\Sch(\1_{\mathcal S}R; \mu,\nu)$. Let us show that $(P^n)$ converges towards $P^*$ at a linear rate. The case of $(Q^n)$ follows the same arguments. The idea is that if $(\tilde P^n)_{n \in \N^*}$ and $(\tilde Q^n)_{n \in \N^*}$ are given by~\eqref{eq:sinkhorn} and~\eqref{eq:computable Sinkhorn} applied to $\Sch(\1_{\mathcal S}R; \mu,\nu^*)$, then for all $n \in \N^*$, $P^n = \tilde P^n$. As the problem $\Sch(\1_{\mathcal S}R; \mu,\nu^*)$ is scalable (its solution, $P^*$, has the same support as $\1_{\mathcal S}R$), the rate of convergence of $(\tilde P^n)$ towards $P^*$ is linear, and the result follows.

So let us prove by induction that for all $n \in \N^*$, $P_n = \tilde P_n$. According to~\eqref{eq:sinkhorn}, $P^1$ and $\tilde P^1$ are solutions to the same problem, and therefore coincide. Let us now consider $n \in \mathbb{N}^*$ such that $P^n = \tilde P^n$ and show that $P^{n+1} = \tilde P^{n+1}$. By construction, the support of $P^{n+1}$ and $\tilde P^{n+1}$ is $\mathcal S$, so we just need to check that for all $(x_j,y_j)  \in \mathcal S$, $P^{n+1}_{ij} = \tilde P^{n+1}_{ij}$. By the first line of~\eqref{eq:optimality_conditions_mu*_nu*}, for all $(x_i,y_j) \in \mathcal S$, we have:
 \begin{equation*}
     \nu_j = \frac{\mu^*_i \nu^*_j}{\mu_i}.
 \end{equation*}
Hence, for all $(x_i,y_j) \in \mathcal S$:
\begin{align*}
P^{n+1}_{ij} = \frac{\mu_i}{\mu^{Q^n}_i}Q^n_{ij}
= \frac{\mu_i}{ \displaystyle \sum_{j'} \frac{\nu_{j'}}{\nu^{P^n}_{j'}} P^n_{ij'}} \times \frac{\nu_j}{\nu^{P^n}_{j}}P^n_{ij}
= \frac{\mu_i}{\displaystyle \frac{\mu_i^*}{\mu_i} \sum_{j'} \frac{\nu^*_{j'}}{\nu^{\tilde P^n}_{j'}} \tilde P^n_{ij'}} \times \frac{\mu_i^*}{\mu_i}\frac{\nu^*_j}{\nu^{\tilde P^n}_j}\tilde P^n_{ij}
= \frac{\mu_i}{\mu^{\tilde Q^n}_i}\tilde Q^n_{ij}
= \tilde P^{n+1}_{ij},
\end{align*}
where the change from $P^n$ to $\tilde P^n$ in the middle coming from the induction assumption $P^n=\tilde P^n$. The result follows.
\end{proof}
 
 Therefore, even in the non-scalable case, a way to improve the Sinkhorn algorithm consists in first finding $\mathcal S$, and then computing the solution of a scalable problem. We propose in this subsection a theoretical procedure allowing to get this support without using the Sinkhorn algorithm in both the approximately and non-scalable cases, and we will propose an approximate method for achieving this task numerically at Section~\ref{sec_numerical}. 
 
 \bigskip

To detail our procedure, we introduce a class of subsets of $\DD$ associated with a triple $(R;\mu,\nu)$.

\begin{defn}
Let $R \in \mathcal M_+(\DD \times \FF)$, $\mu \in \mathcal M_+(\DD)$ and $\nu\in \mathcal M_+(\FF)$ satisfying Assumption~\ref{ass:full_support}. Let us consider $R^*$ from Theorem~\ref{thm:link_P*_Q*_R*}. We say that a subset $A \subset \DD$ is the \emph{source of an isolated scalable problem} (or for short that $A$ is a SISP set) for $(R;\mu,\nu)$ if $A \neq \emptyset$ and:
\begin{itemize}
    \item The set $(\DD \backslash A) \times F_R(A)$ is $R^*$-negligible, \emph{i.e.}
    \begin{equation}
    \label{eq:SISP_isolated}
        R^*\Big((\DD \backslash A) \times F_R(A)\Big) = 0.
    \end{equation}
    \item For all $x_i \in A$ and $y_j \in \FF$,
    \begin{equation}
    \label{eq:support_R_R*}
        R^*_{ij}>0 \quad \Leftrightarrow \quad R_{ij}>0.
    \end{equation}
\end{itemize}
\end{defn}
We show at~\cref{fig:SISP_set} an illustration of what a SISP is.

Of course, as $P^*$ and $Q^*$ from Theorem~\ref{thm:convergence_sinkhorn} are equivalent to $R^*$ in the sense of measures, we could have replaced $R^*$ in the previous definition by one of them.

\begin{figure}
	\centering
	\includegraphics[scale=.15]{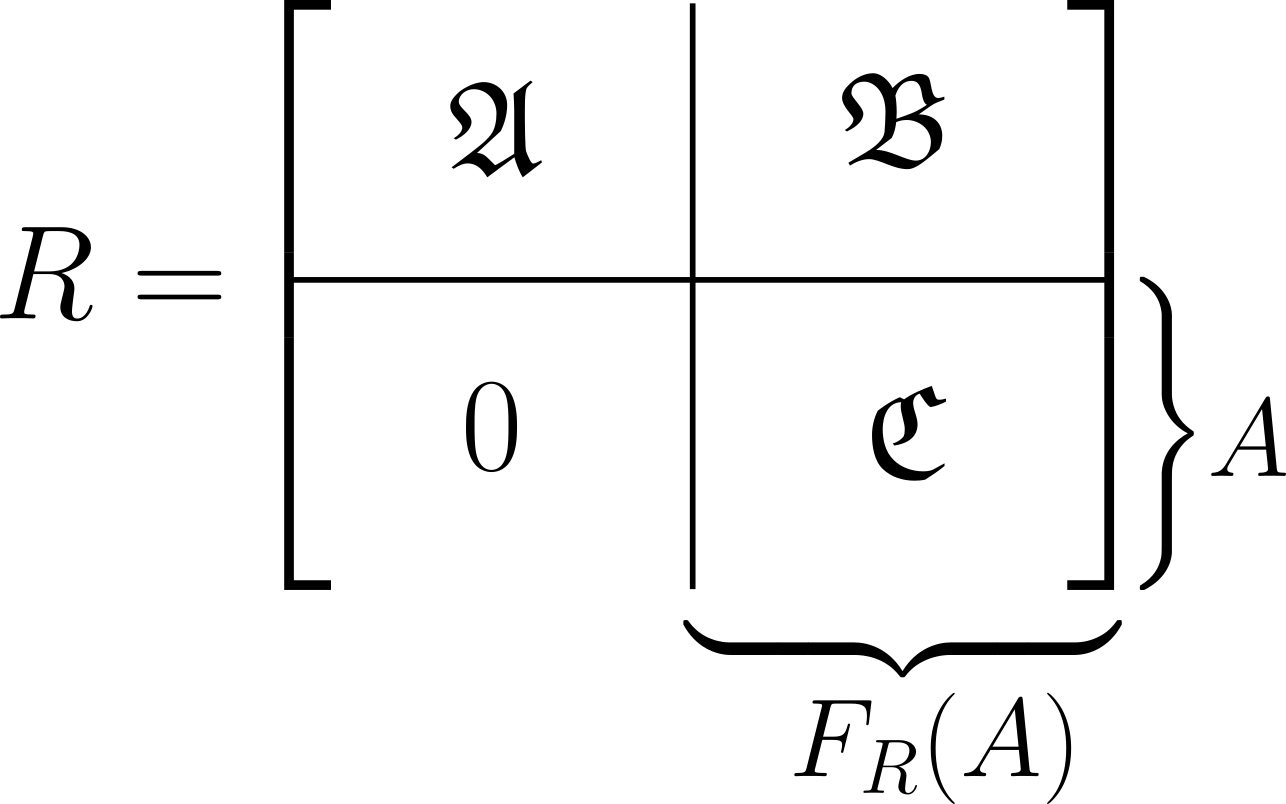}
	\caption{\label{fig:SISP_set} If $A\subset \DD$, up to reordering the lines, we can assume that it corresponds to the last lines. Up to reordering the columns, we can assume that $F_R(A)$ corresponds to the last columns. Then, $R$ has the form given in the picture. In this situation, $A$ is a SISP set for $(R; \mu,\nu)$ if $R^*$ cancels on the block $\mathfrak B$ and if the supports of $R$ and $R^*$ coincide on the block $\mathfrak C$.}
\end{figure}
 On the one hand, SISP sets always exist, at least under Assumption~\ref{ass:full_support}, as announced in the following lemma. Its proof is our main task in this part of our work, and is given at the end of the subsection.
\begin{lem}
	\label{lem:SISP_sets_exist}
	Let $R \in \mathcal M_+(\DD \times \FF)$, $\mu \in \mathcal M_+(\DD)$ and $\nu\in \mathcal M_+(\FF)$ satisfying Assumption~\ref{ass:full_support}. Then there exists a SISP set for $(R; \mu,\nu)$.
\end{lem}

On the other hand, once we know how to find SISP sets, an iterative procedure consisting in finding SISP sets for a sequence of more and more restricted problems makes is possible to reconstruct the whole subset $\mathcal S$.

\begin{prop}
	\label{prop:find_S}
	Let $R \in \mathcal M_+(\DD \times \FF)$, $\mu \in \mathcal M_+(\DD)$ and $\nu\in \mathcal M_+(\FF)$ satisfying Assumption~\ref{ass:full_support}. Let us call $\mathcal S$ the common support of $P^*$ and $Q^*$ from Theorem~\ref{thm:convergence_sinkhorn}, and $R^*$ from Theorem~\ref{thm:link_P*_Q*_R*}. 
	
	We define by inference $(R^n)_{n \in \N}$ a sequence in $\mathcal M_+(\DD \times \FF)$, $(\DD^n)_{n \in \N}$ a nonincreasing sequence of subsets of $\DD$ and $(\FF^n)_{n \in \N}$ a nonincreasing sequence of subsets of $\FF$ in the following way:
	\begin{itemize}
		\item For $n=0$, we set $R^0 := R$, $\DD^0 := \DD$ and $\FF^0 := \FF$;
		\item For all $n \in \N$, if $\DD^n$ and $\FF^n$ are nonempty and $(R^n\llcorner_{\DD^n \times \FF^n}; \mu\llcorner_{\DD^n}, \nu\llcorner_{\FF^n})$ satisfies Assumption~\ref{ass:full_support}, we pick $M_n$ a SISP set as given by Lemma~\ref{lem:SISP_sets_exist}, and we set:
		\begin{gather*}
		\DD^{n+1} := \DD^n \backslash M_n, \qquad \FF^{n + 1} := \FF^n \backslash F_{R^n\llcorner_{\DD^n \times \FF^n}}(M_n), \\[10pt] 
		\forall i,j, \quad R^{n+1}_{ij} := \left\{ 
		\begin{aligned}
		&0, &&\mbox{if } y_j \in F_{R^n\llcorner_{\DD^n \times \FF^n}}(M_n) \mbox{ and } x_i \in \DD^{n+1},\\[5pt]
		& R^n_{ij}, &&\mbox{otherwise}.
		\end{aligned} \right.
		\end{gather*}
		Otherwise, we set $R^{n+1} := R^n$, $\DD^{n+1} := \DD^n$ and $\FF^{n+1} := \FF^n$.
	\end{itemize}
	
	With this construction, the sequence $(R^n, \DD^n, \FF^n)_{n \in \N}$ is stationary. More precisely, there exists $N \in \N^*$ such that for all $n \geq N$,
	\begin{equation*}
	\DD^n = \emptyset, \qquad \FF^n = \emptyset, \qquad R^n = \1_{\mathcal S} R.
	\end{equation*}
\end{prop}
\begin{figure}
	\centering 
	\includegraphics[scale=.15]{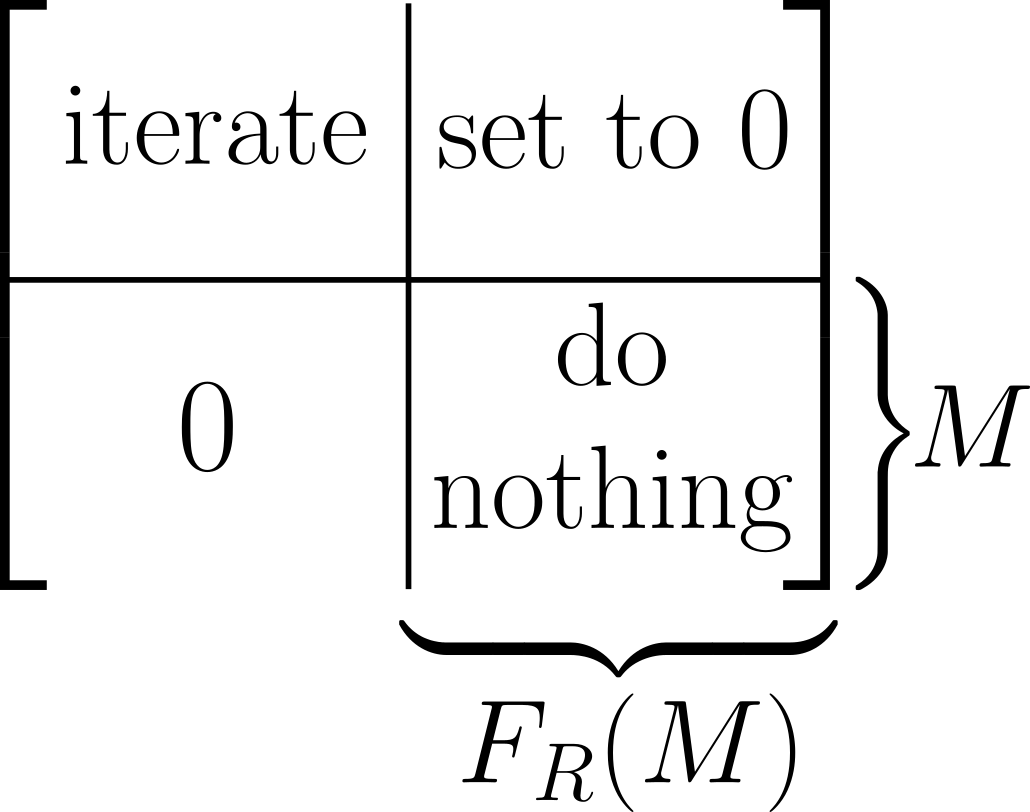}
	\caption{\label{fig:minimal_maximal_lambdaset} In the situation of~\cref{fig:SISP_set} where we have reordered the lines and columns, our procedure consists in recursively add zeros to $R$ at positions where we know thanks to~\eqref{eq:SISP_isolated} that $R^*$ admits a zero. We know that we did not forget any zero in $M \times F_R(M)$ thanks to~\eqref{eq:set_of_scalability}.}
\end{figure}

An illustration of the procedure at each iteration, is provided in \cref{fig:minimal_maximal_lambdaset}. An illustration of the full procedure in a specific non-scalable case is provided in \cref{fig:procedure}.

\begin{figure}
    \centering
	    \includegraphics[width=\textwidth]{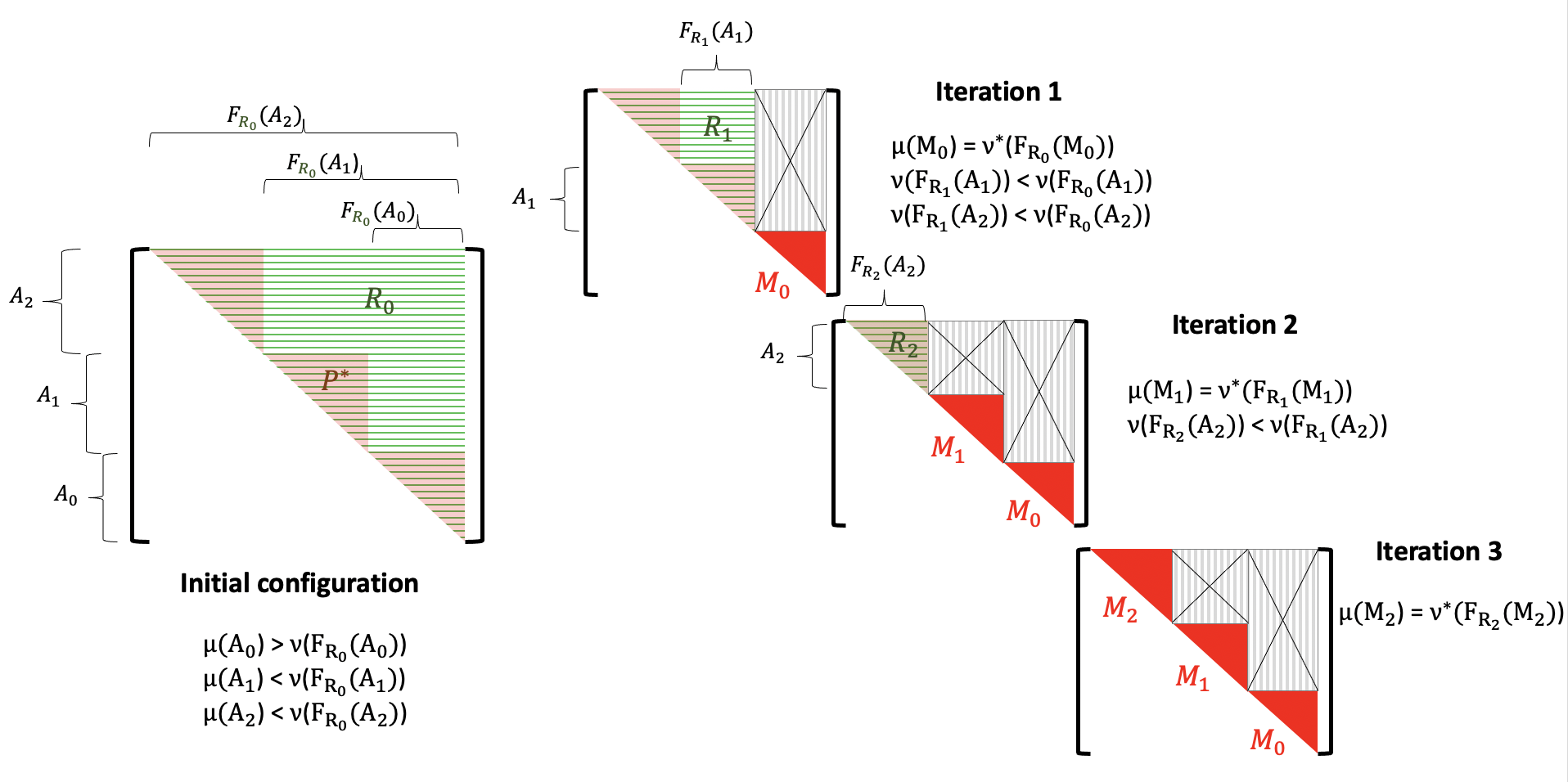}
    \caption{Illustration of the procedure of Proposition~\ref{prop:find_S} when the matrix $R$ is upper diagonal and $R^*$ a staircase matrix (see Appendix~\ref{app:no_solution} for more details). Only $A_0$ is a SISP set for $(R; \mu, \nu)$, but $A_1$ and $A_2$ are SISP sets for the restricted problems at the iterations 2 and 3. In this example, the procedure is stationary after 3 steps. In this example, the SISP set at each iteration is the unique maximal $\theta$-set existing for the reduced problem, as presented in Definition \ref{def:maximal_theta_set}. We remark that we can also build $\nu^*$, the second marginal of $P^*$ defined in Theorem~\ref{thm:convergence_sinkhorn}, on the successive SISP sets obtained along the procedure, thanks to the second step of the proof of Proposition \ref{lem:SISP_sets_exist} which ensures that the ratio $\frac{\nu}{\nu^*}$ is constant inside the maximal $\theta$-sets.}
    \label{fig:procedure}
\end{figure}

\begin{proof}
	In this proof, in order to lighten the notations, we call $R^n_r := R^n\llcorner_{\DD^n \times \FF^n}$. We will prove by inference the following facts. For all $n \in \N$:
	\begin{enumerate}
		\item Calling $\mathcal S^n$ the support of $R^n$, and therefore $\mathcal S^0$ the support of $R$,
		\begin{equation}
		\label{eq:Sn=S}\mathcal S^n \cap \Big( (\DD \times \FF) \backslash(\DD^n \times \FF^n)  \Big) = \mathcal S \cap \Big( (\DD \times \FF) \backslash(\DD^n \times \FF^n)  \Big) = \mathcal S^0 \cap \left( \bigcup_{k=0}^{n-1} M_k \times F_{R^k_r}(M_k) \right),
		\end{equation}
		 $\mathcal S \subset \mathcal S^n$, and $\mathcal S^n \cap (\DD^n \times \FF^n) = \mathcal S^0 \cap (\DD^n \times \FF^n)$.
		\item $\DD^n$ is empty if and only if $\FF^n$ is empty.
		\item If $\DD^n$ and $\FF^n$ are not empty, $(R^n_r; \mu\llcorner_{\DD^n}, \nu\llcorner_{\FF^n})$ satisfies Assumption~\ref{ass:full_support} and the matrices $P^{n,*}$, $Q^{n,*}$ and $R^{n,*}$ associated with $(R^n_r; \mu\llcorner_{\DD^n}, \nu\llcorner_{\FF^n})$ through Theorems~\ref{thm:convergence_sinkhorn} and~\ref{thm:link_P*_Q*_R*} are the restrictions of $P^*$, $Q^*$ and $R^*$ to $\DD^n \times \FF^n$.
	\end{enumerate}

	This is enough to prove the proposition: if the conclusion of the inference is true, then by the third point and Lemma~\ref{lem:SISP_sets_exist}, as long as $\DD^n$ and $\FF^n$ are nonempty, $(R^n_r; \mu\llcorner_{\DD^n}, \nu\llcorner_{\FF^n})$ admits a SISP set $M_n$, which is not empty by definition. Therefore, $(\DD^n)$ is strictly decreasing in the sense of inclusion as long as it is not empty, so it has to reach $\emptyset$ at a certain rank $N$. At this rank, because of the first point, we also have $\FF^N = \emptyset$, and because of~\eqref{eq:Sn=S}, $\mathcal S^N = \mathcal S$, so that the conclusion follows. So let us prove the inference.
	
	\bigskip
	
	At rank $0$, everything is clear, so let us assume that the conclusions of points one, two and three hold at rank $n$, and prove them at rank $n+1$. First, if $\DD^n$ is empty, by assumption $\FF^{n}$ is empty as well, so we have reached a stationary point, and everything is still true at rank $n+1$. So we can assume without loss of generality that $\DD^n$, and hence $\FF^n$, are nonempty. By assumption, $(R^n_r; \mu\llcorner_{\DD^n}, \nu\llcorner_{\FF^n})$ satisfies Assumption~\ref{ass:full_support}, and by Lemma~\ref{lem:SISP_sets_exist}, we can find a SISP set $M_n$. In this context, let us check the points one by one at rank $n+1$.
	
	\bigskip
	\noindent \underline{First point}. Observing that $\DD^n$ is the disjoint union of $M_n$ and $\DD^{n+1}$, and that $\FF^n$ is the disjoint union of $F_{R^n_r}(M_n)$ and $\FF^{n+1}$, we have
	\begin{align*}
 (\DD \times \FF) &\backslash(\DD^{n+1} \times \FF^{n+1})   \\
 &=  \Big( (\DD \times \FF) \backslash(\DD^n \times \FF^n) \Big) \cup \big( \DD^{n+1} \times F_{R^n_r}(M_n) \big) \cup \big( M_n \times \FF^{n+1} \big)  \cup \big( M_n \times F_{R^n_r}(M_n) \big). 
	\end{align*}
	So in order to prove~\eqref{eq:Sn=S} at rank $n+1$, we need to show that
	\begin{gather}
   \label{eq:nothing_about_before}\mathcal S^{n+1} \cap \Big( (\DD \times \FF) \backslash(\DD^n \times \FF^n) \Big) =  \mathcal S^{n} \cap \Big( (\DD \times \FF) \backslash(\DD^n \times \FF^n) \Big),\\
    \label{eq:old_zeroes}\mathcal S^{n+1} \cap \big( M_n \times \FF^{n+1} \big) = \mathcal S \cap \big( M_n \times \FF^{n+1} \big) = \emptyset,\\
    \label{eq:new_zeroes}\mathcal S^{n+1} \cap  \big( \DD^{n+1} \times F_{R^n_r}(M_n) \big) = \mathcal S \cap  \big( \DD^{n+1} \times F_{R^n_r}(M_n) \big) = \emptyset,\\
    \label{eq:set_of_scalability} \mathcal S^{n+1} \cap \big( M_n \times F_{R^n_r}(M_n) \big) = \mathcal S \cap \big( M_n \times F_{R^n_r}(M_n) \big) = \mathcal S^{0} \cap \big( M_n \times F_{R^n_r}(M_n) \big).
	\end{gather}
	To prove these equalities, the main tool is the following formula which is a direct consequence of the construction:
	\begin{equation}
	\label{eq:Sn+1}
	\mathcal S^{n+1} = \mathcal S^n \backslash \big( \DD^{n+1} \times F_{R^n_r}(M_n) \big).
	\end{equation}
	
	With this formula at hand, we see that~\eqref{eq:nothing_about_before} follows from $\DD^{n+1} \times F_{R^n_r}(M_n) \subset \DD^n \times \FF^n$. We also deduce very easily that $\mathcal S^{n+1} \cap (\DD^{n+1} \times \FF^{n+1}) = \mathcal S^n \cap (\DD^{n+1} \times \FF^{n+1}) = \mathcal S^0\cap(\DD^{n+1} \times \FF^{n+1})$, where the last equality follows from the first point at rank $n$.
	
	Then, to prove~\eqref{eq:old_zeroes}, as both $\mathcal S$ (by assumption) and $\mathcal S^{n+1}$ (by~\eqref{eq:Sn+1}) are included in~$\mathcal S^n$, it suffices to show that $\mathcal S^{n} \cap ( M_n \times \FF^{n+1} ) = \emptyset$. But that last assertion follows from the definition of $\FF^{n+1} = \FF^n \backslash F_{R^n_r}(M_n)$: these are precisely the columns where $R^n_r$ has only zero entries on the intersection with the lines $M_n$. 
	
	To prove~\eqref{eq:new_zeroes}, let us observe that the equality $\mathcal S^{n+1} \cap  ( \DD^{n+1} \times F_{R^n_r}(M_n) ) = \emptyset$ is a direct consequence of~\eqref{eq:Sn+1}. The other equality, namely, $\mathcal S \cap  ( \DD^{n+1} \times F_{R^n_r}(M_n) ) = \emptyset$ follows from the fact that $M_n$ is a SISP set for $(R^n_r; \mu\llcorner_{\DD^n}, \nu\llcorner_{\FF^n})$, so that~\eqref{eq:SISP_isolated} applies with $M_n$ instead of $A$, $R^n_r$ instead of $R$, $\DD^n$ instead of $\DD$ and $R^{n,*} = R^*\llcorner_{\DD^n \times \FF^n}$ instead of $R^*$ (here, we use the point three at rank $n$). Notice that as $\mathcal S \cap  ( \DD^{n+1} \times F_{R^n_r}(M_n) ) = \emptyset$ and $\mathcal S \subset \mathcal S^{n}$, by~\eqref{eq:Sn+1}, we have also proved that $\mathcal S \subset \mathcal S^{n+1}$.
	
	Finally, to prove~\eqref{eq:set_of_scalability}, as $\mathcal S \subset \mathcal S^{n+1} \subset \mathcal S^0$, we just need to prove that $\mathcal S \cap ( M_n \times F_{R^n_r}(M_n) ) = \mathcal S^{0} \cap ( M_n \times F_{R^n_r}(M_n) )$. But this is a direct consequence of the fact that $M_n$ is a SISP set for $(R^n_r; \mu\llcorner_{\DD^n}, \nu\llcorner_{\FF^n})$, so that~\eqref{eq:support_R_R*} applies with $R^n_r$ instead of $R$, $R^{n,*}$ instead of $R^*$, $M_n$ instead of $A$ and $\FF^n$ instead of $\FF$.
	
		\bigskip
	
	\noindent \underline{Second point}. By definition, $\DD^{n+1}$ is empty if and only if $M_n = \DD^n$. So if $\DD^{n+1} = \emptyset$, then by Assumption~\ref{ass:full_support} applied to $(R^n_r; \mu\llcorner_{\DD^n}, \nu\llcorner_{\FF^n})$, we clearly have $F^{R^n_r}(M_n) = \FF^n$ and so $\FF^{n+1} = \emptyset$. On the other hand, if $\DD^{n+1} \neq \emptyset$ we have 
	\begin{align*}
	0< \mu(\DD^{n+1}) &=  P^*\Big(\DD^{n+1} \times \FF \Big) =   P^*\Big(\DD^{n+1} \times \FF^n \Big) \\
	&=  P^*\Big(\DD^{n+1} \times F_{R^n_r}(M_n)\Big) + P^*\Big(\DD^{n+1} \times \FF^{n+1}\Big)\\
	&= P^*\Big(\DD^{n+1} \times \FF^{n+1}\Big),
	\end{align*}
	where the inequality comes from Assumption~\ref{ass:full_support}, the second equality is an easy consequence of~\eqref{eq:Sn=S} at rank $n$, and the last one comes from the definition of SISP sets (see~\eqref{eq:SISP_isolated}) and from the fact that $P^*$ and $R^*$ has the same support. So $\FF^{n+1}$ cannot be empty, as announced.	
	
	\bigskip
	
	\noindent\underline{Third point}. To check that $(R^{n+1}_r; \mu\llcorner_{\DD^{n+1}}, \nu\llcorner_{\FF^{n+1}})$ satisfies Assumption~\ref{ass:full_support}, we need only need to show that the support of $\mu^{R^{n+1}}\llcorner_{\DD^{n+1}}$ is $\DD^{n+1}$ and the support of $\nu^{R^{n+1}}\llcorner_{\FF^{n+1}}$ is $\FF^{n+1}$. But as we already proved that $\mathcal S \subset \mathcal S^{n+1}$, we know that $P^* \ll R^{n+1}$ and $Q^* \ll R^{n+1}$, so the conclusion follows from the fact that $\mu$ and $\nu$ have full support by Assumption~\ref{ass:full_support} applied to $(R; \mu,\nu)$. The last thing to check, namely that the matrices $P^{n+1,*}$, $Q^{n+1,*}$ and $R^{n+1,*}$ associated with $(R^{n+1}_r; \mu\llcorner_{\DD^{n+1}}, \nu\llcorner_{\FF^{n+1}})$ through Theorems~\ref{thm:convergence_sinkhorn} and~\ref{thm:link_P*_Q*_R*} are the restrictions of $P^*$, $Q^*$ and $R^*$ to $\DD^{n+1} \times \FF^{n+1}$ is a direct consequence of Proposition~\ref{prop:solution_restriction} below (that we wanted to separate to the rest of the proof because we will use it again later), and of~\eqref{eq:new_zeroes}.
	\end{proof}
\begin{prop}
	\label{prop:solution_restriction}
	Let $R \in \mathcal M_+(\DD \times \FF)$, $\mu \in \mathcal M_+(\DD)$ and $\nu\in \mathcal M_+(\FF)$ satisfy Assumption~\ref{ass:condition_sinkhorn_well_defined}. Let $P^*$, $Q^*$ and $R^*$ be the matrices associated with the problem $\Sch(R; \mu,\nu)$ by Theorems~\ref{thm:convergence_sinkhorn} and~\ref{thm:link_P*_Q*_R*}. Finally, let $A \subset \DD$ be such that
	\begin{equation}
	\label{eq:prop_isolated}
	R^*\Big( (\DD \backslash A) \times F_R(A) \Big) = 0.
	\end{equation}
	
	Then (with slightly sloppy notations), $P^*\llcorner_{A \times F_R(A)}$, $Q^*\llcorner_{A \times F_R(A)}$ and $R^*\llcorner_{A \times F_R(A)}$ are the matrices associated with the restricted problem $\Sch(R\llcorner_{A \times F_R(A)}, \mu \llcorner_{A}, \nu\llcorner_{F_R(A)})$ by Theorem~\ref{thm:convergence_sinkhorn} and~\ref{thm:link_P*_Q*_R*}.
	
	Similarly, calling $A' := \DD \backslash A$ and $F':= \FF \backslash F_R(A)$, $P^*\llcorner_{A' \times F'}$, $Q^*\llcorner_{A' \times F'}$ and $R^*\llcorner_{A' \times F'}$ are the matrices associated with the restricted problem $\Sch(R\llcorner_{A' \times F'}, \mu \llcorner_{A'}, \nu\llcorner_{F'})$ by Theorem~\ref{thm:convergence_sinkhorn} and~\ref{thm:link_P*_Q*_R*}.
\end{prop}
\begin{proof}
	We show the result in the case of $P^*$, related to the "block" $A \times F_R(A)$. The case of $Q^*$ is similar, the case of $R^*$ easily follows from the two previous ones, and the similar results on $A'\times F'$ follow the same lines. Let $\nu^*$ be defined by~\eqref{eq:def_mu*_nu*}. The first thing to prove is
	\begin{align*}
	 &\nu^*\llcorner_{F_R(A)}\\
	 &\hspace{5pt}=\argmin \Big\{ H(\bar \nu | \nu\llcorner_{F_R(A)})\, \Big| \, \bar \nu = \nu^P \mbox{ for some } P  \mbox{ with } H(P|R_{A \times F_R(A)}) < + \infty \mbox{ and } \mu^P = \mu\llcorner_{A} \Big\}.
	\end{align*}
	The measure $\nu^*\llcorner_{F_R(A)}$ is a competitor for the problem in the r.h.s.\ because it corresponds to $P := P^*\llcorner_{A \times F_R(A)}$. Let us show that it is the optimizer. To do this, we call $\bar \nu$ the optimizer, and we show that $\nu^*_{F_R(A)} = \bar \nu^*$. Let us consider $\bar P$ a $P$ corresponding to $\bar \nu^*$ in the problem above, $\bar P^*$ the matrix obtained by replacing the entries of $P^*$ on $A \times F_R(A)$ by the entries of $\bar P$, and $\bar \nu^* := \nu^{\bar P^*}$. We have
	\begin{equation*}
	H(\bar \nu^* | \nu ) = H(\bar \nu | \nu \llcorner_{F_R(A)}) + H(\nu^*\llcorner_{F'} | \nu\llcorner_{F'}) \leq H( \nu^*\llcorner_{F_R(A)} | \nu \llcorner_{F_R(A)}) + H(\nu^*\llcorner_{F'} | \nu\llcorner_{F'}) = H(\nu^* | \nu),
	\end{equation*}
	where the inequality, being a consequence of the optimality of $\bar \nu^*$, is an equality if and only if $\nu^*_{F_R(A)} = \bar \nu^*$. But by optimality of $\nu^*$ in~\eqref{eq:def_mu*_nu*} this inequality is indeed an equality, and therefore $\nu^*_{F_R(A)} = \bar \nu^*$.
	
	It remains to show that $P^*\llcorner_{A \times F_R(A)}$ is the solution of $\Sch(R\llcorner_{A \times F_R(A)}, \mu\llcorner_{A},\nu^*\llcorner_{F_R(A)})$. For this, let us consider $\bar P$ the solution of $\Sch(R\llcorner_{A \times F_R(A)}, \mu\llcorner_{A},\nu^*\llcorner_{F_R(A)})$, and $\bar P^*$ the matrix obtained by replacing the entries of $P^*$ on $A \times F_R(A)$ by the entries of $\bar P$. Because of~\eqref{eq:prop_isolated}, we have
	\begin{align*}
	H(\bar P^* | R) &= H(\bar P | R\llcorner_{A \times F_R(A)}) + H(P^*\llcorner_{A' \times F'} | R\llcorner_{A' \times F'}) \\
	&\leq H( P^*\llcorner_{A \times F_R(A)} | R\llcorner_{A \times F_R(A)}) + H(P^*\llcorner_{A' \times F'} | R\llcorner_{A' \times F'}) = H(P^* | R^*),
	\end{align*}
	where the inequality is a consequence of the optimality of $\bar P$, and is an equality if and only if $\bar P = P^*\llcorner_{A \times F_R(A)}$. But by optimality of $P^*$, this inequality is indeed an equality, and we conclude that $\bar P = P^*\llcorner_{A \times F_R(A)}$. The proposition is proved.
\end{proof}

\bigskip

Now, we want to prove Lemma~\ref{lem:SISP_sets_exist}. To do this, we introduce a new class of subsets of $\DD$, associated with a triple $(R;\mu,\nu)$.
\begin{defn}
\label{def:maximal_theta_set}
	Let $R \in \mathcal M_+(\DD \times \FF)$, $\mu \in \mathcal M_+(\DD)$ and $\nu\in \mathcal M_+(\FF)$ satisfying Assumption~\ref{ass:full_support}. The maximal $\theta$ associated to $(R; \mu,\nu)$ is defined by:
	\begin{equation}
	\theta_m := \max_{\substack{A \subset \DD \\ A \neq \emptyset}} \frac{\mu(A)}{\nu(F_{R}(A))}.
	\label{eq:lambdamax}
	\end{equation}
	We say that $A \subset \DD$ is a maximal $\theta$-set for $(R;\mu,\nu)$ if $A$ is a maximizer of~\eqref{eq:lambdamax}. We say that it is a \emph{smallest} maximal $\theta$-set if in addition, it is a minimal element in the sense of inclusion among all maximal $\theta$-sets associated with $(R; \mu,\nu)$.
\end{defn}

As maximal $\theta$-sets are optimizers of a finite function (thanks to Assumption~\ref{ass:full_support}) on a finite set (the set of all nonempty subsets of $\DD$), any triple $(R;\mu,\nu)$ satisfying Assumption~\ref{ass:full_support} admits at least one maximal $\theta$-set. The set of all maximal $\theta$-sets being itself finite, we know that there exists at least one minimal element in this set, so that smallest maximal $\theta$-sets always exist under Assumption~\ref{ass:full_support}. Hence, Lemma~\ref{lem:SISP_sets_exist} is an obvious consequence of the following proposition, whose proof heavily relies on the optimality conditions stated in Proposition~\ref{prop:optimality_conditions_mu*_nu*}.
\begin{prop}
	Let $R \in \mathcal M_+(\DD \times \FF)$, $\mu \in \mathcal M_+(\DD)$ and $\nu\in \mathcal M_+(\FF)$ satisfying Assumption~\ref{ass:full_support}. A smallest maximal $\theta$-set for $(R; \mu,\nu)$ is a SISP set for $(R; \mu,\nu)$.
\end{prop}
\begin{proof}
	Let $\mu^*$ and $\nu^*$ be defined by~\eqref{eq:def_mu*_nu*}. We first define the two following quantities
	\begin{equation*}
	\theta^\DD := \max_{i \in \DD} \frac{\mu_i}{\mu^*_i}, \qquad \theta^\FF := \max_{j \in \FF} \frac{\nu_j^*}{\nu_j}.
	\end{equation*}
	(Recall that under Assumption~\ref{ass:full_support}, by Remark~\ref{rem:swap_mu_mu*}, $\mu^*$ and $\nu^*$ defined by~\eqref{eq:def_mu*_nu*} have full support.) Then, we define the two following sets, that are nonempty subsets of $\DD$ and $\FF$ respectively:
	\begin{equation*}
	\overline M := \Big\{  x_i \in \DD \mbox{ s.t.\ } \frac{\mu_i}{\mu^*_i} = \theta^\DD \Big\}, \qquad \overline F := \Big\{  y_j \in \FF \mbox{ s.t.\ } \frac{\nu^*_j}{\nu_j} = \theta^\FF \Big\}.
	\end{equation*}
	The main argument of the proof consists in showing that $\theta^\DD $ and $ \theta^\FF$ coincide with $\theta_m$, the maximal $\theta$ for $(R; \mu,\nu)$. Even if $\overline M$ is not a smallest maximal-$\theta$ set in general (more precisely, it is a maximal $\theta$-set that is not minimal in general), we show at Step 3 below how this information allows to conclude. As before, $P^*$ is the matrix defined by~\eqref{eq:def_P*_Q*}.
	
	\bigskip
	
 \noindent \underline{Step 1}: $\theta^\DD = \theta^\FF$.

Let $x_i \in \overline M$ and $y_j \in \FF$ be such that $P^*_{ij}>0$ (such a $j$ exists thanks to Assumption~\ref{ass:full_support}). By the first line of~\eqref{eq:optimality_conditions_mu*_nu*}, we have
\begin{equation}
\label{eq:optimality_condition_proof_bartheta}
\frac{\mu^*_i}{\mu_i}\frac{\nu^*_j}{\nu_j} = 1,
\end{equation}
which implies that $\nu_j^* / \nu_j = \theta^\DD$, and hence that $\theta^\FF \geq \theta^\DD$. The other inequality is proved in the same way, and the result follows. From now on, we call
\begin{equation*}
\overline \theta := \theta^\DD = \theta^\FF.
\end{equation*}

\noindent \underline{Step 2}: $\overline \theta= \theta_m$.

First, $\overline \theta \geq \theta_m$. Indeed, for any $A\subset \DD$, by Theorem~\eqref{thm:CNS}, as $\Sch(R; \mu,\nu^*)$ is at least approximately scalable, we have $\mu(A) \leq \nu^*(F_R(A))$. But on the other hand, by definition of $\overline \theta$, we have $\nu^* \leq \overline \theta \nu$, so that actually, $\mu(A) \leq \overline \theta \nu(F_R(A))$, and hence $\overline \theta \geq \theta_m$.

Also, $\overline \theta \leq \theta_m$. To see this, let us first observe that $P^*((\DD \backslash \overline M) \times \overline F) = P^*( \overline M \times (\FF \backslash\overline F))=0$. This is because if $x_i,y_j$ are such that $P^*_{ij}>0$, still by~\eqref{eq:optimality_condition_proof_bartheta}, $\mu_i / \mu^*_i = \overline \theta$ if and only if $\nu^*_j / \nu_j = \overline \theta$, so that $x_i \in \overline M$ if and only if $y_j \in \overline F$. Therefore, on the one hand, $F_R(\overline M) \subset \overline F$, and on the other hand, projecting on both marginals:
\begin{align*}
\mu(\overline M) &= P^*(\overline M \times \FF ) = P^*(\overline M \times \overline F ) + \underbrace{ P^*( \overline M \times (\FF \backslash\overline F))}_{=0} \\
&= P^*(\overline M \times \overline F ) + \underbrace{P^*((\DD \backslash \overline M) \times \overline F)}_{=0} = P^*(\DD \times \overline F) = \nu^*(\overline F).
\end{align*}
As by definition of $\overline F$ and $\overline \theta$, $\nu^*(\overline F) = \overline \theta\nu(\overline F)$, we conclude that 
\begin{equation*}
\overline \theta \nu(F_R(\overline M)) \leq \overline \theta \nu(\overline F) = \mu(\overline M),
\end{equation*}
so that $\overline \theta \leq \theta_m$, as announced.

\bigskip

\noindent \underline{Step 3}: Conclusion.

We are now in position to conclude. Let $A$ be a smallest maximal $\theta$-set. As $\Sch(R; \mu,\nu^*)$ is at least approximately scalable, we know that $\mu(A) \geq \nu^*(F_R(A))$. On the other hand, as $A$ is a maximal $\theta$-set, we know that $\mu(A) = \theta_m \nu(F_R(A)) = \overline \theta \nu(F_R(A))$. But by definition of $\overline \theta$, we know that $\overline \theta \nu \geq \nu^*$, so that $\nu^*(F_R(A)) \leq \mu(A)$. We conclude that $\nu^*(F_R(A)) = \mu(A)$, that $\nu^*\llcorner_A = \overline \theta \nu\llcorner_A = \theta_m \nu\llcorner_A$, and hence that
\begin{equation*}
P^*\Big( (\DD \backslash A) \times F_R(A) \Big) = P^*(\DD \times F_R(A)) - P^*(A \times F_R(A)) = \nu^*(F_R(A)) - \mu(A) = 0,
\end{equation*}
so that~\eqref{eq:SISP_isolated} holds.

In addition, by Proposition~\ref{prop:solution_restriction} the measure $P^*\llcorner_{A \times F_R(A)}$ is the solution of the problem $\Sch(R\llcorner_{A \times F_R(A)}; \mu\llcorner_{A}, \nu^*\llcorner_{F_R(A)})$. So in order to prove~\eqref{eq:set_of_scalability}, it suffices to prove that this problem is scalable. For this purpose, we will use Lemma~\ref{lem:non_degenerate_case}. We call $R_r := R\llcorner_{A \times F_R(A)}$. Let $B$ be a nonempty strict subset of $A$. As $A$ is a minimal element in the set of maximal $\theta$-sets for $(R; \mu,\nu)$, we know that $\mu(B) < \theta_m \nu(F_R(B))$. Then, as $F_R(B) \subset F_R(A)$, we have $F_R(B) = F_{R_r}(B)$, so $\mu(B) < \theta_m \nu(F_{R_r}(B))$. Finally, as $\nu^*\llcorner_A = \theta_m \nu\llcorner_A$, we have $\mu(B) < \nu^*(F_{R_r}(B))$. So Lemma~\ref{lem:non_degenerate_case} applies, and $\Sch(R\llcorner_{A \times F_R(A)}; \mu\llcorner_{A}, \nu^*\llcorner_{F_R(A)})$ is scalable, which concludes the proof.
\end{proof}

We close this section with a remark concerning the stability with respect to union of SISP sets.
\begin{rem}
	It is easy to check that SISP sets associated with a triple $(R; \mu,\nu)$ are stable by union. Therefore, there exists an upper bound in the set of all SISP sets for $(R; \mu,\nu)$, that we call the \emph{largest} SISP set. If we want the procedure described in Proposition~\ref{prop:find_S} to be as fast as possible, it is logical to look for SISP sets that are as large as possible, in order to minimize the rank $N$ at which the procedure reaches its stationary point. This is what we are going to do in the next section.
\end{rem}

\section{Numerical applications}
\label{sec_numerical}
A simple consequence of the theoretical procedure described in the previous section is that in a lot of cases, if the problem $\Sch(R; \mu, \nu)$ is non-scalable, the matrices $P^*$ and $Q^*$ from Theorem~\ref{thm:convergence_sinkhorn}, and $R^*$ from Theorem~\ref{thm:link_P*_Q*_R*} have more zero entries than $R$. For instance, if the problem is balanced (\emph{i.e.}\ $\M(\mu) = \M(\nu)$), and if the bipartite graph of $R$ is connected (that is, $\mu(A) = \nu(F_R(A))$ only holds for $A = \emptyset$ or $A = \DD$, which is a reasonable assumption in a lot of contexts), we can check that the matrix $R^1$ from Proposition~\ref{prop:find_S} cannot coincide with $R$.

 Therefore, typically, the Sinkhorn algorithm in the non-scalable case does not converge linearly. We are going to detail an approximate algorithm allowing to find the common support of $P^*$, $Q^*$ and $R^*$, and therefore to recover a linear rate of convergence for the Sinkhorn algorithm by Proposition~\ref{prop:recover_linear_rate}.

\subsection{Stopping criterion}
Before any numerical application, we need to define a stopping criterion for the Sinkhorn algorithm when the Schr\"odinger problem is non-scalable. When the problem is scalable, the classical criterion that is used is the duality gap estimated at each step $n \in \N^*$ of the Sinkhorn algorithm:
\begin{align}
    \label{Classical_stopping_criterion}
{SC}^n = H(P^n | R) - \cg \log(a^{n}), \mu \cd - \cg \log(b^{n-1}), \nu \cd,
\end{align}
where $P^n$, $a^{n}$ and $b^{n-1}$ are defined by the relations~\eqref{eq:computable Sinkhorn}. Indeed, it is known that this quantity is always positive when the relation $P^n_{ij} = a^n_i b^{n-1}_j R_{ij}$ holds for all $i,j$, and the Fenchel-Rockafellar duality ensures that ${SC}_n \to 0$ as $n \to \infty$ when the problem is scalable, \textit{i.e.}\ when $(a^n)$ and $(b^n)$ converge. 

In the approximately scalable case, numerical instabilities may appear when $n \to \infty$ because $(a^n)$ and $(b^n)$ do not converge, but this criterion may remain useful if the error that is tolerated is not too small. However, in the non-scalable case, this criterion does not hold as the problem $\Sch(R;\mu,\nu)$ has no solution. The results presented in the previous sections allow nevertheless to define an approximate criterion. Indeed, it has been shown in~\cite{Chizat2018} that for a given $\lambda >0$, the problem defined with the notations of Subsection~\ref{subsec:notations} by
\begin{equation}
\label{schrodinger_unbalanced}
\textrm{Schu}_\lambda(R;\mu,\nu) = \min \Big\{ H(P | R) + \lambda H(\nu^P |\nu) \, \Big| \, P \mbox{ s.t.\ } \mu^P = \mu \Big\},
\end{equation}
can be solved numerically with a generalization of the Sinkhorn algorithm. More precisely, the duality gap defined for all $n \in \N^*$ by
\begin{equation}
    \label{Approximate_stopping_criterion}
{SCu}_\lambda^n = H(P^n | R) + \lambda\left(H(\nu^{P^n} | \nu) - \Big\cg 1 - (1/b^{n-1})^{\frac{1}{\lambda}} , \nu \Big\cd\right) - \cg \log(a^n), \mu \cd,
\end{equation}
converges to $0$ whenever $P^n_{ij} := a^n_i b^{n-1}_j R_{ij}$ for all $i,j$, and with $a^n, b^n$ defined for all $n \in \mathbb{N}^*$ by the relations:
\begin{equation}
\label{eq:modified_sinkhorn}
\left\{
\begin{gathered}
\forall j, \quad  b^0_j := 1,\\[10pt]
\begin{aligned}
&\forall n\geq 0, \quad \forall i,& a^{n+1}_i &:= \frac{\mu_i}{\displaystyle \sum_j b^n_j R_{ij}},\\
&\forall n\geq 0,\quad \forall j, & b^{n+1}_j &:= \Bigg( \frac{\nu_j}{\displaystyle\sum_i a^{n+1}_i R_{ij}} \Bigg)^{\frac{\lambda}{1 + \lambda}}.
\end{aligned}
\end{gathered}
\right.
\end{equation}

On the other hand, a slight modification of our $\Gamma$-convergence result of Proposition~\ref{prop:gamma_cv} asserts that $P^*$, from Theorem~\ref{thm:convergence_sinkhorn}, is the limit of the solution of~\eqref{schrodinger_unbalanced} as $\lambda \to + \infty$.
So if we now define $SCu_\lambda^n$ by the formula~\eqref{Approximate_stopping_criterion} where $P^n$, $a^n$ and $b^{n-1}$ are computed with the standard Sinkhorn algorithm~\eqref{eq:computable Sinkhorn}, instead of the modified one~\eqref{eq:modified_sinkhorn}, we conclude that for all $\eps>0$, there exists a threshold $\lambda^\eps$ such that for all $\lambda \geq \lambda^\eps$,
\begin{equation*}
    \limsup_{n \to + \infty} SCu_\lambda^n \leq \eps.
\end{equation*}

Therefore, the stopping criterion~\eqref{Approximate_stopping_criterion} can still be used for the sequence $(P^n)_{n \in \mathbb N}$ generated by the classical Sinkhorn algorithm~\eqref{eq:computable Sinkhorn}, as long as $\lambda^\eps$ is chosen to be sufficiently large w.r.t.\ $\eps$. In practice, we observe that taking $\lambda^\eps = \frac{1}{\eps}$ works well. This is what we are going to do in the following section, considering a level of error $\eps := 10^{-3}$.

\subsection{An approximate numerical method for constructing the support of $R^*$}

An interesting application of the theoretical procedures described in Subsection~\ref{subsec:procedure} is the construction of an approximate algorithm allowing the identification of the support $\mathcal S$ of $R^*$ w.r.t. $R$, when the problem $\Sch(R;\mu,\nu)$ is approximately scalable or non-scalable. Our motivation is double. First, the fact that knowing the support allows to recover a linear rate for the Sinkhorn algorithm, thanks to Proposition \ref{prop:recover_linear_rate}, suggests that if this algorithm is simple enough, the full procedure to obtain $R^*$ (preprocessing to know the $\mathcal S$ and then the Sinkhorn algorithm applied to $\Sch(\1_{\mathcal S}R;\mu,\nu)$ is excepted to be faster than the Sinkhorn algorithm applied directly to $\Sch(R;\mu,\nu)$. Second, even if we have shown that the Sinkhorn algorithm converges in the non-scalable case, the Schrödinger potentials appearing in the Sinkhorn procedure are likely to be too high to be computed numerically before the algorithm has converged: thus, finding the support before to run the Sinkhorn algorithm can be an advantage, as long as this preprocessing prevents this potential's explosion and even if it does not accelerate the procedure.\\

We recall that we identified, at each iteration $n$ of the theoretical procedure described in Proposition~\ref{prop:find_S}, a SISP set, denoted $M_n$, for a reduced problem $\Sch(R\llcorner_{\DD^{n} \times \FF^{n}}; \mu\llcorner_{\DD^n},\nu\llcorner_{\FF^n})$. Regarding the proof of Lemma \ref{lem:SISP_sets_exist}, a natural choice for $M_n$ is an union of smallest maximal $\theta$-sets for $\Sch(R\llcorner_{\DD^n \times \FF^n}; \mu\llcorner_{\DD^n},\nu\llcorner_{\FF^n})$, that we introduced in Definition \ref{def:maximal_theta_set}. The approximate procedure that we design for finding $M_n$ at each iteration consists in two steps:
\begin{enumerate}
    \item Find the largest set $M_n'$ (in the sens of inclusion) such that the supports of $R$ and $R^*$ coincide in $M_n' \times \FF$;
    \item Find $M_n \subset M_n'$, an union of smallest maximal $\theta$-set for $(R\llcorner_{M_n'\times F_R(M_n')}; \mu\llcorner_{M_n'}, \nu\llcorner_{F_R(M_n')})$.
\end{enumerate}

Indeed, finding $M_n'$ in the first step can be achieved with a slightly modified Sinkhorn procedure applied to $\Sch(R\llcorner_{\DD^n\times \FF^n}, \mu\llcorner_{\DD^n}, \nu\llcorner_{\FF^n})$. More precisely, we initialize $U = \DD^n$ and $V = \FF^n$, and at each step of this Sinkhorn procedure, for all $x_i \in \DD^n$, if there exists $y_j \in V$ such that the coupling obtained at this step is smaller than a given threshold, we set $U = U \backslash x_i$. Since Lemma \ref{lem:SISP_sets_exist} ensures that there exists at least one SISP set, $M_n'$ is not empty and thus we do not set all the lines to $0$ (given that the threshold is sufficiently small). For all $y_j \in V$, if $y_j \notin F_{R\llcorner_{\DD^{n} \times \FF^{n}}}(U)$, we set $V = V \backslash y_j$.

After several steps of this modified Sinkhorn procedure, the elements remaining in $U$ and $V$ coincide respectively with $M_n'$ and $F_{R\llcorner_{\DD^{n} \times \FF^{n}}}(M_n')$. Since the Sinkhorn algorithm restricted to $U \times V$ converges linearly (thanks to Proposition \ref{prop:recover_linear_rate}), we then rapidly observe the convergence of the procedure.

 \bigskip
 
 For detailing how to obtain $M_n$ at the second step, we need to introduce the notion of connected components:
 \begin{defn}
 Let $U \subset \DD$ and $V \subset \FF$. We say that $A \times B$ is a connected component of the graph $(U \cup V, \triangle)$ (using the notation of \ref{push forward subset}) if:
 \begin{itemize}
     \item $R\llcorner_{U \times V}(A \times B^c) = R\llcorner_{U \times V}(A^c \times B) = 0$;
     \item $(A \cup B,\triangle)$ is connected.
 \end{itemize}
 \end{defn}
 
Finding connected components for such undirected graph is a classical task in Graph theory, for which there exists ready-to-use algorithms \cite{hagberg2008exploring}.
The following proposition justifies that we can define at the second step of the procedure
\begin{align}
\label{rel:choice_second_step}
    M_n = \bigcup \left\{U_i \,\bigg|\, \frac{\mu(U_i)}{\nu(V_i)} = \max\limits_{j=1,\cdots,C}\frac{\mu(U_j)}{\nu(V_j)}\right\},
\end{align} 
where $U_1 \times V_1,\cdots,U_C \times V_C$ are the connected components of the graph $(M_n' \cup F_{R\llcorner_{\DD^{n} \times \FF^{n}}}(M_n'), \triangle)$. 

\begin{prop}
\label{prop:connected_components}
Let $R \in \mathcal M_+(\DD \times \FF)$, $\mu \in \mathcal M_+(\DD)$ and $\nu\in \mathcal M_+(\FF)$ satisfying Assumption~\ref{ass:full_support}, and let us denote $U$ the largest subset of $\DD$ (in the sens of inclusion) such that the supports of $R$ and $R^*$ coincide in $U \times \FF$, and $V = F_R(U)$. Let us denote $(U_1 \times V_1,\cdots,U_C \times V_C)$ the $C$ connected components of the graph $(U \cup V, \triangle)$. Then, for any $U_i \times V_i$ such that
\begin{align}
\label{rel:max_cc}
    \frac{\mu(U_i)}{\nu(V_i)} =  \max\limits_{j=1,\cdots,C}\frac{\mu(U_j)}{\nu(V_j)},
\end{align}
$U_i$ is a smallest maximal $\theta$-sets for $(R;\mu,\nu)$.
\end{prop}

As we have already shown that for all $n \leq 0$, Assumption~\ref{ass:full_support} holds for $(R\llcorner_{\DD^{n} \times \FF^{n}},\mu\llcorner_{\DD^n}, \nu\llcorner_{\FF^n})$, we can apply Proposition \ref{prop:connected_components} to this triple.

\begin{proof}
The main point consists in showing that for any smallest maximal $\theta$-set for $(R; \mu, \nu)$, that we denote $A$, $A \times F_R(A)$ is a connected component of $(U \cup V, \triangle)$. Indeed, assuming this claim, as there always exists such smallest maximal $\theta$-set, any connected component $U_i \times V_i$ maximizing \eqref{rel:max_cc}, $U_i$ is a maximal $\theta$-set: $U_i$ is then necessarily a smallest one because it cannot contain any other connected component, and then any smallest maximal $\theta$-set neither. We now prove the claim.

As $A$ is a SISP set (thanks to the proof of Proposition \ref{lem:SISP_sets_exist}), it is in $U$ and we have necessarily $R^*(A^c \times F_R(A)) = R^*(A \times F_R(A)^c) = 0$. The supports of $R$ and $R^*$ being the same in $U \times V$ by construction of $U$ and $V$, we have thus $R\llcorner_{U \times V}(A^c \times F_R(A)) = R\llcorner_{U \times V}(A \times F_R(A)^c) = 0$.
Moreover, $A \times F_R(A)$ cannot contain any connected component of $(U \cup V, \triangle)$: if it was the case, this connected component would characterize a maximal $\theta$-set for $(R; \mu, \nu)$, which would contradict the minimality of $A$. Thus, $A \times F_R(A)$ is necessarily a connected component of the graph $(U \cup V, \triangle)$.
\end{proof}

We provide in Algorithm~\ref{algo:preprocessing} the pseudo-code of this iterative method.

\begin{algorithm}
	\caption{Find the support $\textrm{Supp}$ of $R^*$: return $\textrm{Supp}$}
	\begin{algorithmic}
		\REQUIRE
		$\bullet$ A set of minimal factors: $\{m_i, \, i=1,\cdots,N\}$,\\		
		$\bullet$ A stopping criterion: $stop(a, b ,R, \mu, \nu)$.\\
		We set $A = \DD$, $B = \FF$, $\textrm{Supp} = \textrm{Support}(R)$.
		\WHILE {$A \neq \emptyset$}
		\STATE{$\bar R = R\llcorner_{A \times B}$, $\bar \mu = \mu\llcorner_A$, $\bar \nu = \nu \llcorner_B$}
		\STATE $b = \1_{B}$, $a = \1_A$
		\STATE $U = A$, $V = B$
		\WHILE {$stop(a, b ,\bar R, \bar \mu, \bar \nu) \neq 1$}
		\FOR{$x_i \in U$}
		\IF{$\sum\limits_{y_j \in V} \bar R_{ij}b_j = 0$}
		\STATE{$U = U \backslash \{x_i\}$}
		\ELSE
		\STATE{$a_i = \displaystyle{\frac{\bar \mu_i}{\sum\limits_{y_j \in V} \bar R_{ij}b_j}}$}
		\IF{$\min\limits_{y_j \in V} a_i \times b_j < m_i$}
		\STATE{$U = U \backslash \{x_i\}$}
		\ENDIF
		\ENDIF
		\ENDFOR
		
		\STATE{$\bar R= \bar R\llcorner_{U \times V}$, $\bar \mu = \bar \mu\llcorner_{U}$, $a = a \llcorner_{U}$}
		
		\FOR{$y_j \in V$}
		\IF{$ \sum\limits_{x_i \in A} \bar R_{ij}a_i = 0$}
		\STATE{$V = V \backslash \{y_j\}$}
		\ELSE
		\STATE{$b_j =\displaystyle{ \frac{\bar \nu_j}{\sum\limits_{x_i \in A} \bar R_{ij}a_i}}$}
		\ENDIF
		\ENDFOR
		
		\STATE{$\bar R= \bar R\llcorner_{U \times V}$, $\bar \nu = \bar \nu\llcorner_{V}$, $b= b \llcorner_{V}$}

		\ENDWHILE
        \STATE{$(U_1 \times V_1,\cdots, U_C \times V_C) = \textrm{connected components of the graph}\, (U \cup V, \triangle)$}
        \STATE{$U \times V = \bigcup \big\{U_i \times V_i,\,|\, \frac{\mu(U_i)}{\nu(V_i)} = \max\limits_{j=1,\cdots,C}\frac{\mu(U_j)}{\nu(V_j)}\big\}$}
	
		\STATE $A = A \backslash U$, $B = B \backslash V$
		\STATE $\textrm{Supp} = \textrm{Supp}\backslash \big((A \backslash U) \times V \big)$
		\ENDWHILE
	\end{algorithmic}
	\label{algo:preprocessing}
\end{algorithm}

\bigskip

Let us make a few comment on this algorithm. The support of $R^*$ only depends on the support of $R$ and not of its values, so when identifying $\mathcal S$, we can equivalently consider the problem $\Sch(R';\mu,\nu)$, where $R' = \1_{R \neq 0}$. This explains why we consider the minimum $\min\limits_{y_j \in V} a_i \times b_j$ rather than $\min\limits_{y_j \in V} a_i \times b_j \times R_{ij}$ in Algorithm~\ref{algo:preprocessing}.

 The stopping criterion corresponds to the criterion~\eqref{Approximate_stopping_criterion} detailed in the previous section, which has to be smaller than a certain threshold $\varepsilon$ to be satisfied. 
 
 Choosing in an appropriate way the set of minimal factors $\{m_i, \, i=1,\cdots,N\}$, is crucial: it determines the level of approximation that is considered as acceptable, \textit{i.e.}\ the minimal value at which we can consider that the algorithm should create a new zero entries. In practice, we observe that
 \begin{align}
 m_i :=\frac{1}{n}\frac{\mu_i}{\mu_i^R},
 \label{thresold}
 \end{align}
 seems to be a good tradeoff between efficiency and security in most of the cases that we explored.\\

\begin{rem}
\begin{itemize}
 \item This method can be seen as an improvement over the naive approximate method which consists, at each iteration of the Sinkhorn algorithm applied to $\Sch(R;\mu,\nu)$, to set all the entries of $R$ that are smaller to a certain threshold to zero. With our method, we do not have to identify all the zero entries one by one but line by line, which can avoid numerical errors which may appear for some entries converging slowly to zero. However, this is done at the cost of identifying the connected components of some subgraphs of $(\DD \cup \FF,\triangle)$ at every iteration of the procedure, which slows down the algorithm in cases where these subgraphs are large. Note that in typical cases where the matrix $R$ is structured, we do not expect to find more than one connected component at each iteration as in \cref{fig:procedure}.
\item We emphasize the fact that this algorithm is only approximate, and this for two reasons. The first one occurs when the set of thresholds $\{m_i, \, i = 1, \dots, N\}$ are too large. Then, we can set to zero lines which should not be, just because some of the entries of $R^*$ should be small on this line. This must be avoided as then the algorithm cannot converge towards~$R^*$. 
    
    The other case where our algorithm does not identify $\mathcal S$ exactly is either when the threshold $\eps$ of the stopping criterion is large, or when the thresholds $\{m_i, \, i = 1, \dots, N\}$ are small. Then, the Sinkhorn algorithm can satisfy the stopping criterion before all the zeroes have been identified. This is not a big problem, since it means that the algorithm converges well without having to identify the additional zeroes of~$R^*$.
    
    With these observations, we conclude that the thresholds $\{m_i, \, i = 1, \dots, N\}$ need to be taken rather small w.r.t.\ the threshold $\eps$ of the stopping criterion, even though of course, if they are taken too small, efficiency is lost since then the algorithm just behaves as Sinkhorn without any improvement.
\end{itemize}
\end{rem} 

\bigskip
 
 We illustrate in~\cref{fig:preprocessing} the efficiency of this procedure. As emphasized at the beginning of this section, the number of iterations of a Sinkhorn-like algorithm is a better indicator than the computation time of the full procedure because our main motivation to find the support of $R^*$ before running the Sinkhorn algorithm is that we want the Schrödinger potentials not to explode before reaching convergence, which could be the case for non-scalable problems. 
 We thus represent the number of iterations needed for the Sinkhorn algorithm to converge as a function of the number of additional zero entries in $R^*$ w.r.t.\ to $R$, and compare when we apply or not the preprocessing described in Algorithm~\ref{algo:preprocessing}. For varying the number of zero entries, we take $R$ upper-diagonal, build $\mu$ and $\nu$ similarly to what we described in~\cref{fig:procedure}, and then vary the number of blocks from $1$ (corresponding to the scalable case) to $10$. For the case with preprocessing, we consider the sum of the iterations needed for the Sinkhorn-like method described in Algorithm \ref{algo:preprocessing} to find the support $\mathcal S$, and of the ones needed for the Sinkhorn algorithm then applied to the problem $\Sch(\1_{\mathcal S}R;\mu,\nu)$. We observe that the preprocessing makes the number of iterations needed for the convergence to be significantly smaller than for the case without preprocessing when the number of additional zero entries is high. It is also smaller than for the case when the naive approximate method is applied, illustrating the benefit of our approach.

\begin{figure}
    \centering
	    \includegraphics[width=.9\textwidth]{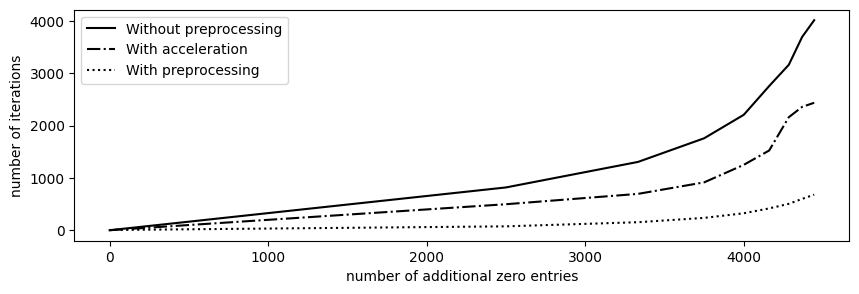}
    \caption{Number of iterations needed for convergence \emph{vs.}\ number of additional zero entries of the limits $P^*$ and $Q^*$ from Theorem~\ref{thm:convergence_sinkhorn} w.r.t.\ $R$, using: the Sinkhorn algorithm~\eqref{eq:computable Sinkhorn} (solid line); the Sinkhorn algorithm where at each step, the entries below the threshold~\eqref{thresold} are set to zero (dashed line); Algorithm~\ref{algo:preprocessing} to compute the support $\mathcal S$ of the limits, and then the Sinkhorn algorithm replacing $R$ by $\1_{\mathcal S} R$ (dotted line). When the threshold $\eps$ is small enough, as it is the case here, these three methods provide the same limits.}
    \label{fig:preprocessing}
\end{figure}

\subsection{Comparison of the method with the balanced and unbalanced Sinkhorn algorithms}
\label{numerical_results}

We compared in~\cref{fig:comparisons} the outputs of the Sinkhorn algorithm when the problem $\Sch(R;\mu,\nu)$ is non-scalable, given by the geometric mean described in Theorem~\ref{thm:link_P*_Q*_R*}, and two alternatives:
\begin{itemize}
    \item When the reference coupling $R$ is modified such that it has only positive entries. For that, we built a new coupling $R_\eps$ by adding on every zero entry of $R$ a small quantity $\eps$. We then found the optimizers $R^*_\eps$ of the Schr\"odinger problems $\Sch(R_\eps;\mu,\nu)$ and compared its distance in total variation of its solution to $R^*$, for different values of $\eps$.
    \item When we the marginal constraints are replaced by marginal penalizations, leading to an unbalanced problem of the form \eqref{eq:unbalanced_problem_lambda},
    using the scaling algorithm described in~\cite{Chizat2018}. We then compared the distance in total variation of its solution to $R^*$, for different values of~$\lambda$.
\end{itemize}
For the comparison realized here, we took a coupling $R$ of size $100 \times 100$ and $R, \mu, \nu$ as in~\cref{fig:procedure} in such a way that $R^*$ has only two blocks $A_1 \times B_1$ and $A_2 \times B_2$ for which the factor $\lambda$ appearing in the procedure of Proposition~\ref{prop:find_S} is greater than one for the first component, and smaller for the second one. The problem is thus non-scalable. As expected, we observe in \cref{comp_balanced} that in the first case it is impossible for the solution $R^*_\varepsilon$ of $\Sch(R_\eps;\mu,\nu)$ to be close to $R^*$ (and then to recover the right minimum entropy), and that the faster is the convergence, the further $R^*_\varepsilon$ is from $R^*$. In the second case, we observe in \cref{comp_unbalanced} that the solution $R^*_\lambda$ of the unbalanced problem with penalization $\lambda$ converges to $R^*$ when $\lambda \to \infty$. However, the convergence goes faster only for values of $\lambda$ smaller than $150$, for which we still observe a significant difference between $R^*$ and $R^*_\lambda$.

Note that these results are not limited to Schrödinger problems similar to the one described in~\cref{fig:procedure}, and that we observed the same type of results for randomly-generated $R$, $\mu$ and $\nu$ in non-scalable cases.

\begin{figure}
\begin{subfigure}{\textwidth}
    \centering
	    \includegraphics[width=.9\textwidth]{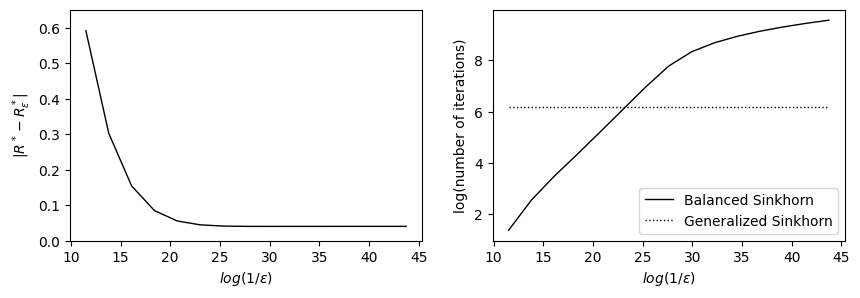}
	  \caption{}
	  \label{comp_balanced}
\end{subfigure}
\begin{subfigure}{\textwidth}
    \centering
	    \includegraphics[width=.9\textwidth]{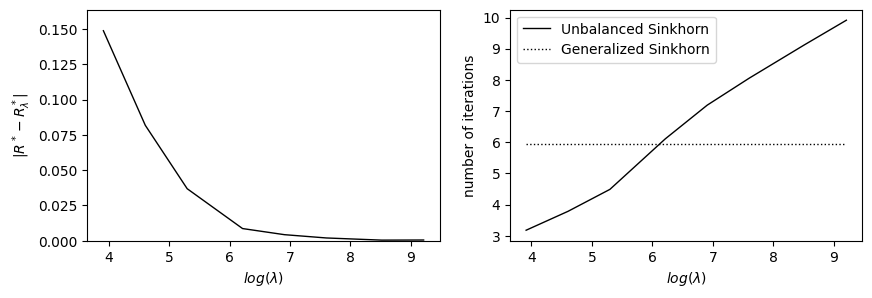}
	    \caption{}
	    \label{comp_unbalanced}
\end{subfigure}
    \caption{Comparison of the outputs of the Sinkhorn algorithm in a non-scalable case, where $R^*$ is given by Theorem~\ref{thm:link_P*_Q*_R*}, and:~\ref{comp_balanced} the Sinkhorn algorithm~\eqref{eq:computable Sinkhorn} where the zero entries of $R$ are replaced by a small value $\eps$;~\ref{comp_unbalanced} the unbalanced Sinkhorn algorithm from~\cite{Chizat2018} applied to solve $\textrm{Schu}_\lambda(R;\mu,\nu)$ for large values of $\lambda$.}
    \label{fig:comparisons}
\end{figure}

\section*{Acknowledgments}
This work was supported by funding from French agency ANR (SingleStatOmics; ANR-18-CE45-0023-03). We would like to thank Olivier Gandrillon, Thibault Espinasse and Thomas Lepoutre for enlightening discussions, and the latter and Gauthier Clerc for critical reading of the manuscript. We finally thank the BioSyL Federation, the LabEx Ecofect (ANR-11-LABX-0048) and the LabEx Milyon of the University of Lyon for inspiring scientific events.

\begin{appendices}

\section{Example of Schr\"odinger problems without solutions}
\label{app:no_solution}

There exists a lot of degenerate cases where the problem $\Sch(R;\mu,\nu)$ has no solution. Indeed, in the extreme situation where most of the entries of $R$ cancel, two randomly chosen vectors $\mu$ and $\nu$ have more chance to be non-scalable than to satisfy the conditions of Theorem~\ref{thm:CNS}. For example, in the typical example of a squared diagonal reference coupling $R$, we must necessarily have $\mu = \nu$ for these conditions to be satisfied.

When illustrating our results at Sections~\ref{sec:existence_of_a_solution} and~\ref{sec_numerical}, we chose $R$ to be a squared upper-diagonal matrix (see~\cref{fig:procedure}). This is of particular interest, as it corresponds to a case that typically arises when considering entropy minimization problems in cell biology. Indeed, the dynamics of mRNA levels within a cell, which drives cellular differentiation processes, is often modeled by a piecewise deterministic Markov process, where stochastic bursts of mRNAs compensate their deterministic degradation~\cite{Ventre2021, ventre2022one}. Considering the simplest cartoonish but enlightening situation where there is no degradation, a constant number of cells, and where we measure the activity of only one gene, the quantity of mRNAs in the cells corresponding to this gene can only increase with time. Therefore, if $R$ is the matrix whose entry $R_{ij}$ gives the number of cells having $i$ molecules of mRNA at a first timepoint and $j$ molecules at a later timepoint, $R$ must be upper-diagonal.

To give an insight of the behaviour of the Sinkhorn algorithm in the non-scalable case with an upper-diagonal reference matrix, let us treat explicitly a simple example. We consider:
\begin{equation*}
R =\begin{pmatrix}
1 & 1 & 1\\
0 & 1 & 1\\
0 & 0 & 1
\end{pmatrix}, \quad  \mu = (2, 2, 2),\quad \nu = (2, 3, 1).
\end{equation*}
In this example, $\mu_3 > \nu_3$ while the image of $x_3$ by the graph associated to $R$ is reduced to $y_3$, that is, with the formalism of Section~\ref{sec:existence_of_a_solution}, $F_R(\{x_3\}) = \{y_3\}$ and hence $\nu(F_R(\{x_3\}) < \mu(\{x_3\})$. In view of Theorem~\ref{thm:CNS} (which is very easy to check in our simple situation), the problem is therefore indeed non-scalable: no matrix can satisfy the marginal constraints and be absolutely continuous w.r.t $R$ at the same time. 

With the notations of~\eqref{eq:computable Sinkhorn}, let us reproduce below the output of the Sinkhorn algorithm at some of the first iterations. Starting at Iteration 5, we only give approximate numerical values.

\noindent \underline{Iteration 1}:
\begin{equation*}
a^1 = (2/3, 1, 2),\quad b^0 = (1, 1, 1),\quad P^1 =\begin{pmatrix}
2/3 & 2/3 & 2/3\\
0 & 1 & 1\\
0 & 0 & 2
\end{pmatrix}.
\end{equation*}
\noindent \underline{Iteration 2}:
\begin{equation*}
 a^1 = (2/3, 1, 2),\quad b^1 = (3, 9/5, 3/11), \quad Q^1 =\begin{pmatrix}
2 & 6/5 & 2/11\\
0 & 9/5 & 3/11\\
0 & 0 & 6/11
\end{pmatrix}.
\end{equation*}

\noindent \underline{Iteration 5}:
\begin{equation*}
 a^3 = (2.7e^{-1},8.6e^{-1}, 1.7e^1),\quad b^2 = (5.0, 2.2, 1.2e^{-1}), \quad P^3 =\begin{pmatrix}
1.4 & 0.59 & 3.2e^{-2}\\
0 & 1.9 & 1.0e^{-1}\\
0 & 0 & 2.0
\end{pmatrix}.
\end{equation*}

\noindent \underline{Iteration 11}:
\begin{equation*}
 a^6 = (1.2e^{-1}, 5.1e^{-1}, 1.5e^2),\quad b^5 = (1.3e^1, 3.9, 1.3e^{-2}), \quad P^6 =\begin{pmatrix}
1.55 & 0.45& 1.5e^{-3}\\
0 & 2.0 & 6.6e^{-3}\\
0 & 0 & 2.0
\end{pmatrix}.
\end{equation*}

\noindent \underline{Iteration 80}:
\begin{equation*}
 a^{40} =   (5.5e^{-5}, 2.8e^{-4}, 2.7e^{12}),\quad b^{40} = (3.6e^4, 9.1e^3, 3.8e^{-13}),\quad Q^{40} =\begin{pmatrix}
2.0 & 0.50 & 1.7 e^{-17} \\
0 & 2.5 & 8.4 e^{-16} \\
0 & 0 & 1.0
\end{pmatrix}.
\end{equation*}

\noindent \underline{Iteration 81}:
\begin{equation*}
 a^{41} = (4.4e^{-5}, 2.2e^{-4}, 5.3e^{12}), \quad b^{40} = (3.6e^4, 9.1e^3, 3.8e^{-13}),\quad P^{40}= 
 \begin{pmatrix}
1.6 & 0.40 & 4.2e^{-17}\\
0 & 2.0 & 2.1e^{-16}\\
0 & 0 & 2.0
\end{pmatrix}.
\end{equation*}

Of course, in this case, the matrices $P^*$, $Q^*$, $R^*$ from Theorem~\ref{thm:convergence_sinkhorn},~\ref{thm:link_P*_Q*_R*} are given by
\begin{equation*}
 P^* = \begin{pmatrix}
 8/5&2/5&0 \\ 0&2&0 \\ 0&0&2
 \end{pmatrix},\quad
  Q^* =  \begin{pmatrix}
 2&1/2&0 \\ 0&5/2&0 \\ 0&0&1
 \end{pmatrix},\quad
 R^* = \begin{pmatrix}
 4 / \sqrt 5  &1/\sqrt 5&0 \\ 0 &\sqrt 5&0 \\ 0&0&\sqrt 2
 \end{pmatrix}.
\end{equation*}
Finally, to get $\bar R{}^*$ from Theorem~\ref{thm:balanced}, it suffices to normalize $R^*$. 

\bigskip

This very simple example illustrates the different points developed in this article:
\begin{itemize}
\item When $R$ does not have only positive entries, the limits of the sequences $(P^n)$ and $(Q^n)$ given by the Sinkhorn algorithm may be different and have more zero entries than $R$;
\item Because new zero entries appear, the potentials $(a^n)$ and $(b^n)$ that are updated at each iteration of the Sinkhorn algorithm cannot converge: some of their coordinates have to tend to $0$ and then some other ones need to diverge to $+\infty$ as the number of iterations increases;
\item More precisely, for $(i,j)$ on the common support of $P^*$ and $Q^*$ from Theorem~\ref{thm:convergence_sinkhorn}, the infinitely small and high values of the two potentials are compensated. For $(i,j)$ outside of this common support, but still in the one of $R$, the multiplication of the two potentials generate infinitely small values. Outside of the support of $R$, the multiplication of the potentials can diverge. Also, the zero entries of $R$ prevent the sums involved in the computations of $(a^n)$ and $(b^n)$ to diverge: the large values of the potentials are sent to zero in the multiplication with $R$;
\item When the problem is non-scalable, the algorithm still converges to two limits and the algorithm alternates between them. These two limits correspond to solutions of the Schr\"odinger problem with modified marginals, that is with modified $\mu$ or modified $\nu$ alternatively (see the iterations 80 and 81).
\end{itemize}

Going back to the context of the beginning of the section, where the upper-diagonal $R$ models the evolution of the quantity of mRNAs corresponding to one gene in a population of cells between two timepoints, we see that the non-scalable case appears when there exists a threshold such that more cells with less mRNAs than the threshold are measured at the second timepoint than at the first one, which is incompatible with the model where the quantity of mRNAs can only increase. If we believe enough in our model, it is natural to look for a solution with modified marginals -- like for instance the law $\bar R^*$ described in Theorem~\ref{thm:balanced} -- and to advocate for a bad sampling or imprecise measurements when collecting data. 

If we consider that such incompatibilities between the theoretical model $R$ and the observations $\mu$ and $\nu$ should be rare, this law $\bar R^*$ is a natural choice since among all the solutions of a Schr\"odinger problem w.r.t.\ $R$, it is the one whose marginals are the closest (in a specific entropic sense) to the experimental ones.

\section{Proof of Theorem~\ref{thm:CNS}}
\label{proof of theorem CNS}
In this section, we prove Lemma~\ref{lem:non_degenerate_case} and then Theorem~\ref{thm:CNS}. The following classical proposition will be used in the proof of Lemma~\ref{lem:non_degenerate_case}:
\begin{prop}
\label{prop:R*_has_biggest_support}
Let $R \in \mathcal M_+(\DD \times \FF)$, $\mu \in \mathcal M_+(\DD)$ and $\nu\in \mathcal M_+(\FF)$ be such that the problem $\Sch(R;\mu,\nu)$ admits a solution $R^*$. Then for all $P \in \Pi(\mu, \nu)$ such that $P \ll R$, we necessarily have $P \ll R^*$.
\end{prop}

\begin{proof}[Proof of Proposition~\ref{prop:R*_has_biggest_support}]
Let $P \in \Pi(\mu, \nu)$ be such that $P \ll R$. For all $\eps >0$, $P^\eps := (1-\eps)R^* + \eps P$ is a competitor for $\Sch(R; \mu, \nu)$, so that by minimality of $R^*$, $H(P^\eps | R) \geq H(R^* | R)$. But because $s \mapsto s \log s$ is decreasing with infinite slope near $s = 0$, this is possible for $\eps$ small only if $P \ll R^*$.
\end{proof}

 Let us now prove Lemma~\ref{lem:non_degenerate_case}.
\begin{proof}[Proof of Lemma~\ref{lem:non_degenerate_case}]
In Lemma~\ref{lem:non_degenerate_case}, we are looking for necessary and sufficient conditions for $\Sch(R; \mu,\nu)$ to be scalable. As $\M(\mu) = \M(\nu)$ is clearly a necessary condition because of Remark~\ref{rem:coupling_total_mass}, we assume once for all that it is true. Up to normalizing, we assume that $\mu \in \mathcal P(\DD)$ and $\nu \in \mathcal P(\FF)$.

In order to clarify what~\ref{item:cond_A'} and~\ref{item:cond_B'} mean, we start by considering the case where the reference matrix $R$ is such that the graph $(\DD \cup \FF, \triangle)$ is connected. In that case, recalling that $\mu$ and $\nu$ are assumed to be probability measures, the conditions~\ref{item:cond_A'} and~\ref{item:cond_B'} are equivalent to:
	\begin{enumerate}[label={(\alph*'')}, leftmargin=.1\textwidth]
		\item The measures $\mu$ and $\nu$ have full support, and for all $\emptyset \subsetneq A \subsetneq \DD$, $\mu(A) < \nu(F_R(A))$,
		\label{item:cond_A''}
		\item The measures $\mu$ and $\nu$ have full support, and for all $\emptyset \subsetneq B \subsetneq \FF$, $\nu(B) < \mu(D_R(B))$.
		\label{item:cond_B''}
	\end{enumerate} 
Indeed, it is easy to see that in the balanced and connected case, the only subsets $A$ of $\DD$ for which $\mu^R(A) = \nu^R(F_R(A))$ are $A = \emptyset$ and $A = \DD$. Similarly, the only subsets $B$ of $\FF$ for which $\nu^R(B) = \mu^R(D_R(B))$ are $B = \emptyset$ and $B = \FF$.  \\

We are going to prove Lemma~\ref{lem:non_degenerate_case} under this connectivity assumption. Passing to the general case is direct, up to restricting ourselves to connected components of $(\DD \cup \FF, \triangle)$. 

We only prove~\ref{item:cond_B''}$\Leftrightarrow$\ref{item:existence_solution'}, as~\ref{item:cond_A''}$\Leftrightarrow$\ref{item:existence_solution'} is proved in the same way. The idea of the proof is to fix $\mu \in \P(\DD)$ of full support, that is, such that for all $i=1, \dots, N$, $\mu_i>0$, and to introduce the two following subsets of $\P(\FF)$:
\begin{gather*}
\mathfrak A := \{ \nu \in \P(\FF) \, | \, \forall \emptyset \subsetneq B \subsetneq \DD, \, \nu(B) < \mu (D_R(B)) \mbox{ and } \forall j =1,\dots, M, \, \nu_j>0 \},\\
\mathfrak B := \{ \nu := \nu^{\bar R} \, | \, \bar R \sim R \mbox{ and } \mu^{\bar R}= \mu \}.
 \end{gather*}
 
 With these definitions, proving \ref{item:cond_B''}$\Leftrightarrow$\ref{item:existence_solution'} exactly means proving
 \begin{equation*}
      \mathfrak A = \mathfrak B,
 \end{equation*}
 and this is what we will prove now. To do so, we will first show that $\mathfrak B \subset \mathfrak A$, and then that $\mathfrak B$ is open and closed in $\mathfrak A$. As $\mathfrak A$ is convex, and hence connected, the result will follow.
 
 \bigskip
 
 \noindent \underline{Step 1}: $\mathfrak B \subset \mathfrak A$.

 Let us show that $\mathfrak B \subset \mathfrak A$. To this end, let us consider $\nu \in \mathfrak B$, and $\bar R$ such that $R \in \Pi(\mu, \nu)$ and $\bar R \sim R$. For all $\emptyset \subsetneq B \subsetneq \DD$ we have:
 \begin{equation*}
 \nu(B) = \sum\limits_{x_i \in \DD} \sum_{y_j \in B} \bar R_{ij} = \sum_{x_i \in D_R(B)} \sum_{y_j \in B} \bar R_{ij} \leq \sum_{x_i \in D_R(B)} \sum_{y_j \in \FF} \bar R_{ij} = \mu(D_R(B)),
 \end{equation*}
 and the equality holds only if for all $(x_i,y_j) \in D_R(B) \times B^c$, $\bar R_{ij} = 0$ and hence $R_{ij} = 0$. As by definition of $D_R(B)$, for all $(x_i,y_j) \in D_R(B)^c \times B$, we have $R_{ij} = 0$, an equality in the formula above would imply that $D_R(B) \times B$ is not connected to $D_R(B)^c \times B^c$ in $(\DD \times \FF, \triangle)$, which contradicts our connectivity assumption. We have thus a strict inequality. Finally, the full support of $\nu$ is also an consequence of the connectivity of $(\DD \times \FF, \triangle)$: For all $y_j \in \FF$, let $x_i \in \DD$ such that $R_{ij}>0$, and hence $\bar R_{ij}>0$. We have $\nu_j \geq \bar R_{ij}>0$. Thus, $\mathfrak B \subset \mathfrak A$.

 	\bigskip
 	
 \noindent \underline{Step 2}: $\mathfrak B$ is open in $\mathfrak A$.
 
 It is clear by its definition that $\mathfrak A$ is open in $\mathcal P(\FF)$. Therefore, to prove that $\mathfrak B$ is open in $\mathfrak A$, it suffices to prove that $\mathfrak B$ is open in $\mathcal P(\FF)$. Let $\nu \in \mathfrak B$, and $\bar R$ be such that $\bar R \sim R$ and $\nu^{\bar R} = \nu$. We choose:
\begin{align}
\label{def_epsilon}
    0< \eps <\min \{\bar R_{ij} \, | \, i,j \mbox{ s.t.\ } \bar R_{ij}>0\},
\end{align}
which is possible because $\DD$ and $\FF$ are finite. By convexity, it is enough to prove that for all $j \neq j'$ in $\{1, \dots, M\}$, $\nu + \eps (\delta_j - \delta_{j'}) \in \mathfrak B$. As $(\DD\cup\FF, \triangle)$ is connected, we can find $j=j_0, i_1, j_1, \dots, i_p, j_p = j'$ such that
 \begin{equation*}
 y_j = y_{j_0} \triangle x_{i_1} \triangle y_{j_1} \triangle \dots \triangle x_{i_p} \triangle y_{j_p} = y_{j'}.
 \end{equation*}
 Then we set
 \begin{equation*}
 P := \bar R + \eps \sum_{n=1}^p \Big( \delta_{i_n j_{n-1}} - \delta_{i_n j_n}\Big).
 \end{equation*}
 It is easy to check that $P \sim R$ (by the definition \eqref{def_epsilon} of $\eps$), that $\mu^P = \mu$, and that $\nu^P = \nu + \eps(\delta_j - \delta_{j'})$, which therefore belongs to $\mathfrak B$.
 
 \bigskip 
 
 \noindent\underline{Step 3}: $\mathfrak B$ is closed in $\mathfrak A$, strategy of the proof.
 
 Let us introduce the following subset of $\mathcal P(\FF)$:
 \begin{equation*}
     \mathfrak C := \{ \nu := \nu^{\bar R} \, | \, \bar R \ll R \mbox{ and } \mu^{\bar R} = \mu \}.
 \end{equation*}
 This set is clearly closed in $\mathcal P(\FF)$, so to prove that $\mathfrak B$ is closed in $\mathfrak A$, it suffices to prove that $\mathfrak B = \mathfrak A \cap \mathfrak C$. As we have already seen that $\mathfrak B \subset \mathfrak A$, and as clearly $\mathfrak B \subset \mathfrak C$, the only inclusion that needs to be justified is $\mathfrak A \cap \mathfrak C \subset \mathfrak B$. 
 
 Let us choose $\nu \in \mathfrak A \cap \mathfrak C$, and let us consider $R^*$ the solution of $\Sch(R;\mu,\nu)$. We will prove by contradiction that $R^* \sim R$, and hence that $\nu \in \mathfrak B$.

 So let us assume that $R^* \nsim R$. Once again, we choose $0< \eps <\min \{R^*_{ij} \, | \, i,j \mbox{ s.t.\ } R^*_{ij}>0\}$. For all $i,j$, we write $$x_i \blacktriangle y_j$$ whenever $R^*_{ij}>0$. We will first prove that $(\DD \cup \FF, \blacktriangle)$ is connected (this is the hardest part of the proof), and then that it coincides with $(\DD \cup \FF, \triangle)$, which exactly means that $R^* \sim R$.
 
 \bigskip
 
 \noindent\underline{Step 4}: $(\DD \cup \FF, \blacktriangle)$ is connected.
 
 We call $\DD_1 \cup \FF_1, \dots, \DD_p \cup \FF_p$ the connected components of $(\DD \cup \FF, \blacktriangle)$. Let us assume that $p>1$, and show that it leads to a contradiction. 
 
 First, we claim that if $p>1$, there exist $k_1, \dots, k_l\in\{1, \dots, p\}$ a family of two by two distinct indices, $x_{i^{k_1}}, \dots, x_{i^{k_l}} \in \DD$ and $y_{j^{k_1}}, \dots, y_{j^{k_l}} \in \FF$ such that for all $q = 1, \dots, l$, $x_{i^{k_q}} \in \DD_{k_q}$ and $y_{j^{k_q}} \in \FF_{k_q}$, and with the convention $l+1 = 1$, $y_{j^{k_q}} \triangle x_{i^{k_{q+1}}}$. 
 
 For proving this claim, we start by building a directed graph structure on $\{1, \dots, p\}$, the set of indices of the connected components of $(\DD \cup \FF, \blacktriangle)$. For all $k,k' \in \{1, \dots, p\}$, we write $k \leadsto k'$ whenever $k \neq k'$ and there exists $y_j \in \FF_k$ and $x_i \in \DD_{k'}$ such that $y_j \triangle x_i$. Of course, our claim precisely means that the directed graph $(\{1, \dots, p\}, \leadsto)$ admits a cycle. Let us prove that for all $k = 1, \dots, p$, there exists $k' \in \{1, \dots, p\}$ such that $k \leadsto k'$, which is clearly enough to conclude. 
 
 Let us consider $k \in \{1, \dots, p\}$. Because $\nu \in \mathfrak A$, we have $\mu(D_R(\FF_k)) > \nu( \FF_k )$. On the other hand, $\nu(\FF_k) = \mu(\DD_k)$ (as under $R^*$, all the mass on $\DD_k$ is sent to $\FF_k$ and \emph{vice versa}), and $\DD_k \subset D_R(\FF_k)$ (as $R^* \ll R$). Therefore, $\mu(D_R(\FF_k) \backslash \DD_k) = \mu(D_R(\FF_k)) - \mu(\DD_k) >0$, and we conclude that $D_R(\FF_k) \backslash \DD_k\neq \emptyset$. Let $x_i \in D_R(\FF_k) \backslash \DD_k$, and $k'$ such that $x_i \in \DD_{k'}$. As $x_i \in D_R(\FF_k)$, there is $y_j \in \FF_k$ such that $y_j \triangle x_i$. It follows that $k \leadsto k'$, and the claim is proved, and we can consider $k_1, \dots, k_l$ satisfying the properties above.\\
 
We now show that we reach a contradiction, which allows to conclude that actually, $p$ must be equal to $1$ and hence that $(\DD \cup \FF, \blacktriangle)$ needs to be connected.
 
For all $q=1, \dots, l$, as $(\DD_{k_q} \cup \FF_{k_q}, \blacktriangle)$ is connected, we can find a family of indices $i^{k_q} = i^{k_q}_1, j^{k_q}_1, \dots, i^{k_q}_{n_q}, j^{k_q}_{n_q} = j^{k_q}$ such that
 	\begin{equation*}
 x_{i^{k_q}} =  x_{i^{k_q}_1} \blacktriangle y_{j^{k_q}_1} \blacktriangle \dots \blacktriangle x_{i^{k_q}_{n_q}} \blacktriangle y_{j^{k_q}_{n_q}} = y_{j^{k_q}}
 \end{equation*}
 Now, still keeping the convention $l+1 = 1$, we set
 \begin{equation*}
 P:= R^* + \eps \sum_{q=1}^l \delta_{i^{k_{q+1}}j^{k_q}} - \eps \sum_{q=1}^l\left(  \sum_{n=1}^{n_q-1} \Big( \delta_{i^{k_q}_n j^{k_q}_n} - \delta_{i^{k_q}_{n+1} j^{k_q}_n} \Big) + \delta_{i^{k_q}_{n_q} j^{k_q}_{n_q}} \right).
 \end{equation*}
The matrix $P$ has less zeros than $R^*$: by the definition \eqref{def_epsilon} of $\eps$, it has no additional zero and we have for instance $P_{i^{k_2}j^{k_1}} > 0$ and $R^*_{i^{k_2}j^{k_1}} = 0$. Moreover, $P \in \Pi(\mu, \nu)$, and the construction of the indices ensures that $P$ has new non-zero entries w.r.t.\ $R^*$ only on $(x_i,y_j)$ such that $R_{ij} > 0$, which ensures that $P \ll R$. In virtue of Proposition~\ref{prop:R*_has_biggest_support}, this contradicts the fact that $R^*$ is the solution of $\Sch(R; \mu, \nu)$, and we conclude that $p=1$.
  
  \bigskip 
  
  \noindent \underline{Step 5}: $(\DD \cup \FF, \blacktriangle) = (\DD \cup \FF, \triangle)$.
  
   Our last task is to prove that whenever $x_i \triangle y_j$, for some $x_i \in \DD$ and $y_j \in \FF$, then we also have $x_i \blacktriangle y_j$ (the reciprocal statement follows from $R^* \ll R$). So let us consider $x_i \in \DD$ and $y_j \in \FF$ with $x_i \triangle y_j$, assume that we do not have $x_i \blacktriangle y_j$, and show that we reach a contradiction. As we know that $(\DD \cup \FF, \blacktriangle)$ is connected., we can find
  $i=i_1, j_1,i_2, \dots, i_p, j_p = j'$ such that
  \begin{equation*}
  x_i =  x_{i_1} \blacktriangle y_{j_1} \blacktriangle x_{i_2} \dots \blacktriangle x_{i_p} \blacktriangle y_{j_p} = y_{j}.
  \end{equation*}
  We set
  \begin{equation*}
  P := R^* + \eps \delta_{ij} - \eps \left(\sum_{n=1}^{p-1} \Big( \delta_{i_n j_{n}} - \delta_{i_{n+1} j_n}\Big) + \delta_{i_p j_p}\right).
  \end{equation*}
  Once again, $P$ has less zeros than $R^*$, which is a contradiction, and the result follows.
\end{proof}
 
 We are now ready to conclude the proof of Theorem~\ref{thm:CNS}.
 
 \begin{proof}[Proof of Theorem~\ref{thm:CNS}]
It remains to show that if $\mu$ and $\nu$ are such that~\ref{item:cond_A} is verified (and not~\ref{item:cond_A'}), then the problem is approximately scalable (once again, \ref{item:cond_B}$\Rightarrow$\ref{item:existence_solution} is proved in the same way, and \ref{item:existence_solution}$\Rightarrow$\ref{item:cond_A} and \ref{item:existence_solution}$\Rightarrow$\ref{item:cond_B} are easy).\\

Let us first remark that the problem $\Sch(R;\mu^R,\nu^R)$ is obviously scalable. Let us consider $\eps\in (0,1)$. We define:
 \begin{equation*}
     \mu^\eps := (1-\eps)\mu + \eps \mu^R \quad \mbox{and} \quad \nu^\eps=(1-\eps)\nu + \eps \nu^R.
 \end{equation*}
The condition~\ref{item:cond_A} implies that for all $A \subset \DD$:
\begin{equation*}
\mu^\eps (A) \leq (1-\eps)\nu(F_R(A)) + \eps \nu^R(F_R(A))= \nu^\eps(A).
\end{equation*}
with a strict inequality whenever $\mu^R(A) < \nu^R(F_R(A))$.\\

Let us now assume that Assumption~\ref{ass:full_support} holds. In this case, in virtue of Lemma~\ref{lem:non_degenerate_case}, the problem $\Sch(R;\mu^\eps,\nu^\eps)$ is scalable. In particular, there exists $R^\eps \in \Pi(\mu^\eps,\nu^\eps)$ such that $R^\eps$ and $R$ have the same support. As the family $(R^\eps)$ has value in the compact set 
\begin{equation*}
\{R' \in \mathcal M_{+}(\DD \times \FF) \, | \, \mathsf M (R') \leq \max(\mathsf M(\mu), \mathsf M(R))\},
\end{equation*}
we can chose one of its limit points $\bar R$ as $\eps \to 0$. Obviously, $\bar R \ll R$ and $\bar R \in \Pi(\mu,\nu)$ so that $\Sch(R;\mu,\nu)$ is approximately scalable.\\

It remains to prove that~\ref{item:cond_A}$\Rightarrow$\ref{item:existence_solution} even when Assumption~\ref{ass:full_support} does not hold. To do so, we claim that under~\ref{item:cond_A}, assuming Assumption~\ref{ass:full_support} is not restrictive. The reason is that~\ref{item:cond_A} implies Assumption~\ref{ass:condition_sinkhorn_well_defined}, and hence Assumption~\ref{ass:full_support} up to restricting the problem to the supports of $\mu$ and~$\nu$, as explained in Subsection~\ref{subsec:ass}. 

So let us prove that~\ref{item:cond_A} implies Assumption~\ref{ass:condition_sinkhorn_well_defined}. We suppose that \ref{item:cond_A} holds, and we consider $\mathcal E$ and $R^0$ as defined in Assumption~\ref{ass:condition_sinkhorn_well_defined}. 

Let us show that $\mu \ll \mu^{R^0}$. Let $x_i \in \DD$ be such that $\mu_i>0$, and let us show that $\mu^{R^0}_i>0$. By~\ref{item:cond_A}, $\nu(F_R(\{x_i\})\geq \mu_i>0$. Therefore, $F_R(\{x_i\})$ is nonempty, and there exists $y_j \in F_R(\{x_i\})$ such that $\nu_j>0$. This pair $(x_i,y_j)$ belongs to $\mathcal E$, so $R^0_{ij}>0$, and then $\mu^{R^0}_i>0$.

Let us show that $\nu \ll \nu^{R^0}$. Let $y_j \in \FF$ be such that $\nu_j>0$, and let us show that $\nu^{R^0}_j>0$. Let us call $\DD'$ the support of $\mu$. By~\ref{item:cond_A}, $\mathsf M(\nu) \geq \nu(F(\DD')) \geq\mu(\DD') = \M(\mu)$. But as $\M(\nu) = \M(\mu)$, we conclude that $\nu(F(\DD')) = \M(\nu)$, and so in particular that $y_j \in F(\DD')$. So there exists $x_i \in \DD'$ such that $R_{ij}>0$. This pair $(x_i,y_j)$ belongs to $\mathcal E$, so $R^0_{ij}>0$, and then $\nu^{R^0}_j>0$.
\end{proof}
\end{appendices}

\bibliography{biblio}

\begin{thebibliography}{10}

\bibitem{Achilles1993}
E.~Achilles.
\newblock Implications of convergence rates in sinkhorn balancing.
\newblock {\em Linear Algebra and its Applications}, 187:109--112, 1993.

\bibitem{benamou2015iterative}
J-D. Benamou, G.~Carlier, M.~Cuturi, L.~Nenna, and G.~Peyr{\'e}.
\newblock {Iterative Bergman projections for regularized transportation
  problems}.
\newblock {\em SIAM Journal on Scientific Computing}, 37(2):A1111--A1138, 2015.

\bibitem{braides2002gamma}
A.~Braides.
\newblock {\em Gamma-convergence for Beginners}, volume~22.
\newblock Clarendon Press, 2002.

\bibitem{Brualdi1968}
R.~Brualdi.
\newblock Convex sets of non-negative matrices.
\newblock {\em Canadian Journal of Mathematics}, 20:144--157, 1968.

\bibitem{carlier2017convergence}
G.~Carlier, V.~Duval, G.~Peyr{\'e}, and B.~Schmitzer.
\newblock Convergence of entropic schemes for optimal transport and gradient
  flows.
\newblock {\em SIAM Journal on Mathematical Analysis}, 49(2):1385--1418, 2017.

\bibitem{Chizat2018}
L.~Chizat, G.~Peyr{\'e}, B.~Schmitzer, and FX. Vialard.
\newblock Scaling algorithms for unbalanced optimal transport problems.
\newblock {\em Mathematics of Computation}, 87(314):2563--2609, 2018.

\bibitem{Csiszar1975}
I.~Csisz{\'a}r.
\newblock I-divergence geometry of probability distributions and minimization
  problems.
\newblock {\em The annals of probability}, pages 146--158, 1975.

\bibitem{Cuturi2013}
M.~Cuturi.
\newblock Sinkhorn distances: Lightspeed computation of optimal transport.
\newblock In {\em Advances in neural information processing systems}, pages
  2292--2300, 2013.

\bibitem{follmer1988random}
H.~F{\"o}llmer.
\newblock Random fields and diffusion processes.
\newblock In {\em {\'E}cole d'{\'E}t{\'e} de Probabilit{\'e}s de Saint-Flour
  XV--XVII, 1985--87}, pages 101--203. Springer, 1988.

\bibitem{fortet1940resolution}
R.~Fortet.
\newblock R\'{e}solution d'un syst\`eme d'\'{e}quations de {M}.
  {S}chr\"{o}dinger.
\newblock {\em J. Math. Pures Appl.}, 19:83--105, 1940.

\bibitem{Gentil2017}
I.~Gentil, C.~L{\'e}onard, and L.~Ripani.
\newblock About the analogy between optimal transport and minimal entropy.
\newblock {\em Annales de la Facult{\'e} des Sciences de Toulouse.
  Math{\'e}matiques. S{\'e}rie 6}, 3:569--600, 2017.

\bibitem{hagberg2008exploring}
A.~Hagberg, P.~Swart, and D.~Chult.
\newblock Exploring network structure, dynamics, and function using networkx.
\newblock Technical report, Los Alamos National Lab.(LANL), Los Alamos, NM
  (United States), 2008.

\bibitem{Herbach2017}
U.~Herbach, A.~Bonnaffoux, T.~Espinasse, and O.~Gandrillon.
\newblock Inferring gene regulatory networks from single-cell data: a
  mechanistic approach.
\newblock {\em BMC Systems Biology}, 11:105, 2017.

\bibitem{idel2016review}
M.~Idel.
\newblock A review of matrix scaling and sinkhorn's normal form for matrices
  and positive maps, 2016.

\bibitem{kondratyev2016new}
S.~Kondratyev, L.~Monsaingeon, and D.~Vorotnikov.
\newblock A new optimal transport distance on the space of finite radon
  measures.
\newblock {\em Advances in Differential Equations}, 21(11/12):1117--1164, 2016.

\bibitem{Lavenant2021}
H.~Lavenant, S.~Zhang, Y-H. Kim, and G.~Schiebinger.
\newblock Towards a mathematical theory of trajectory inference.
\newblock {\em arXiv preprint arXiv:2102.09204}, 2021.

\bibitem{Leonard2012}
C.~L{\'e}onard.
\newblock {From the Schr{\"o}dinger problem to the Monge--Kantorovich problem}.
\newblock {\em Journal of Functional Analysis}, 262(4):1879--1920, 2012.

\bibitem{leonard2013survey}
C.~L{\'e}onard.
\newblock {A survey of the Schr{\"o}dinger problem and some of its connections
  with optimal transport}.
\newblock {\em Discrete \& Continuous Dynamical Systems-A}, 34(4):1533--1574,
  2014.

\bibitem{Leonard2014}
C.~L{\'e}onard.
\newblock {Some properties of path measures}.
\newblock In {\em S{\'e}minaire de Probabilit{\'e}s XLVI}, pages 207--230.
  Springer, 2014.

\bibitem{Liero2018}
M.~Liero, A.~Mielke, and G.~Savar{\'e}.
\newblock Optimal entropy-transport problems and a new hellinger--kantorovich
  distance between positive measures.
\newblock {\em Inventiones mathematicae}, 211(3):969--1117, 2018.

\bibitem{mikami2004monge}
T.~Mikami.
\newblock Monge's problem with a quadratic cost by the zero-noise limit of
  h-path processes.
\newblock {\em Probability theory and related fields}, 129(2):245--260, 2004.

\bibitem{Nutz2021}
M.~Nutz.
\newblock Introduction to entropic optimal transport.

\bibitem{peyre2019computational}
G.~Peyr{\'e} and M.~Cuturi.
\newblock Computational optimal transport: With applications to data science.
\newblock {\em Foundations and Trends{\textregistered} in Machine Learning},
  11(5-6):355--607, 2019.

\bibitem{ruschendorf1993note}
L.~R\"{u}schendorf and W.~Thomsen.
\newblock Note on the {S}chr\"{o}dinger equation and {$I$}-projections.
\newblock {\em Statist. Probab. Lett.}, 17(5):369--375, 1993.

\bibitem{sanov1958probability}
I.N. Sanov.
\newblock {On the probability of large deviations of random variables}.
\newblock Technical report, North Carolina State University. Dept. of
  Statistics, 1958.

\bibitem{Schiebinger2019}
G.~Schiebinger, J.~Shu, M.~Tabaka, B.~Cleary, V.~Subramanian, A.~Solomon,
  J.~Gould, S.~Liu, S.~Lin, P.~Berube, L.~Lee, J.~Chen, J.~Brumbaugh,
  P.~Rigollet, K.~Hochedlinger, R.~Jaenisch, A.~Regev, and E.~S. Lander.
\newblock Optimal-transport analysis of single-cell gene expression identifies
  developmental trajectories in reprogramming.
\newblock {\em Cell}, 176:928--943 e22, 2019.

\bibitem{Schrodinger1931}
E.~Schr{\"o}dinger.
\newblock {\em {{\"U}ber die Umkehrung der naturgesetze}}.
\newblock Verlag Akademie der wissenschaften in kommission bei Walter de
  Gruyter u. Company, 1931.

\bibitem{Schrodinger1932}
E.~Schr{\"o}dinger.
\newblock {Sur la th{\'e}orie relativiste de l'{\'e}lectron et
  l'interpr{\'e}tation de la m{\'e}canique quantique}.
\newblock In {\em {Annales de l'institut Henri Poincar{\'e}}}, volume~2, pages
  269--310, 1932.

\bibitem{Sinkhorn1964}
R.~Sinkhorn.
\newblock A relationship between arbitrary positive matrices and doubly
  stochastic matrices.
\newblock {\em The annals of mathematical statistics}, 35(2):876--879, 1964.

\bibitem{sinkhorn1967diagonal}
R.~Sinkhorn.
\newblock Diagonal equivalence to matrices with prescribed row and column sums.
\newblock {\em The American Mathematical Monthly}, 74(4):402--405, 1967.

\bibitem{Ventre2022thesis}
E.~Ventre.
\newblock Analyse, calibration et évaluation de modèles stochastiques
  d'expression des gènes.
\newblock {\em Hal}, 2022.

\bibitem{Ventre2021}
E.~Ventre.
\newblock Reverse engineering of a mechanistic model of gene expression using
  metastability and temporal dynamics.
\newblock {\em In Silico Biology}, (Preprint):1--25, 2022.

\bibitem{ventre2022one}
E.~Ventre, U.~Herbach, T.~Espinasse, G.~Benoit, and O.~Gandrillon.
\newblock One model fits all: combining inference and simulation of gene
  regulatory networks.
\newblock {\em bioRxiv}, 2022.

\bibitem{zambrini1986variational}
J-C. Zambrini.
\newblock Variational processes and stochastic versions of mechanics.
\newblock {\em Journal of Mathematical Physics}, 27(9):2307--2330, 1986.

\end{thebibliography}
\bibliographystyle{plain}
\end{document}